\documentclass[11pt]{amsart}
\usepackage{amscd,amsmath,amssymb,amsfonts,verbatim}
\usepackage[cmtip, all]{xy}
\usepackage{MnSymbol}
\usepackage{comment}

\newtheorem{thm}{Theorem}[section]
\newtheorem{prop}[thm]{Proposition}
\newtheorem{lem}[thm]{Lemma}
\newtheorem{cor}[thm]{Corollary}

\theoremstyle{definition}

\newtheorem{defn}[thm]{Definition}
\theoremstyle{remark}
\newtheorem{remk}[thm]{Remark}
\newtheorem{remks}[thm]{Remarks}

\newtheorem{exm}[thm]{Example}
\newtheorem{exms}[thm]{Examples}
\newtheorem{notat}[thm]{Notation}
\numberwithin{equation}{section}

{\hfill$\square$\end{defn}}
{\hfill$\square$\end{remk}}
{\hfill$\square$\end{remks}}
{\hfill$\square$\end{exm}}
{\hfill$\square$\end{exms}}
{\hfill$\square$\end{notat}}

\newcommand{\thmref}{Theorem~\ref}
\newcommand{\propref}{Proposition~\ref}
\newcommand{\corref}{Corollary~\ref}

\newcommand{\lemref}{Lemma~\ref}

\newcommand{\sC}{{\mathcal C}}
\newcommand{\sD}{{\mathcal D}}

\newcommand{\sF}{{\mathcal F}}

\newcommand{\sO}{{\mathcal O}}
\newcommand{\sP}{{\mathcal P}}

\newcommand{\sR}{{\mathcal R}}

\newcommand{\sV}{{\mathcal V}}
\newcommand{\sW}{{\mathcal W}}

\newcommand{\A}{{\mathbb A}}

\newcommand{\F}{{\mathbb F}}
\newcommand{\G}{{\mathbb G}}
\renewcommand{\H}{{\mathbb H}}

\newcommand{\N}{{\mathbb N}}
\renewcommand{\P}{{\mathbb P}}
\newcommand{\Q}{{\mathbb Q}}

\newcommand{\W}{{\mathbb W}}

\newcommand{\Z}{{\mathbb Z}}

\newcommand{\fm}{{\mathfrak m}}

\newcommand{\fp}{{\mathfrak p}}

\newcommand{\CH}{{\rm CH}}

\newcommand{\surj}{\twoheadrightarrow}
\newcommand{\inj}{\hookrightarrow}
\newcommand{\red}{{\rm red}}

\newcommand{\Spec}{{\rm Spec \,}}

\newcommand{\Tr}{{\rm Tr}}

\newcommand{\Sch}{{\operatorname{\mathbf{Sch}}}}

\newcommand{\op}{{\text{\rm op}}}

\newcommand{\Sm}{{\mathbf{Sm}}}

\newcommand{\SmAff}{{\mathbf{SmAff}}}

\newcommand{\ds}{{/\kern-3pt/}}

\renewcommand{\log}{{\operatorname{log}}}

\newcommand{\ess}{\text{\rm{ess}}}

\newcommand{\Proj}{{\operatorname{Proj}}}

\newcommand{\colim}{\mathop{\text{\rm colim}}}

\newcommand{\TZ}{{\operatorname{Tz}}}
\renewcommand{\TH}{{\operatorname{TCH}}}
\newcommand{\un}{\underline}
\newcommand{\ov}{\overline}

\newcommand{\dgn}{{\operatorname{degn}}}
\renewcommand{\dim}{\text{\rm dim}}

\newcommand{\tuborg}{\left\{\begin{array}{ll}}
\newcommand{\sluttuborg}{\end{array}\right.}

\newcommand{\tch}{{\mathcal{TCH}}}
\newcommand{\sfs}{{\rm sfs}}
\newcommand{\fs}{{\rm fs}}

\newcommand{\zar}{{\rm zar}}

\newcommand{\wt}{\widetilde}
\newcommand{\wh}{\widehat}

\newcounter{elno}

\newcounter{elno-abc}   

\newcounter{elno-abc-prime}

\begin{document}
\title{de Rham-Witt sheaves via algebraic cycles}
\author{Amalendu Krishna and Jinhyun Park\\(With an appendix by Kay R\"ulling)}
\address{School of Mathematics, Tata Institute of Fundamental Research,  
1 Homi Bhabha Road, Colaba, Mumbai, India}
\email{amal@math.tifr.res.in}
\address{Department of Mathematical Sciences, KAIST, 291 Daehak-ro Yuseong-gu, 
Daejeon, 34141, Republic of Korea (South)}
\email{jinhyun@mathsci.kaist.ac.kr; jinhyun@kaist.edu}
\address{Bergische Universit\"at Wuppertal, Fakult\"at Mathematik und Naturwissenshaften, Gauss-strasse 20, D-42119 Wuppertal, Germany}
\email{kay.ruelling@uni-wuppertal.de}



\begin{abstract}
We show that the additive higher Chow groups of regular schemes over a field induce a Zariski sheaf of pro-differential graded algebras, whose Milnor range is isomorphic to the Zariski sheaf of big de Rham-Witt complexes. This provides an explicit cycle-theoretic description of the big de Rham-Witt sheaves. Several applications are derived.
\end{abstract}
\setcounter{tocdepth}{1}
\maketitle
\tableofcontents

\section{Introduction}\label{sec:Intro}
\subsection{Motivation}\label{sec:Motivation}
For a presheaf $\sF$ on the category of smooth schemes essentially of finite type over a field with values in abelian groups, it is often
desirable to describe the values of $\sF$ in terms of algebraic cycles
because the cycles are usually relatively explicit by nature.

For instance, by the results of Bloch \cite{Bl1}, Bloch-Lichtenbaum \cite{BL}, and Friedlander-Suslin \cite{FS} and Levine \cite{Levine}, one knows that the higher algebraic $K$-groups of smooth schemes over a field can be described in terms of the higher Chow groups of Bloch through a spectral sequence. A property of $K$-theory of smooth schemes that helps achieve this is its $\A^1$-invariance.  However, there are several presheaves on smooth schemes which do not
satisfy the $\A^1$-invariance property and for which one often has to
seek a description in terms of algebraic cycles.

Examples of non-$\A^1$-invariant presheaves we have in mind are the differential forms and the $K$-theory of non-reduced infinitesimal extensions of smooth schemes. These are some of the most useful presheaves on the category of smooth schemes over a field, as they carry significant information about the underlying schemes. Perhaps in order to study this question for these two presheaves, Bloch and Esnault \cite{BE1}, \cite{BE2} invented the additive higher Chow groups of fields. This is related to earlier works of Bloch (e.g. \cite{Bloch tangent}) on the study of
the tangent space to the space of cycles and $K$-theory. The additive higher Chow groups were subsequently defined and studied for more general schemes in the works of Krishna-Levine \cite{KL}, Park \cite{P1} and R{\"u}lling \cite{R}.

The results of \cite{BE1} and \cite{BE2} suggest that additive higher Chow groups could be useful in describing the above two non-$\A^1$-invariant presheaves
in terms of algebraic cycles. Soon after, R{\"u}lling \cite{R} proved a more general result that when the underlying scheme is the spectrum of a field, the big de Rham-Witt complex of Hesselholt-Madsen \cite{HeMa} is isomorphic to a complex consisting of additive higher Chow groups of the field.

\subsection{The main result}\label{sec:MR}
The purpose of this paper is to show that the sheaf of big de Rham-Witt complexes on the big Zariski site of all regular schemes essentially of finite type over any field coincides with the sheaf of complexes of the (Milnor range of) additive higher Chow groups. 
Among several consequences, this provides a new description 
of the crystalline cohomology of smooth quasi-projective schemes over a 
perfect field  purely in terms of algebraic cycles. This is the 
cycle-theoretic avatar of Bloch's description of the crystalline cohomology via
algebraic $K$-theory \cite{Bloch crys}.
The main result from which we derive the above is the following:

\begin{thm}\label{thm:Main-1}
Let $k$ be an arbitrary field and let $R$ be a regular semi-local $k$-algebra essentially of finite type. Then for all $m, n \geq 1$, there are isomorphisms between the big de Rham-Witt forms and the additive higher Chow groups
\begin{equation}\label{eqn:intro-main-eqn}
\tau^R _{n,m} \colon 
\mathbb{W}_m \Omega_R ^{n-1} \overset{\simeq}{\to} \TH^n (R, n;m),
\end{equation}
which are natural in $R$. In particular, the additive higher Chow groups of $\Spec (R)$ in the Milnor range form the universal restricted Witt-complex over $R$.
\end{thm}

\thmref{thm:Main-1} is a direct extension of R{\"u}lling's theorem \cite{R} proven originally for fields to the case of regular semi-local algebras over fields. This can also be considered as a non-$\A^1$-invariant analogue of the results of Elbaz-Vincent--M\"uller-Stach \cite{EVMS} and Kerz \cite{Kerz}, \cite{Kerz10}, which together imply that the Milnor $K$-theory of such rings are isomorphic to their higher Chow groups. The corresponding result for fields was shown by Totaro \cite{Totaro}. 

\subsection{Applications}\label{sec:Apln}
The first interesting consequence of Theorem \ref{thm:Main-1} is that it provides new geometric perspectives to the Witt vectors and the de Rham-Witt forms. Take $n=1$, for instance. Then the map $\tau_{1, m} ^R$ is a ring isomorphism. The underlying abelian group of the ring $\mathbb{W}_m (R)$ is not so difficult to understand. However, its product structure is notoriously complicated. The ring isomorphism $\tau_{1, m} ^R$ then shows that the addition of $\mathbb{W}_m (R)$ corresponds simply to the summation of algebraic cycles in $\TH^1 (R, 1;m)$, while the complicated multiplication of $\mathbb{W}_m (R)$ corresponds to a simple Pontryagin-intersection product of algebraic cycles, thus giving a conceptual and simple description of the ring structure of Witt vectors. 

For general $n \geq 1$, the differential operator, the wedge product, and Frobenius and Verschiebung operators on $\W_m\Omega^\bullet_R$ are nontrivial objects to describe. However, their counterparts for additive higher Chow groups are simply the Pontryagin-intersection, flat pull-back and finite push-forward of algebraic cycles.

Apart from the above, we derive several further applications of \thmref{thm:Main-1}. We show that the additive higher Chow groups have $p$-typical decomposition in characteristic $p > 0$. Using this and \thmref{thm:Main-1}, we show that the crystalline cohomology of smooth quasi-projective schemes over a perfect field is motivic in a broad sense. We also show using \thmref{thm:Main-1} and the results \cite{HessK} that the algebraic $K$-theory of the truncated polynomial algebras over regular semi-local rings over a field admit descriptions through cycles.

We show that in characteristic $p > 0$, the mod-$p$ motivic cohomology and mod-$p$ Milnor as well as Quillen $K$-theory Zariski sheaves on regular schemes over a field coincide with the sheaf of $p$-typical additive higher Chow groups. We also deduce a Gersten resolution for additive higher Chow groups in the sense that the Cousin complex of the Zariski sheaf of additive higher Chow groups on a regular scheme essentially of finite type over a field is exact. These results show that the $p$-typical additive higher Chow groups provide  cycle-theoretic descriptions of existing mod-$p$ cohomology theories of regular schemes.

In another application, we use \thmref{thm:Main-1} and the existence of push-forward maps for additive higher Chow groups to give new and direct construction of the theory of trace maps on big de Rham-Witt complexes of regular algebras essentially of finite type over a field. Such a theory is abstractly known to exist (in the $p$-typical case) by the duality theory for de Rham-Witt complex by Ekedahl \cite{Ekedahl}. However, Ekedahl's theory of trace is obtained using the complicated machinery of Grothendieck duality in the derived category of quasi-coherent sheaves over the sheaves of
$p$-typical Witt vectors. This makes it very hard to work with his trace. 
The construction of trace was also previously known for the de Rham-Witt complex of fields (see \cite{R}). It follows \emph{a posteriori} that Ekedahl's trace
and the push-forward of cycles are compatible, and the traces now have more explicit descriptions. See Theorems \ref{thm:trace} and \ref{thm:trace-gen}.

In our final application, we use the Galois descent property of the de Rham-Witt complex to prove
a similar descent for the additive higher chow groups of regular semi-local algebras over a field.
The precise versions of these applications are given in
\S~\ref{sec:Appln}. Besides these, \thmref{thm:Main-1} also plays crucial roles in the proofs of the main
results of \cite{GK-1} and \cite{GK-2}.

\subsection{Outline of proofs}\label{sec:Outline}
The proof of \thmref{thm:Main-1} relies on several results of previous papers of the authors. Among these, \cite{KP2}, \cite{KP3} and \cite{KP sfs} play major roles. The first step in the proof of \thmref{thm:Main-1} is to construct the `de Rham-Witt-Chow homomorphism' $\tau^R_{n,m}$ from the big de Rham-Witt complex $\W_m\Omega^{n-1}_R$ of a regular semi-local ring $R$ to its additive higher Chow groups $\TH^n(R,n;m)$. This itself is non-trivial and requires us to show that the additive higher Chow groups of $R$ form a restricted Witt-complex over $R$ in the sense of \cite[Definition~1.14]{R}. Modulo a new limit argument in the case of imperfect base fields discussed in the present article, the proof of this step was the main purpose of the papers \cite{KP2} and \cite{KP3}.

The next step is to show that de Rham-Witt-Chow homomorphism $\tau^R_{n,m}$ is injective. This is not hard. We use a reduction step and a Gersten type result for the de Rham-Witt complex to reduce to the case where $R$ is a field. The latter case is due to R{\"u}lling \cite{R} (and see \S \ref{sec:Appendix} for the characteristic $2$ case).

The surjectivity of $\tau^R_{n,m}$ is a difficult part of the proof of \thmref{thm:Main-1}. In order to achieve this, we follow several reduction steps. We first prove this in the case when $k$ is infinite and perfect. To deal with this case, we rely on the `$\sfs$-moving lemma' of \cite{KP sfs} (see Theorem \ref{thm:sfs-TCH}). When the base field is finite, we use a pro-$\ell$ extension trick to reduce it to the case of infinite perfect base fields. The remaining case when $k$ is imperfect is reduced to when $k$ is perfect using a limit argument. 

For the infinite perfect base field case, using the \sfs-moving lemma, we prove that for a given cycle $\alpha \in \TH^n(R,n;m)$, we can find a finite extension of regular semi-local $k$-algebras $f \colon R \to R'$, and a cycle $\alpha' \in \TH^n(R',n;m)$ such that: (1) $f_*(\alpha') = \alpha$ and (2) $\alpha'$ is `symbolic', i.e., lies in the image of $\tau^{R'}_{n,m}$. The problem we face now is that we do not know how to conclude from this that $\alpha$ is symbolic using the existing theory of trace maps for de Rham-Witt forms due to Ekedahl \cite{Ekedahl}
(see \S~\ref{sec:Trace*}). To circumvent this problem, we devise a technique of \emph{traceability} of de Rham-Witt forms via cycles.

The traceability of a form $\omega' \in \W_m\Omega^{n-1}_{R'}$ yields a trace element of $\omega'$ in $\W_m\Omega^{n-1}_{R}$ up to images in the additive higher Chow groups. This luckily suffices for our purpose. We show, by factoring the extension $R \to R'$ into a composition of simple extensions, that every form in $\W_m\Omega^{n-1}_{R'}$ is traceable to $R$ (i.e., has a trace element in $\W_m\Omega^{n-1}_{R}$), see \propref{prop:Trace-simple}. The proof of this requires us to work with $\W_m\Omega^{n-1}_{R}$ for all $m \ge 1$ and $n \ge 1$ simultaneously rather than working with these groups individually. This uses the full strength of the Witt-complex structures on de Rham-Witt forms and additive higher Chow groups, and also the fact that $\tau^R_{n,m}$ is a morphism of Witt-complexes. The traceability allows us to find a pre-image of $\alpha$ in $\W_m\Omega^{n-1}_{R}$.

We recall the definitions of Milnor $K$-theory and de Rham-Witt complex, and then prove some of their properties in \S \ref{sec:MKDW}. We recall the definitions of additive higher Chow groups and some basic results on them in \S \ref{sec:ACH}. The main result of \S \ref{sec:WCADC} is to prove the restricted Witt-complex structure on the additive higher Chow groups of regular semi-local algebras over arbitrary fields. The de Rham-Witt-Chow homomorphism $\tau^R_{n,m}$ is constructed in \S \ref{sec:The-map}. We prove the key result on traceability in this section. We finish the proof of  \thmref{thm:Main-1} in \S \ref{sec:PRF} and its applications are derived in  \S \ref{sec:Appln}. 

The last section is an appendix by Kay R{\"u}lling, which proves the characteristic $2$ case of his result from \cite{R} for fields. We thank him for this contribution.

\section{Milnor $K$-theory and de Rham-Witt complex}\label{sec:MKDW}
We fix an arbitrary field $k$. We let $p \ge 1$ denote the exponential characteristic of $k$. Any further restriction on the nature of $k$ will be explicitly stated. In this section, we fix our conventions on notations. We then recall the definitions of Milnor $K$-groups, rings of big Witt vectors, and the big de Rham-Witt complexes of Hesselholt-Madsen \cite{HeMa}. We discuss some properties of these objects that will be used throughout this paper. 

\subsection{Conventions}\label{sec:Notn} 
In this paper, a $k$-scheme is a separated scheme of finite type over $k$, unless we say otherwise. A $k$-variety is a reduced $k$-scheme. The product $X \times Y$ means usually $X \times _k Y$, unless we specify otherwise. We let $\Sch_k$ be the category of $k$-schemes, $\Sm_k$ of smooth $k$-schemes, and $\SmAff_k$ of smooth affine $k$-schemes. 
By a scheme essentially of finite type over $k$, we mean a Noetherian $k$-scheme which is the projective limit of a cofiltered collection of open subschemes of a finite type $k$-scheme. We let $\Sch^{\ess}_k$ be the category of such schemes. An affine scheme is an object of $\Sch^{\ess}_k$ if and only if it is the spectrum of a ring obtained by localizing a finite type $k$-algebra at a multiplicatively closed subset. By a semi-local scheme, we shall mean an affine $k$-scheme essentially of finite type with only finitely many closed points.

A ring $R$ in this paper will always mean a commutative Noetherian $k$-algebra. We say that $R$ is regular if all its local rings are regular local rings. Equivalently, $R_\fm$ is a regular local ring for every maximal ideal $\fm \subset R$. We let $\N$ denote the set of all positive integers.

\subsection{Milnor $K$-groups}\label{subsection:Milnor K}
Recall that the Milnor $K$-ring $K_* ^M (R)$ of $R$ is the quotient $T_{\mathbb{Z}} ^* (R^{\times}) / I$ of the tensor algebra of $R^\times$ by the two-sided ideal $I$ generated by $\{ a \otimes (1-a) \ | \ a \in R^{\times}, 1-a \in R^{\times} \}$. Its degree $n$-part is $K_n ^M (R)$ and $\{ a_1, \cdots, a_n \}$ is the image of $a_1 \otimes \cdots \otimes a_n$ in $K^M_n(R)$, where $a_i \in R^{\times}$.

When $k$ is finite, it is known that the Milnor $K$-groups $K^M_*(A)$ are not most ideally defined. For instance, the Gersten conjecture fails even if $R$ is a regular semi-local ring. To remedy this, Gabber (unpublished) and Kerz \cite{Kerz10} defined an `improved Milnor $K$-theory' $\wh{K}^M_*(R)$. These groups are equipped with a natural map of graded rings $\psi_R \colon K^M_*(R) \to \wh{K}^M_*(R)$ satisfying the following properties.
\begin{enumerate}
\item $\psi_R$ is an isomorphism if either $k$ is infinite or $R$ is a (finite or infinite) field.
\item $\psi_R$ is surjective if $R$ is local.
\item
The Gersten conjecture holds for $\wh{K}^M_*(R)$ if $R$ is regular semi-local essentially of finite type over $k$.
\end{enumerate}

In this paper, the Milnor $K$-theory $K^M_*(R)$ will always mean the improved Milnor $K$-theory of Gabber and Kerz.  

\subsection{The ring of Witt-vectors}\label{subsection:Witt ring}
Let $R$ be a ring. We recall the definition of the ring of big Witt-vectors of $R$, e.g., from \cite[Appendix A]{R}. A \emph{truncation set} $S \subset \mathbb{N}$ is a nonempty subset such that if $s \in S$ and $t | s$, then $t \in S$. As a set, let $\mathbb{W}_S (R):= R^S$ and define the map $w: \mathbb{W}_S (R) \to R^S$, $a = (a_s)_{s \in S} \mapsto w(a) = (w(a)_s)_{s \in S}$, where $  w(a)_s := \sum_{ t | s} t a_t ^{ \frac{s}{t}}$ (called the `ghost map'). When $R^S$ on the target of $w$ is given the component-wise ring structure, it is known that there is a unique functorial ring structure on $\mathbb{W}_S (R)$ such that $w$ is a ring homomorphism. See \cite[Proposition 1.2]{Hesselholt2}. 
The ghost map is an isomorphism if $R$ contains $\Q$.

For two truncation sets $S \subset S'$, there is a restriction $\mathfrak{R} : \mathbb{W}_{S'} (R) \to \mathbb{W}_{S}(R)$. When $S= \{ 1, \ldots, m \}$, we write $\mathbb{W}_m (R): = \mathbb{W}_S (R)$. We let $\W(R):= \mathbb{W}_{\mathbb{N}} (R)$, which is $ \varprojlim_m \W_m(R)$. For a fixed prime $p$ and $S_i= \{ 1, p, \ldots, p^{i-1} \}$, we write $W_i (R) = \mathbb{W}_{S_i} (R)$ and $W(R) := \mathbb{W}_{1, p, p^2 , \cdots} (R)$, which is $ {\varprojlim}_i W_i(R)$. They are the $p$-typical rings of Witt vectors.

There is another description of the rings $\mathbb{W}_S (R)$. There is a natural bijection $\mathbb{W}(R) \simeq (1+ TR[[T]])^{\times}$, where $T$ is an indeterminate and the addition of the ring $\mathbb{W}(R)$ corresponds to the multiplication of the formal power series. For a truncation set $S$, we have $\mathbb{W}_S (R) \cong (1+ TR[[T]])^{\times} / I_S$ for a  suitable subgroup $I_S$. See \cite[A.7]{R} for details. In case $S= \{ 1, \ldots, m \}$, we have an isomorphism 
\begin{equation}\label{eqn:WF-1}
\gamma: \mathbb{W}_m (R) \simeq (1+ TR[[T]])^{\times} / (1 + T^{m+1} R[[T]])^{\times};\ \ (a_i)_{1 \le i \le m}\mapsto \prod_{i=1} ^m  (1- a_iT^i).
\end{equation}

The Teichm\"uller lift $[ - ]_S : R \to \mathbb{W}_S (R)$ is given by $a \mapsto 1-aT$, which is a multiplicative map. If $S= \{ 1, \ldots, m \}$, we write $[-]_m$ for $[-]_S$. When the truncation set is understood, we shall write $[-]_S$ also as $[-]$. For each $i \geq 1$, we have the $i$-th Verschiebung map $V_i$ given by $V_i([a]_{\lfloor{m/i}\rfloor})= (1- aT^i)$ under the identification \eqref{eqn:WF-1}, where for a non-negative real number $c$, one denotes by $\lfloor c \rfloor$ the greatest integer not bigger than $c$. By 
\cite[Properties A.4(i)]{R}, for $x= (x_i) \in \mathbb{W}_m (R)$, we have
\begin{equation}\label{eqn:Witt present}
x = \sum_{i=1}^m  V_i ( [x_i]_{ \lfloor{m/i}\rfloor}).
\end{equation}

\subsection{The de Rham-Witt complex}\label{sec:DRW}
Let $R$ be a $k$ algebra. Recall from \cite[Definition 1.1.1]{HeMa} that \emph{a restricted Witt-complex over $R$} is a pro-system of differential graded 
$\mathbb{Z}$-algebras $((E^\bullet_m)_{m \in \mathbb{N}}, \mathfrak{R}: 
E^\bullet_{m+1} \to E^\bullet_m)$ together with families of homomorphisms of graded rings $(F_r: E^\bullet_{rm+r-1} \to E^\bullet_m)_{m , r \in \mathbb{N}}$ called 
Frobenius maps, and homomorphisms of graded groups 
$(V_r: E^\bullet_m \to E^\bullet_{rm+r-1})_{m,r \in \mathbb{N}}$ called Verschiebung 
maps, satisfying the following relations for all $n, r, s \in \mathbb{N}$:
\begin{enumerate}
\item [(i)] $\mathfrak{R}F_r= F_r \mathfrak{R}^r, \mathfrak{R}^r V_r = V_r \mathfrak{R}, F_1 = V_1= {\rm Id}, F_r F_s = F_{rs}, V_r V_s = V_{rs}$.
\item [(ii)] $F_r V_r = r$. When $(r,s) = 1$, then $F_r V_s = V_s F_r$ on 
$E^\bullet_{rm+r-1}$.
\item [(iii)] $V_r (F_r (x)y) = x V_r (y)$ for all $x \in E^\bullet_{rm+r-1}$ and 
$y \in E^\bullet_m$ (projection formula).
\item [(iv)] $F_r d V_r = d$ (where $d$ is the differential operator of the DGAs).
\end{enumerate}
Furthermore, there is a homomorphism of pro-rings $(\lambda: \mathbb{W}_m (R) \to E_m ^0)_{m \in \mathbb{N}}$ compatible with $F_r$ and $V_r$, such that
\begin{enumerate}
\item [(v)] $F_r d \lambda ([a]) = \lambda ([a]^{r-1}) d \lambda ([a])$ for all $a \in R$ and $r \in \mathbb{N}$,
\end{enumerate}
where $[a] \in \mathbb{W}_m (R)$ is the Teichm\"uller lift of $a \in R$.

By \cite{Hesselholt2}, the category of restricted Witt-complexes over $R$ has an initial object, called the de Rham-Witt complex $(\W_m\Omega^\bullet_R)_{m \in \N}$ of $R$. For fixed $m \in \N$, the complex $\W_m\Omega^\bullet_R$ has an explicit description as a quotient of the differential graded algebra $\Omega^\bullet_{{\W_m(R)}/{\Z}}$ (the absolute de Rham complex of $\W_m(R)$) by a differential graded ideal $N^\bullet_m(R)$, which is defined by an explicit set of generators (see \cite[\S 4]{Hesselholt2} or \cite[Proposition 1.2]{R}). One property of $\W_m\Omega^\bullet_R$ is that the canonical maps $\lambda \colon \W_m(R) \to \W_m\Omega^0_R$ (for all $m$) and $\W_1\Omega^\bullet_R \to \Omega^\bullet_R$ (as a map of DGA's) are isomorphisms.

More generally, for any finite truncation set $S \subset \N$, we have the de Rham-Witt complex $\W_S\Omega^\bullet_R := \Omega_{\mathbb{W}_S (R)/\mathbb{Z}}^{\bullet}/ N_S ^{\bullet},$ where $N_S ^{\bullet}$ is a differential graded ideal given by some generators. When a prime $p$ is fixed, taking $S = S_i = \{1, p, \ldots , p^{i-1}\}$, we get $\W_{S_i}\Omega^\bullet_R$, which we shall denote by  $W_i\Omega^\bullet_R$. This is the classical $p$-typical de Rham-Witt complex of $R$  defined by Bloch, Deligne and Illusie (see \cite{Bloch crys} and
\cite{Illusie}).

\begin{remk}
We remark that our definition of Witt complex over $R$ from \cite{HeMa} is a bit different from the version in \cite[Definition 4.1]{Hesselholt2}, but when $2$ is invertible or zero in $R$, both of them coincide. We only consider the case when $R$ contains a field so that $2$ is always either invertible or zero in $R$. We will stick to this definition from \cite{HeMa}, as this definition is simpler. 

In general, for rings that do not necessarily contain a field, \cite{Hesselholt2} makes substantial usages of an element $d \log [-1] = [-1] d [-1]$ for the Teichm\"uller lift $[-1] \in \mathbb{W}_S (R)$ to make corrections for $2$-torsions. We won't need these so that we won't recall them, but one aspect of them motivates Lemma \ref{lem:dada=0} even when $R$ contains a field. We use it later in Proposition \ref{prop:Trace-simple}. \qed.
\end{remk}

\subsection{The $p$-typicalization of de Rham-Witt complex}\label{sec:p-typical}
Let $R$ be a $k$-algebra as above. Then one knows that the ring of Witt-vectors $\W_m(R)$ has a finite set of idempotents which decompose it as a finite product of $p$-typical rings of Witt-vectors. Furthermore, this decomposition is functorial in $R$ and compatible with respect to the restriction maps $\mathfrak{R} \colon \W_m(R) \to \W_{m-1}(R)$ for all $m \in \N$. As with any $\W_m(R)$-module, we get a functorial decomposition (see \cite[Theorem 1.11]{R})
\begin{equation}\label{eqn:p-typical-decom}
\theta_R \colon \W_m\Omega^\bullet_R \xrightarrow{\simeq}
{\underset{n \in I_p}\prod} \W_{\sP \cap (\un{m}/{n})}\Omega^\bullet_R,
\end{equation}
where $\un{m}= \{ 1, \cdots, m \}$, $\sP= \{ 1, p, p^2, \cdots \}$, the set $I_p$ consists of all natural numbers prime to $p$, and $(\un{m}/n)=\{a \in \N| an \in \un{m}\}$. 

\subsection{N{\'e}ron-Popescu approximation}\label{sec:Approx}
To prove the main result of this paper over an imperfect base field, we will need an approximation method based on the N{\'e}ron-Popescu desingularization \cite{Popescu} of regular rings, which we recall below. Details of the proof of this result can also be found in \cite[Theorem~1.1]{Swan}.

\begin{thm}\label{thm:Neron-Popescu}
Let $R$ be a regular semi-local $k$-algebra. Let $k' \subset k$ be the prime field. Then $R$ can be written as a filtering inductive limit $\varinjlim_{\lambda \in I} R_\lambda$ such that each $R_\lambda$ is a semi-local ring which is essentially of finite type and smooth over $k'$.
\end{thm}

When $R$ is essentially of finite type in addition, we can say a bit more (e.g, see \cite[Proof of Theorem~5.11]{Quillen K}):

\begin{lem}\label{lem:Popescu-eft}
Suppose that $p > 1$ and $R$ is a regular semi-local $k$-algebra essentially of finite type. Then we can find a direct system of smooth and essentially of finite type semi-local $\F_p$-algebras $\{R_i\}_{i \in I}$ such that the transition maps $\lambda_{ij} \colon R_i \to R_{j}$ for $i \le j$ are injective and faithfully flat. Moreover, $R = \varinjlim_i R_i$ and each inclusion $R_i \inj R$ is faithfully flat.
\end{lem}

\begin{proof}
Since $R$ is regular, it is a finite direct product of regular semi-local $k$-algebras essentially of finite type, each of which is an integral domain. We can therefore assume without loss of generality that $R$ is an integral domain.

Since $R$ is essentially of finite type over $k$ and semi-local, we can find an integral domain $\bar{R}$, which is a $k$-algebra of finite type, and a finite set $\Sigma$ of primes in $\bar{R}$ such that $R = \bar{R}_{\Sigma}$, the localization at $\Sigma$. Since $R$ is regular and the regular locus of $\Spec(\bar{R})$ is open, we can assume that $\bar{R}$ is regular.

We can now find a subfield $k' \subset k$ which is finitely generated over $\F_p$, a $k'$-algebra $\bar{A}$ of finite type such that $\bar{R} = \bar{A} \otimes_{k'} k$. Since the prime ideals lying in the finite set $\Sigma$ are finitely generated, we can assume (after replacing $k'$ by a bigger finitely generated subfield of $k$) that there is a finite set $\Sigma'$ of primes in $\bar{A}$ such that $\Sigma = \{\fp \bar{R}| \fp \in \Sigma'\}$. By replacing the elements of $\Sigma'$ by the contractions of the elements  in $\Sigma$, we can also assume that $\Sigma' = \{\fp \cap \bar{A}| \fp \in \Sigma\}$. 

Since $\bar{R}$ is a regular $\F_p$-algebra and $\F_p$ is perfect, we see that $R$ is geometrically regular over $\F_p$. Since the inclusion $\bar{A} \inj \bar{R}$ of Noetherian $\F_p$-algebras is faithfully flat, the ring $\bar{A}$ is also geometrically regular over $\F_p$ (see e.g.\cite[{Lemma 07NH}]{stacks-project}). In particular, it is regular.

Setting $R' = {\bar{A}}_{\Sigma'}$ and $R'_{k} =  R' \otimes_{k'} k$, we have the canonical injection $R'_k \inj \bar{R}_{\Sigma} = R$ which is a localization map. Writing $k$ as a direct limit of its finitely generated subfields, we can find  a direct system of finitely generated subfields $\{k_i\}_{i \in I}$ such that $k' \subset k_i \subset k_{j} \subset k$ for $i \le j$, and $k = \bigcup_{i \in I} k_i$. In particular, $R'_k = {\varinjlim_i} \ R'_i$, where $R'_i:= R' \otimes_{k'} k_i$. Since $R'_{j} = R' \otimes_{k'} k_{j} = (R' \otimes_{k'} k_i) \otimes_{k_i}  k_{j} = R'_i \otimes_{k_i} k_{j}$ for $i \le j$, it follows that each transition map $\lambda_{ij} \colon R'_i \to R'_{j}$ is an injective faithfully flat map of Noetherian $\F_p$-algebras such that $R'_k =\varinjlim_i  R'_i$. Furthermore, $R'_{k} = R'_i \otimes_{k_i} k$ and hence the inclusion $R'_i \inj R'_k$ is also faithfully flat. In particular, each $R'_i$ is a regular integral domain.

Since $R'$ is essentially of finite type over $k'$ and $k'$ is finitely generated over $\F_p$, it follows that $R'$ is an $\F_p$-algebra essentially of finite type. Since each $k_i$ is finitely generated over $\F_p$, the same argument implies that each $R'_i$ is an regular $\F_p$-algebra essentially of finite type. Since $\F_p$ is perfect, it follows that each $R'_i$ is smooth over $\F_p$. We let $\Sigma_i = \{\fp \cap R'_i| \fp \in \Sigma\}$ under the inclusion $R'_i \inj R$. It is clear that for each pair of elements $i \le j$ in $I$, the transition map $\lambda_{ij}: R'_i \inj R'_{j}$ induces an inclusion $\lambda_{ij} \colon (R'_i)_{\Sigma_i} \inj (R'_j)_{\Sigma_j}$ such that $\Sigma_i =  \{\fp \cap (R'_i)_{\Sigma_i}| \fp \in \Sigma_j\}$. We let $R_i = (R'_i)_{\Sigma_i}$ for $i \in I$.

Since each $\lambda_{ij}$ is flat and every closed point of $\Spec(R_i)$ is in the image of the map $\lambda^*_{ij} \colon \Spec(R'_j) \to \Spec(R_i)$, it follows that $\lambda_{ij}$ is faithfully flat (see e.g. \cite[Lemma 00HQ]{stacks-project} or \cite[Theorem 7.2, p.47]{Matsumura}). The same reasoning implies that each inclusion $R_i \inj R$ is faithfully flat.

Finally, we note that every element of $R'_k$ which becomes invertible in $R$ lies in the image of $R'_i \inj R$ for some $i \in I$. Such an element must become invertible in $R_i$. It follows that $R =\varinjlim_i  R_i$. This finishes the proof.
\end{proof}

\begin{remk}\label{remk:Popoescu-eft-gen}
If $R$ is a regular $k$-algebra essentially of finite type, that is not necessarily semi-local, the proof of \lemref{lem:Popescu-eft} shows that a weaker version is still valid. Namely, we can find a direct system of smooth $\F_p$-algebras $\{R_i\}_{i \in I}$ essentially of finite type, such that the transition maps $\lambda_{ij} \colon R_i \to R_{j}$ for $i \le j$ are injective and flat. Moreover, $R = \varinjlim_i R_i$ and each inclusion $R_i \inj R$ is flat. However, we may not be able to ensure that either $R_i \to R_j$ or $R_i \to R$ is faithfully flat.
\end{remk}

\begin{lem}\label{lem:Popescu-eft-natural}
Suppose that $p > 1$ and $f \colon R \to S$ is a morphism between two regular semi-local $k$-algebras essentially of finite type. Then we can find direct systems $\{R_i\}_{i \in I}$ and $\{S_i\}_{i \in I}$ as in \lemref{lem:Popescu-eft} together with maps of $\F_p$-algebras $f_i \colon R_i \to S_i$, compatible with the transition maps of the direct systems such that $R = \varinjlim_i R_i$, $S= \varinjlim_i S_i$ and $f = \varinjlim_i f_i$.
\end{lem}

\begin{proof}
Let $\bar{R}$ be the regular finite type $k$-algebra and $\Sigma$ a finite subset of $\Spec(\bar{R})$ as in the proof of \lemref{lem:Popescu-eft} such that $R = \bar{R}_{\Sigma}$. We can similarly find a regular finite type $k$-algebra $\bar{S}$ and a finite subset $\Phi \subset \Spec(S)$ such that $S = \bar{S}_\Phi$. Since $\bar{R}$ is a finite type $k$-algebra, we can extend the map $f \colon R \to S$ to a $k$-algebra morphism $f \colon \bar{R} \to \bar{S}$ (after possibly replacing $\bar{S}$ by a localization which is still finite type over $k$).

Next, we saw in the proof of \lemref{lem:Popescu-eft} that we can find a subfield $k' \subset k$ which is finitely generated over $\F_p$, a finite type $k'$-algebra $\bar{A}$ (resp. $\bar{B}$), and a finite subset $\Sigma' \subset \Spec(\bar{A})$ (resp. $\Phi' \subset \Spec(\bar{B})$) such that $\bar{R} = \bar{A} \otimes_{k'} k$ (resp. $\bar{S} = \bar{B} \otimes_{k'} k$) and $\Sigma$ (resp. $\Phi$) is the extension of $\Sigma'$ (resp. $\Phi'$). Furthermore, $R = R' \otimes_{k'} k$, where $R' = {\bar{A}}_{\Sigma'}$. Similarly, $S = S' \otimes_{k'} k$, where $S' = \bar{B}_{\Phi'}$.

Since $\bar{A}$ is a finite type $k'$-algebra, we can find a field $k''$ such that it is finitely generated over $\F_p$ with $k' \subset k'' \subset k$ and a $k''$-algebra morphism $f \colon \bar{A} \otimes_{k'} k'' \to \bar{B} \otimes_{k'} k''$ which extends $f \colon \bar{R} = \bar{A} \otimes_{k'} k \to \bar{B} \otimes_{k'} k = \bar{S}$. In other words, we can choose our finitely generated subfield $k' \subset k$ over $\F_p$ such that $f \colon \bar{R} \to \bar{S}$ is induced by a morphism of finite type $k'$-algebras $f \colon \bar{A} \to \bar{B}$.

Since after base change to $k$, the map $f \colon \bar{A} \to \bar{B}$ induces the map $R'_k = {\bar{A}}_{\Sigma'} \otimes_{k'} k \to {\bar{B}}_{\Phi'} \otimes_{k'} k = S'_k$, it follows that $f$ must induce a map of localizations $f \colon R' = {\bar{A}}_{\Sigma'} \to {\bar{B}}_{\Phi'} = S'$. We have therefore shown that there is a finitely generated subfield $k' \subset k$ over $\F_p$, a morphism between essentially of finite type regular semi-local $k'$-algebras $f \colon R' \to S'$ which induces our original map $f \colon R =  (R'_k)_{\Sigma} \to (S'_k)_{\Phi} = S$.

We now choose a direct system of finitely generated subfields $\{k_i\}_{i \in I}$ such that $k' \subset k_i \subset k_{j} \subset k$ for every $i \le j$, and $k = \bigcup_{i \in I} k_i$. We let 
\[
R'_i = R \otimes_{k'} k_i, \  S'_i = S' \otimes_{k'} k_i, \ \Sigma_i = \Sigma \cap R'_i, \ \Phi_i = \Phi \cap S'_i,  \ R_i = (R'_i)_{\Sigma_i}, \ \mbox{and} \  S_i = (S'_i)_{\Phi_i}.
\]
The rings $R$ and $S$ are semi-local with the sets of maximal ideals $\Sigma$ and $\Phi$, respectively. It is easy to check using this that the map $f_i \colon R'_i \to S'_i$ induced by $f \colon R' \to S'$ descends to a map between the localizations $f_i \colon (R'_i)_{\Sigma_i} \to (S'_i)_{\Phi_i}$. Equivalently, a map $f_i \colon R_i \to S_i$. The rest of the proof now is a repetition of the proof of \lemref{lem:Popescu-eft}.  
\end{proof}

\subsection{Some properties of de Rham-Witt complex}\label{sec:Basic-results}
We collect some results about de Rham-Witt complexes needed in this paper. 

\begin{lem}\label{lem:Witt-complex-limit}
Let $\{R_\lambda\}_{\lambda \in I}$ be a filtering system of $k$-algebras with $R = \varinjlim_{\lambda \in I} R_\lambda$. Let $(E^\bullet_{m}(\lambda))_{m \in \N, \lambda \in I}$ be a filtering system of differential graded $\Z$-algebras such that $(E^\bullet_{m}(\lambda))_{m \in \N}$ is a restricted Witt-complex over $R_\lambda$ for each $\lambda \in I$. We let $E^\bullet_m = \varinjlim_{\lambda \in I} E^\bullet_{m}(\lambda)$ for $m \in \N$. Then $(E^\bullet_{m})_{m \in \N}$ is a restricted Witt-complex over $R$.
\end{lem}
\begin{proof}
It follows, e.g., from \cite[Proposition~1.16, Lemma~1.17]{R}.
\end{proof} 

\begin{lem}\label{lem:dada=0}
Let $R$ be a $k$-algebra. Then $d [a] \wedge d [a] = 0$ in $\mathbb{W}_m \Omega_R ^2$ for any $m \in \N$ and $a \in R$.
\end{lem}

\begin{proof}
By \cite[Definitions 3.1, 3.6 and p.186]{Hesselholt2}, we have $d[a] \wedge d [a] = d\log[-1] \wedge F_2 ( d [a])$ for each $a \in R$. If ${\rm char} (k) = 2$, then $-1 = 1$ in $k$ so that $d\log [-1] = d\log [1] = 0$. Hence, $d[a] \wedge d [a] = d\log [-1] \wedge F_2 (d [a]) = 0$. If ${\rm char} (k) \not = 2$, then by the anti-commutativity, we have $d [a] \wedge d [a] = - d [a] \wedge d[a]$ so that $2 d [a] \wedge d [a] = 0$. Since $2 \in k^{\times}$ so that $2 \in \mathbb{W}_m (k)^{\times}$, we deduce $d [a] \wedge d [a] = 0$.
\end{proof}

\begin{prop}\label{prop:inject RK}
Let $R$ be a regular semi-local $k$-algebra and let $K$ be the total ring of quotients for $R$. Then for all $m \geq 1$ and $n \geq 0$, the natural map $\mathbb{W}_m \Omega^n _R \to \mathbb{W}_m \Omega^n _K$ is injective.
\end{prop}

\begin{proof}
Let $K$ denote the total ring of quotients of $R$. This is a product of fields since $R$ is regular, and hence reduced. Using N\'eron-Popescu (\thmref{thm:Neron-Popescu} in \S \ref{sec:Approx}), we can write $R = {\varinjlim}_{\lambda \in I} R_\lambda$ such that each $R_\lambda$ is a semi-local ring which is essentially of finite type and smooth over the prime field of $R$. Since taking the total rings of quotients commutes with a filtering inductive limit of rings, it follows that $K = {\varinjlim}_{\lambda \in I} K_\lambda$, where $K_\lambda$ is the total ring of quotients of $R_\lambda$.

We now have a commutative diagram of canonical maps
\begin{equation}\label{eqn:inject RK-0}
\xymatrix@C.8pc{
{\varinjlim}_{\lambda \in I} \mathbb{W}_m \Omega^n _{R_\lambda} 
\ar[r] \ar[d] & \mathbb{W}_m \Omega^n _{R} \ar[d] \\
{\varinjlim}_{\lambda \in I} \mathbb{W}_m \Omega^n _{K_\lambda} 
\ar[r] & \mathbb{W}_m \Omega^n _{K}}
\end{equation}
in which the horizontal arrows are isomorphisms by \cite[Proposition~1.16]{R}. It suffices therefore to prove the proposition when $k$ is perfect and $R$ is a semi-local ring which is essentially of finite type and smooth over $k$. In this case, $R$ is a finite product of integral domains. We can therefore assume further that $R$ is an integral domain.

We prove this case in two steps.

\textbf{Step 1.} Let $m=1$. We show that $\Omega_{R/\mathbb{Z}} ^n \to \Omega_{K/ \mathbb{Z}} ^n$ is injective.

\noindent \textbf{Claim:} \emph{$\Omega^i _{R/ \mathbb{Z}}$ is a free $R$-module, possibly of infinite rank.}

This is obvious for $i=0$. Suppose $i \geq 1$. Consider the Jacobi-Zariski exact sequence of the maps $\mathbb{Z} \to k \to R$ from \cite[3.5.5.1]{Loday} (which generalizes \cite[Proposition 8.3A]{Hartshorne}):
\[
\cdots \to D_1 (R|k) \to \Omega_{k/\mathbb{Z}} ^1 \otimes_k R \to \Omega_{R/ \mathbb{Z}} ^1 \to \Omega_{R/k} ^1 \to 0,
\]
where $D_1 (R|k)$ is the first Andr\'e-Quillen homology of M. Andr\'e \cite{Andre} and D. Quillen \cite{Quillen} (see \cite[3.5.4]{Loday}). Since $R$ is smooth over $k$, we have $D_1 (R|k)= 0$ by \cite[Theorem 3.5.6]{Loday}. On the other hand, since $R$ is a smooth semi-local $k$-algebra, $\Omega_{R/k}^1$ is a free $R$-module. Thus, we have an isomorphism $\Omega_{R/\mathbb{Z}}^1 \simeq \Omega_{R/k} ^1 \oplus ( \Omega_{k/\mathbb{Z}} ^1 \otimes_k R$). Since $\Omega_{k/\mathbb{Z}}^1$ is a free $k$-module (a $k$-vector space), the space $\Omega _{k/\mathbb{Z}} ^1 \otimes_k R$ is a free $R$-module. Hence, $\Omega_{R/\mathbb{Z}}^1$ is a free $R$-module. Taking wedge products, we deduce that $\Omega_{R/\mathbb{Z}} ^i$ is a free $R$-module for all $i \geq 1$, proving the \textbf{Claim}.

Going back to the proof of \textbf{Step 1}, we apply the functor $-\otimes_R \Omega_{R/\mathbb{Z}} ^i$ to the inclusion $R \hookrightarrow K$. By \textbf{Claim}, the module $\Omega ^i _{R/\mathbb{Z}}$ is free so that we get an injection $\Omega_{R/\mathbb{Z}} ^i \hookrightarrow K \otimes_R \Omega_{R/\mathbb{Z}}^i$, where the latter group is isomorphic to $\Omega_{K/\mathbb{Z}} ^i$ by \cite[Proposition 8.2A]{Hartshorne}. Hence, \textbf{Step 1} is done.

\textbf{Step 2.} Now suppose $m \geq 1$. When ${\rm char} (k)  = 0$, by \cite[Remark 1.12]{R}, $\mathbb{W}_m \Omega_R ^n \to \mathbb{W}_m \Omega_K ^n$ decompose into a direct product of maps $\Omega_{R/\mathbb{Z}}^n \to \Omega_{K/\mathbb{Z}} ^n$, each of which is injective by \textbf{Step 1}. Hence, the direct product is also injective.

When ${\rm char} (k) = p>0$, the map $\mathbb{W}_m \Omega_R ^n \to \mathbb{W}_m \Omega_K ^n$ decomposes into a direct product of some copies of maps of $p$-typical de Rham-Witt forms $W_s \Omega_R ^n \to W_s \Omega_K  ^n$ (for various finite values of $s$) by ~\eqref{eqn:p-typical-decom}. Hence, it suffices to prove the injectivity for the $p$-typical de Rham-Witt complex. But, it was proven, for instance by M. Gros \cite[Proposition 5.1.2]{Gros}, that for any smooth $k$-scheme $X$, the Cousin complex of $W_s \Omega_X ^n$ is a resolution of $W_s \Omega_X^n$. In particular, each $W_s \Omega_R ^n \to W_s \Omega_K ^n$ is injective. This completes the proof of the proposition.
\end{proof}

\section{The additive higher Chow groups} \label{sec:ACH}
In this section, we recall the definitions of higher Chow groups and additive higher Chow groups. We shall prove some basic properties of additive Chow groups which will be used in the main proofs. We fix an arbitrary base field $k$.

\subsection{Higher Chow groups}\label{sec:HCG}
We recall (see \cite{Bl1}, \cite{Totaro}) the definition of higher Chow groups. Let $X \in \Sch_k ^{\ess}$ be equidimensional. Let $\P_k^1=\Proj\, k[Y_0,Y_1]$ and $\square^n=(\P_k^1\setminus\{1\})^n$. Let $(y_1, \cdots, y_n) \in \square^n$ be the coordinates. A \emph{face} of $\square^n$ is a closed subscheme defined by a set of equations of the form $y_{i_1} = \epsilon_1, \cdots, y_{i_s} = \epsilon_s$, where $\epsilon_j \in \{ 0, \infty\}$. For $1 \leq i \leq n$ and $\epsilon =0, \infty$, we let $\iota_i ^{\epsilon}: \square^{n-1} \inj \square^n$ be the closed immersion given by $(y_1, \cdots, y_{n-1}) \mapsto (y_1, \cdots, y_{i-1}, \epsilon, y_{i}, \cdots, y_{n-1})$. Its image gives a codimension $1$ face.

Let $q, n \geq 0$. Note that when $X$ is obtained by localizing at a non-closed point, the notion of dimensions for closed subschemes of $X \times \square^n$ could be ambiguous but the codimensions are well-defined. We keep this in mind and in what follows, we use codimensions only unless some ambiguity appears.

Let $\un{z}^q (X, n)$ be the free abelian group on the set of integral closed subschemes of $X \times \square^n$ of codimension $q$, that intersect properly with $X \times F$ for each face $F$ of $\square^n$. We define the boundary map $\partial_i ^{\epsilon} (Z):= [ ({\rm Id}_X \times \iota_i ^{\epsilon})^* (Z)]$. This collection of data gives a cubical abelian group $(\un{n} \mapsto \un{z} ^q (X, n))$ in the sense of \cite[\S 1.1]{KL}, and the groups $z^q (X, n) := \un{z} ^q (X, n) / \un{z} ^q (X, n)_{\rm degn}$ (in the notations of \emph{loc.cit.}) give a complex of abelian groups, whose boundary map at level $n$ is given by $\partial:= \sum_{i=1} ^n (-1)^i (\partial_i ^{\infty} - \partial_i ^0)$. The homology $\CH ^q (X, n):= {\rm H}_n (z^q (X, \bullet), \partial)$ is called the higher Chow group of $X$.

\subsection{Additive higher Chow groups}\label{sec:additive complex}
We recall the definition of additive higher Chow groups from \cite[\S 2]{KP2}. Let $X\in \Sch_k^{\ess}$ be equidimensional. Let $\A^1=\Spec k[t]$,  $\G_m=\Spec k[t,t^{-1}]$, and $\ov{\square} = \mathbb{P}^1_k$. For $n \ge 1$, let $B_n = \mathbb{A}^1 \times \square^{n-1}$, $\ov{B}_n = \A^1 \times \ov{\square}^{n-1}$ and $\widehat{B}_n = {\P}^{1} \times \ov{\square}^{n-1} \supset \ov{B}_n$, where $\square = \ov{\square} \setminus \{1\}$. 
Let $(t, y_1, \cdots , y_{n-1})\in \ov{B}_n$ be the coordinates.

On $\ov{B}_n$, define the Cartier divisors $F_{n,i} ^1 := \{ y_i = 1 \}$ for $1 \leq i \leq n-1$,  $F_{n,0} := \{ t = 0 \}$, and let $F_n ^1 := \sum_{i=1} ^{n-1} F_{n, i} ^1$. A {\em face} of $B_n$ is a closed subscheme defined by a set of equations of the form $y_{i_1}=\epsilon_1, \cdots,  y_{i_s}=\epsilon_s$ where $\epsilon_j\in\{0,\infty\}.$ For $1 \leq i \leq n-1$ and $\epsilon=0, \infty$, let $\iota_{n,i} ^{\epsilon}\colon B_{n-1}\inj B_n$ be the inclusion $(t,y_1,\cdots, y_{n-2})\mapsto (t,y_1,\cdots, y_{i-1}, \epsilon,y_i,\cdots, y_{n-2}).$ Its image is a codimension $1$ face.

The additive higher Chow complex is defined similarly using the schemes $B_n$ instead of $\square^n$, but together with proper intersections with all faces, we impose additional conditions called the \emph{modulus conditions}, that control how the cycles should behave at ``infinity''.

Let $X \in \Sch_k ^{\ess}$, and let $V$ be an integral closed subscheme of $X \times B_n$. Let $\ov V$ denote the Zariski closure of $V$ in $X \times \ov{B}_n$ and let $\nu \colon {\ov V}^N \to {\ov V} \subset X \times \ov{B}_n$ be the normalization of $\ov V$. Let $m \ge 0$ and $n \geq 1$ be integers. As in \cite[Definition 2.1]{KP2}, we say that $V$ satisfies the \emph{modulus $m$ condition} on $X \times B_n$, if we have $(m+1)[{\nu}^*(F_{n,0})] \le [\nu^*(F^1_n)]$  in the free abelian group of 
Weil divisors on ${\ov V}^N$. (N.B. When $n=1$, we have $F_1 ^1 = \emptyset$, so it means $\nu^* (F_{1,0}) = 0$, or $\{ t = 0 \} \cap \ov{V} = \emptyset$.) If $V$ is a cycle on $X \times B_n$, we say that $V$ satisfies the modulus $m$ condition if each of its irreducible components satisfies the modulus $m$ condition. When $m$ is understood, often we just say that $V$ satisfies the modulus condition. N.B. Since $F_{n,0} = \{ t = 0 \} \subset \ov{B}_n$, replacing $\ov{B}_n$ by $\widehat{B}_n$ in the definition does not change the nature of the modulus condition on $V$.

Recall from \cite[Definitions 2.5, 2.6]{KP2} the following.
Let $X \in \Sch_k ^{\ess}$ be equidimensional. Let $m \ge 0$ and $n,q \ge1$ be integers. We let $\un{\TZ}^q (X, 1; m)$ be the free abelian group on integral closed subschemes $Z$ of $X \times \mathbb{A}^1$ of codimension $q$, satisfying the modulus condition.

For $n>1$, let $\un{\TZ}^q (X, n; m)$ be the free abelian group on integral closed subschemes $Z$ of $X \times B_n$ of codimension $q$ such that (1) each face $F$ of $B_n$, $Z$ intersects  $X \times F$ properly on $X \times B_n$ and (2) $Z$ satisfies the modulus $m$ condition on $X \times B_n$.   

For each $ 1 \leq i \leq n-1$ and $\epsilon= 0, \infty$, let $\partial_i ^\epsilon (Z) :=[ ({\rm Id_X}  \times\iota_{n, i}^{\epsilon})^*(Z)]$. The proper intersection with faces ensures that $\partial_i ^{\epsilon} (Z)$ are well-defined.

The cycles in $\un{\TZ}^q (X, n;m)$ are called the \emph{admissible cycles} (or, often the \emph{additive higher Chow cycles, or additive cycles}). This gives the cubical abelian group $(\un{n}\mapsto \un{\TZ}^{q}(X, n+1; m))$ in the sense of \cite[\S 1.1]{KL}. Using the containment lemma \cite[Proposition 2.4]{KP}, that each face $\partial_i ^{\epsilon} (Z)$ lies in $\un{\TZ} ^q (X, n-1; m)$ is implied from $(1)$ and $(2)$. 

The {\em additive higher Chow complex}, or just the \emph{additive cycle complex}, $\TZ^q(X, \bullet ; m)$ of $X$ in codimension $q$ with modulus $m$ is the non-degenerate complex associated to the cubical abelian group $(\un{n}\mapsto \un{\TZ}^{q}(X, n+1; m))$, {i.e.}, $\TZ^q(X, n; m)$ is the quotient $ \un{\TZ}^q(X, n; m)/{\un{\TZ}^q(X, n; m)_{\dgn}}.$

The boundary map of this complex at level $n$ is given by $\partial := \sum_{i=1} ^{n-1} (-1)^i (\partial^{\infty}_i - \partial^0_i)$, and it satisfies $\partial ^2 = 0$. The homology $\TH^q(X, n; m): = {\rm H}_n (\TZ^q(X, \bullet ; m))$ for $n \ge 1$ is the {\it additive higher Chow group} of $X$ with modulus $m$. When $X = \Spec(A)$ is affine, we usually write $\TH^q(X, n; m)$ as $\TH^q(A, n; m)$. In this paper, we study only the Milnor range additive chow groups $\{\TH^n(X,n;m)\}_{m \ge 0, n \ge 1}$.

\subsection{Some properties}\label{sec:ACH-prop}
Like Bloch's higher Chow groups, the additive higher Chow groups are also equipped with pull-back map associated to flat morphisms and push-forward map associated to proper morphisms between schemes. Another less obvious property we need in this paper is the following.

\begin{prop}\label{prop:cap product}
Let $R$ be a smooth $k$-algebra essentially of finite type. Then there is a natural cap product map
$$
\cap_R \colon  \TH^q (R, n; m) \otimes_{\mathbb{Z}} \CH^p (R, n') \to \TH^{p+q} (R, n+n'; m).
$$ 
This commutes with the pull-back and push-forward whenever they exist. This is given by $a \cap_R b = \Delta_R ^* (a \times b)$, where $\Delta_R^*$ is the pull-back via the diagonal map $\Delta_R \colon \Spec (R) \to \Spec (R) \times \Spec(R)$.
\end{prop}

\begin{proof}
This is an immediate consequence of \cite[Theorem~3.12]{KPv}.
\end{proof}

\subsection{Subcomplexes associated to some algebraic subsets}\label{sec:subcx ass}

Let $X \in \Sch_k ^{\ess}$ be equidimensional. Here are some subgroups of $\TZ^q (X, n;m)$ with a finer intersection property with a given finite set $\mathcal{W}$ of locally closed algebraic subsets of $X$:

\begin{defn}[{\emph{cf.} \cite[Definition 4.2]{KP}}]\label{defn:subcx ass}
Define $\un\TZ^q _{\mathcal{W}} (X, n;m)$ to be the subgroup of $\un\TZ^q (X, n;m)$ generated by integral closed subschemes $Z \subset X \times B_n$ that additionally satisfy 
\begin{equation}\label{eqn:add Chow W}
{\rm codim}_{W \times F} (Z \cap (W \times F)) \geq q  \mbox{ for all } W \in \mathcal{W} \mbox{ and all faces } F \subset B_n.
\end{equation} 
The groups $ \un\TZ^q _{\mathcal{W}} (X, n+1;m)$ for $n \geq 0$ form a cubical subgroup of $(n \mapsto \un\TZ^q (X, n+1;m))$ and they give the subcomplex $\TZ^q _{\mathcal{W}} (X, \bullet;m) \subset \TZ^q (X, \bullet;m)$ by modding out by the degenerate cycles. The homology groups are denoted by $\TH^q_{\mathcal{W}} (X, n;m)$.
\end{defn}

\subsection{Atlas for semi-local $k$-schemes}

Recall that a semi-local $k$-algebra $R$ essentially of finite type is of the form $R = \mathcal{O}_{X,\Sigma}$, where $X$ is a quasi-projective $k$-scheme and $\Sigma \subset X$ is a finite set of points (not necessarily closed). The pair $(X, \Sigma)$ will be called \emph{an atlas for} $V= \Spec (R)$. An \emph{affine open subatlas} $(Y,\Sigma)$ of $(X,\Sigma)$ for $V$ is an atlas for $V$ such that $Y \subset X$ is an affine open subset. A \emph{smooth atlas} $(X, \Sigma)$ is an atlas given by a smooth $X$. Note that since $\Sigma$ is finite and $X$ is quasi-projective, we can always choose an affine atlas for $V$. If $R$ is smooth over $k$, then we can find a smooth affine atlas.

\subsection{The $\fs$ and $\sfs$-cycles}\label{section:SFS}
We recall the following notions of fs morphisms, fs-cycles and sfs-cycles from \cite[\S~2]{KP sfs}. We also recall some properties of these cycles.

\begin{defn}[{\cite[Definition 2.6]{KP sfs}}]\label{defn:fs}
Let $X, Y \in \Sch_k ^{\ess}$. When $Y$ is irreducible, we say that a morphism $Y \to X$ of $k$-schemes is \emph{fs over $X$} (or simply \emph{fs} when $X$ is clear), if it is finite and the map $Y \to X_i$ is surjective, where $X_i \subset X$ is an irreducible component in which the image of the entire $Y$ lies. 

When $Y$ is not necessarily irreducible, we say $Y \to X$ is \emph{fs over $X$} if for each irreducible component $Y_j \subset Y$, the induced map $Y_j \to X$ is fs over $X$. 

We generalize it a bit further: let $f: Y \to X$ be a morphism in $\Sch_k ^{\ess}$ and let $U \to X$ be a flat morphism. We say that $Y \to X$ is \emph{fs over $U$}, if the fiber product $f' : Y \times_X U \to U$ is fs.

These notions behave well under base change and finite push-forwards. See \cite[Lemmas 2.7, 2.8]{KP sfs}.
\end{defn}

Let $X \in \Sch_k ^{\ess}$ be equidimensional. Let $m \ge 0$ and
$n \ge 1$ be integers. For $1\le j \le n$, let $\pi_j \colon B_n \to B_{j}$ be the projection given by $(t, y_1, \ldots, y_{n-1}) \mapsto (t, y_1, \ldots, y_{j-1})$. For an irreducible closed subscheme $Z \subset X \times B_n$, let $Z^{(j)} =({\rm Id}_X \times  \pi_j)(Z)$. This $Z^{(j)}$ is in general not a closed subscheme of $X \times B_j$. However, if $Z$ is finite over $X$, then the morphisms in the sequence $Z = Z^{(n)} \to Z^{(n-1)} \to \cdots \to Z^{(1)} \to X$ are all finite, so each $Z^{(j)}$ is closed in $X \times B_j$. If $Z$ is a cycle on $X \times B_n$, we extend it $\mathbb{Z}$-linearly. Recall now the following
from \cite[Definitions 2.11, 2.14]{KP sfs}.

\begin{defn}\label{defn:sfs}
Let $X = \Spec(A)$ be an essentially of finite type
smooth affine $k$-scheme and let $\Sigma \subset X$ be a finite set of points. 
Let $V= \Spec (\mathcal{O}_{X,\Sigma})$.

\begin{enumerate}
\item A cycle $\alpha \in \TZ^n _{\Sigma} (V, n;m)$ is said to be an $\fs$-cycle if each irreducible component of $\alpha$ is fs over $V$. The subgroup of $\fs$-cycles is denoted by $\TZ^n_{\fs} (V, n;m)$.
\item A cycle $\alpha \in \TZ^n_{\Sigma} (V, n;m)$ is said to be an $\sfs$-cycle if it is an $\fs$-cycle, and each irreducible component of $\alpha ^{ (j)}$ is smooth over $k$ for all $1 \leq j \leq n$. The subgroup of $\sfs$-cycles is 
denoted by $\TZ^n_{\sfs} (V, n;m)$.
\end{enumerate}
\end{defn}

We have the following characterization of \sfs-cycles over the semi-local schemes.

\begin{prop}\label{prop:sfs cycles}
Let $V = \Spec(R)$ be a semi-local scheme which is essentially of finite type and smooth over $k$. Let $\Sigma$ be the set of closed points of $V$.  Then an irreducible cycle $Z \in \TZ^n _{\Sigma} (V, n;m)$ is an $\sfs$-cycle if and only if there is a smooth affine atlas $(X = \Spec (A),\Sigma)$ for $V$ and an irreducible cycle $\bar{Z} \in \TZ^n_\Sigma(X, n;m)$ for which the following hold.
\begin{enumerate}
\item $Z= \bar{Z}_V = \bar{Z} \times_X V$.
\item $\bar{Z}$ is closed in $X \times {\widehat{B}}_n$, contained in $ X \times \A^1 \times (\A^1 )^{n-1}$, and does not intersect any proper face $F \subset \square^{n-1}$. In particular, $\bar{Z}$ is fs over $X$.
\item For each $1 \le j \le n$, each $\bar{Z}^{(j)} = ({\rm Id}_X \times \pi_j) (\bar{Z}) \subset X \times B_j$ is an irreducible closed subscheme, where $\pi_j$ is given by $(t, y_1, \ldots, y_{n-1}) \mapsto (t, y_1, \ldots, y_{j-1})$. For the coordinate rings $k[{\bar{Z}}^{(j)}] = {A[t, y_1, \ldots, y_{j-1}]}/{I({\bar{Z}}^{(j)})}$, for $a= \ov{t}$, $b_j = \ov{y}_j$ for $1 \leq j \leq n-1$, we have a sequence of finite extensions of integral domains, $A \subset A[a] \subset A[a, b_1] \subset \cdots \subset A[a, b_1, \ldots, b_{n-1}],$ such that each ring in the sequence is smooth over $k$.
\item There are irreducible monic polynomials $P(t) \in A[t]$ and $Q_j(y_j) \in A[a, b_1, \ldots , b_{j-1}][y_j]$ in $y_j$ for $1 \le j \le n-1$ such that $$A[a] = {A[t]}/{(P(t))} \mbox{  and  } A[a, b_1, \ldots , b_j] = 
{A[a, b_1, \ldots , b_{j-1}][y_j]}/{(Q_j(y_j))}.$$
\end{enumerate}
\end{prop}

\begin{proof}
For the $(\Leftarrow)$ direction, one sees that the existence of a smooth atlas $(X,\Sigma)$ for $V$ and a cycle $\bar{Z} \in \TZ^n_\Sigma(X, n; m)$ satisfying (1)$\sim$(4) imply that $Z \in \TZ^n _\Sigma (V, n; m)$ is an $\sfs$-cycle over $V$. So we need to prove the converse $(\Rightarrow)$.

Suppose that $Z \in \TZ^n _{\Sigma} (V, n ; m)$ is an irreducible sfs-cycle. Since it is an fs-cycle by definition, applying \cite[Lemma 2.9, Proposition 2.13, Lemma 2.21]{KP sfs}, we see that there is a smooth affine atlas $(X, \Sigma)$ for $V$ such that the closure $\bar{Z}$ of $Z$ in $X \times B_n$ is an irreducible admissible cycle in $\TZ^n _\Sigma (X, n;m)$ which satisfies (1) and (2). 

Since $\bar{Z} \to X$ is fs over $X$, there is a sequence of finite maps $\bar{Z} = {\bar{Z}}^{(n)} \to \cdots \to {\bar{Z}}^{(1)} \to X = \Spec(A)$ such that each ${\bar{Z}}^{(j)}$ is fs over $X$. Furthermore, each $\bar{Z}^{(j)} _V $ is smooth over $k$ because $Z$ is an $\sfs$-cycle.

We now want to show that there is a smooth affine open subatlas $(U,\Sigma)$ of $(X,\Sigma)$ for $V$ such that the restrictions $\bar{Z}^{(j)} _U$ to $U$ $1 \leq j \leq n-1$ are all smooth over $k$. To prove it, set $A^j = k [\bar{Z} ^{(j)}]$. As each $\bar{Z} ^{(j)}_V$ is smooth over $k$ and finite over $V$, we see that $\Omega_{A^j/k} ^1$ is a finite $A$-module such that $\Omega_{S_\Sigma^{-1} A^j/k} ^1$ is a free $R$-module, where $R = S_\Sigma^{-1} A$ for the multiplicative subset $S_\Sigma \subset A$ corresponding to the finite set of closed points $\Sigma$. Hence, there is an affine open neighborhood $U$ of $\Sigma$ in $X$ such that each $\Omega^1 _{A^j/k} |_U$ is a free $k[U]$-module. Replacing the given atlas $(X, \Sigma)$ by the new $(U,\Sigma)$, we may thus assume that each $\bar{Z}^{(j)}$ is smooth over $k$. Hence, we have proven (3).

To prove (4), we observe that if we replace $A$ by its semi-local ring $R$ via localization, then we get a sequence of finite extensions of smooth semi-local rings. Note that such rings are UFDs by Auslander-Buchsbaum and $R[a] = R[t]/ I_1$ for the prime ideal $I_1= I ({Z}^{(1)})$. Since $\dim (R)= \dim (R[a]) = \dim (R[t]) -1$, we have ${\rm ht} (I_1) = 1$. But, $R[t]$ is a UFD so that $I_1$ must be principal by \cite[Theorem 20.1, p.161]{Matsumura}. Thus, if $P(t) \in R[t]$ is a monic irreducible polynomial of $a$, then we have $I_1 = (P(t))$. Similarly, we have $R[a,b_1] = R[a][y_1]/I_2$ for the prime ideal $I_2 = I({Z} ^{(2)})$, and since $R[a]$ is a UFD,  so is $R[a][y_1]$. Hence, we obtain $I_2 = (Q_1 (y_1))$ in the same way. 

Continuing as above, we get the irreducible monic polynomials $P(t) \in R[t]$ and $Q_j(y_j) \in R[a, b_1, \ldots , b_{j-1}][y_j]$ for which the property (4) holds over $R$. Choose lifts of these polynomials over $A$, and then there is a localization $A' = A[f^{-1}]$ for some $f \in A$ with the inclusions $A \inj A' \inj R$ such that the property (4) holds over $A'$. Replacing $(X,\Sigma)$ by $(\Spec (A'), \Sigma)$, we obtain a new atlas for $V$ for which all of (1)$\sim$(4) hold.
\end{proof}

\subsection{Additive higher Chow groups of \fs \ and \sfs-cycles}\label{section:ACH-sfs}
Let $m \ge 0$ and $n \geq 1$ be two integers. Recall from \cite[\S~2.6]{KP sfs} that for a smooth and essentially of finite type semi-local $k$-scheme $V$ with the set of closed points $\Sigma$, the additive higher Chow groups of fs and sfs-cycles are defined as follows. We let

\[
\TH^n _{ \fs} (V, n; m) = \frac{{\ker}(\partial: \TZ^n _{\fs}(V, n;m) \to \TZ^n(V, n-1;m))}{{\rm im}  (\partial: \TZ^n (V, n+1;m) \to \TZ^n(V, n;m)) \cap \TZ^n _{\fs} (V, n;m)}.
\]
\[
\TH^n _{\sfs} (V, n; m) = \frac{{\ker}(\partial: \TZ^n _{\sfs}(V, n;m) \to \TZ^n(V, n-1;m))}{{\rm im} (\partial: \TZ^n (V, n+1;m) \to  \TZ^n(V, n;m)) \cap \TZ^n _{ \sfs} (V, n;m)}.
\]

We have a sequence of homomorphisms
 \begin{equation}\label{eqn:moving arrows}
\TH^n _{\sfs} (V, n;m) \to \TH^n _{ \fs} (V, n; m) \to {\TH}^n _{\Sigma} (V, n; m) \to \TH^n (V, n;m).
 \end{equation}

The last arrow in ~\eqref{eqn:moving arrows} is an isomorphism by \cite[Theorem 4.10]{KP3}. The main result of \cite{KP sfs} is the following. This will be another key ingredient in the proof of our main theorem.

\begin{thm}[{\cite[Theorem 1.1]{KP sfs}}]\label{thm:sfs-TCH}
Let $k$ be an infinite perfect field. Let $V = \Spec(R)$ be a smooth semi-local $k$-scheme essentially of finite type with the set of closed points $\Sigma$. Let $m \ge 0$ and $n \ge 1$ be any two integers. Then all maps in ~\eqref{eqn:moving arrows} are isomorphisms.
\end{thm}

\section{Witt-complex structure on additive higher Chow groups}\label{sec:WCADC}
Let $k$ be an arbitrary field. Let $R$ be a regular (not necessarily semi-local) $k$-algebra essentially of finite type. Our goal here is to define the structure of a restricted Witt-complex on the additive higher Chow groups of $\Spec(R)$ over $R$. We need the results of \cite{KP2} and \cite{KP3}. We write $V = \Spec(R)$ and $\TH^\bullet(R;m) = {\underset{n \ge 1}\bigoplus} \TH^n(R, n;m)$ for any $m \ge 0$. Then $(\TH^\bullet(R;m))_{m \in \N}$ is a projective system of graded abelian groups.

Before we describe the restricted Witt-complex structure on the additive higher Chow groups, we recall the following baby case from \cite[Proposition~7.6]{KP3}. For any polynomial $f(t) \in R[t]$ with $f(0) \in R^{\times}$, the closed subscheme $V(f(t)) \subset V \times \A^1$ defines an additive cycle in $\TH^1(R,1;m)$ for any $m \ge 0$. We denote this cycle by $\Gamma_{(f(t))}$. As part of the restricted Witt-complex structure over $R$ on the additive higher Chow groups, the map
\begin{equation}\label{eqn:Structure-map}
\lambda_R \colon \W_m(R) \to \TH^1(R,1;m)
\end{equation}
is determined by $\lambda_R([a]) = \Gamma_{(1-at)}$ for $a \in R$. This uniquely defines $\lambda_R$ because every element in $\W_m(R)$ has a unique expression $w = (a_i) = \sum_{i=1} ^m  V_i([a]_{\lfloor{m/i}\rfloor})$ (see ~\eqref{eqn:WF-1}). In particular, $\lambda_R(w) = \sum_{i=1} ^m  \Gamma_{(1- a_it^i)}$. This map is an isomorphism by \cite[Theorem~7.12]{KP3} if $R$ is a unique factorization domain (e.g., $R$ is regular semi-local).

\subsection{Witt-complex structure for smooth algebras}\label{sec:W-perf}
We now describe the remaining part of the Witt-complex structure on the additive higher Chow groups. 

\begin{lem}\label{lem:WC-structure-perf}
Assume that $R$ is a smooth $k$-algebra. 
Then $(\TH^\bullet(R; m))_{m \in \N}$ is a restricted Witt-complex over $R$.
\end{lem}

\begin{proof}
We assume first that $k$ is perfect. 
Let $V= \Spec (R)$. 
If $p \neq 2$, then the lemma is  \cite[Theorem~1.2]{KP3}. So we assume that $p =2$. In this case, if one closely follows the arguments of \cite{KP2} and \cite{KP3}, then one will notice that there is exactly one place where we used the assumption $p \neq 2$, namely, to show that $\delta^2 = 0$ for the cycle-theoretic differential operator $\delta$ on $\TH^\bullet(R;m) = \TH^\bullet(V;m)$ (see \cite[Definition~6.2]{KP3}). However, it turns out that this vanishing actually follows from the other properties of $\TH^\bullet(V;m)$ proven in \emph{ibids.} without using any assumption on $p$. We explain this below.

Let $\sR \colon \TH^\bullet(V;m+1) \to \TH^\bullet(V;m)$ denote the obvious restriction map. Let $F_r \colon \TH^\bullet(V;rm+r-1) \to \TH^\bullet(V;m)$ denote the Frobenius map given by the push-forward via the map $\phi_r \colon \A^1 \to \A^1$, where $\phi_r(t) = t^r$ (see \cite[\S~7.1]{KP3}). Let $V_r \colon \TH^\bullet(V;m) \to \TH^\bullet(V;rm+r-1)$ be the Verschiebung map given by the flat pull-back via $\phi_r$. It was shown in the proofs of \cite[Theorem~5.13]{KP2} and \cite[Theorem~7.1]{KP3} that $F_r, V_r$ and $\delta$ satisfy the identity
\begin{equation}\label{eqn:WC-structure-perf-0}
F_r \delta V_r = \delta \ \forall \ r \ge 1
\end{equation}
on the additive higher Chow groups of $V$.

As part of the proof of this identity, it was also shown (as the first claim) in the proof of \cite[Theorem~5.13, Part~2, p.~45]{KP2} that $F_r$ and $\delta$ satisfy the identity
\begin{equation}\label{eqn:WC-structure-perf-1}
rF_r \delta = \delta F_r \ \forall \ r \ge 1
\end{equation}
on the additive higher Chow groups of $V$. Furthermore, the proofs of the above two identities do not use any condition on $k$. (N.B. Recall from \cite[Lemma~4.3]{Hesselholt2} that ~\eqref{eqn:WC-structure-perf-1} is one of the properties of a restricted Witt-complex.) We now show that these two identities are enough to conclude that $\delta^2 = 0$.

Suppose in general that $m \ge 0$ is an integer and let $\alpha \in \TH^q (V, n;m)$ be a cycle class. Then the normalization theorem \cite[Theorem 3.2]{KP3} says that $\alpha$ is represented by a cycle (also denoted by $\alpha$) such that $\partial_i ^{\epsilon} (\alpha) = 0$ for all $1 \leq i \leq n-1$ and $\epsilon \in \{ 0, \infty\}$. It follows therefore from \cite[Lemma~6.5]{KP3} that $2 \delta^2 (\alpha) = 0$. Combining this with ~\eqref{eqn:WC-structure-perf-1}, we get $\delta F_2 \delta(\alpha) = 2 F_2 \delta^2(\alpha) = F_2(2 \delta^2(\alpha)) = 0$. Since $\alpha$ and $m \ge 0$ were arbitrary, we conclude that $\delta F_2 \delta = 0$ on $\TH^\bullet(V;m)$ for all $m \ge 0$. In particular, $\delta F_2 \delta V_2 = 0$ on $\TH^\bullet(V;m)$ for all $m \ge 0$. 

We now use ~\eqref{eqn:WC-structure-perf-0} to conclude $\delta^2 = \delta (F_2 \delta V_2) = \delta F_2 \delta V_2 = 0$. This finishes the proof of the 
lemma when $k$ is perfect.

To prove the lemma when $k$ is imperfect, the only thing we need to recall for the 
reader is that in the construction of the Witt-complex structure on additive higher Chow groups in \cite{KP2} and \cite{KP3}, the perfectness of the ground field 
was required exactly at one place, namely, to construct the product structure 
on the additive Chow groups of smooth schemes. More specifically, this was needed only
in the proof of \cite[Lemma~2.11]{KP2}. However, as pointed out by the
referee, the proof of the latter result actually goes through without the
perfectness assumption. The improved version is \lemref{lem:prod-imperfect}
below. The rest of the proof of the Witt-complex structure is identical
to the perfect base field case.
\end{proof}

We let $B_{2, n} = \A^2 \times \square^{n-2}$ and
$\ov{B}_{2, n} = \A^2 \times (\P^1)^{n-2}$ for $n \ge 2$.
Recall the notion on modulus $(m_1, m_2)$ condition from
\cite[Definition~2.9]{KP2}.

\begin{lem}\label{lem:prod-imperfect}
Let $X$ and $Y$ be smooth $k$-schemes, and let $V_1$ and $V_2$ be irreducible cycles satisfying 
the modulus $m_1$ and $m_2$ conditions on $X \times B_{n_1 }$ and 
$Y \times B_{n_2}$, respectively. Then, $V_1 \times V_2$ satisfies the 
modulus $(m_1, m_2)$ condition on $X \times Y \times B_{2, n_1 + n_2}$.
\end{lem}
\begin{proof}
Let $W \subset V_1 \times V_2$ be an irreducible component. It is enough to show that $W$ satisfies the modulus $(m_1, m_2)$ condition. Let $\ov{V}_1 \subset X \times \ov{B}_{n_1}$ and $\ov{V}_2 \subset Y \times \ov{B}_{n_2}$ be the Zariski closures of $V_1$ and $V_2$, respectively. 
Then $\ov{W} \subset X \times Y \times \ov{B}_{2, n_1 + n_2}$ is an irreducible component of $\ov{V}_1 \times \ov{V}_2$.
Let $\ov{V} = (\ov{V}_1 ^N \times \ov{V}_2 ^N)_\red$. 

We now note that $\ov{V}_1 \times \ov{V}_2$ may be neither reduced nor
irreducible. Nonetheless, it is equidimensional and the projection
$\ov{W} \to \ov{V}_i$ is dominant for each $i$.
It follows that $\ov{V}^N$ is equidimensional with $\ov{W}^N$ one of its
irreducible components and the projection $\ov{W}^N \to \ov{V}^N_i$ is
dominant for each $i$. This gives rise to a commutative diagram

\begin{equation}\label{eqn:prod-imperfect-0}
\xymatrix@C.8pc{
\ov{W}^N \ar[r]^-{v_1} \ar[d]_-{v_2} \ar[dr]^-{\nu_{\ov{W}}} 
& \ov{V}^N_1 \times Y \times \ov{B}_{n_2} 
\ar[d]^-{\iota_1} \\
X \times \ov{B}_{n_1} \times \ov{V}^N_2 \ar[r]^-{\iota_2} &
X \times Y \times \ov{B}_{2, n_1 + n_2},}
\end{equation}
where $\iota_1 = (\nu_{\ov{V}_1} \times {\rm Id})$ and $\iota_2 = 
({\rm Id} \times \nu_{\ov{V}_2})$. 
Let $(t_1, t_2, y_1, \cdots , y_{n-1}) \in \mathbb{A}^2 \times \left( \mathbb{P} ^1 \right) ^{n-1}$ be the coordinates. Then, for $D^1: =   \sum_{i=1 } ^{n_1 -1} \{ y_i = 1 \}  - (m_1 +1) \{ t_1 = 0 \}$, $D^2:= \sum_{i={n_1}} ^{ n-1} \{ y_i = 1 \} - (m_2 +1) \{ t_2 = 0 \}$, and $n-1 = (n_1 -1) + (n_2 -1)$, we have 
\begin{equation}\label{eqn:prod-imperfect-1}
\iota^*_1 (D^1) \geq 0 \ \
\mbox{and} \ \ \iota^*_2 (D^2) \geq 0.
\end{equation}
We should note here that $\ov{V}^N_1 \times Y \times \ov{B}_{n_2}$ and
$X \times \ov{B}_{n_1} \times \ov{V}^N_2$ are normal (this is ensured by the smoothness of $X$ and $Y$).

Since $\ov{W}^N \to \ov{V}^N_i$ is dominant, we see that 
$v^*_1 \circ \iota^*_1 (\sum_{i=1 } ^{n_1 -1} \{ y_i = 1 \})$
and $v^*_1 \circ \iota^*_1((m_1 +1) \{ t_1 = 0 \})$
are effective Cartier divisors on $\ov{W}^N$.
It follows from ~\eqref{eqn:prod-imperfect-1} that 
$v^*_1 \circ \iota_1^* (D^1) \geq 0$.
Similarly, we have $v^*_2 \circ \iota^*_2(D^2) \ge 0$. 
Combining the two, we get $\nu^*_{\ov{W}}(D^1+ D^2) \ge 0$, which is the
modulus $(m_1, m_2)$ condition for $W$. 
\end{proof}

Recall from \cite[Theorem~4.5]{KP3} that for any map $f \colon X' \to X$ of $k$-schemes, where $X$ is smooth and affine (or projective) over $k$, there is a pull-back map $f^* \colon \TH^p(X, q;m) \to \TH^p(X', q;m)$. This is compatible with respect to composition of two morphisms. Moreover, the existence of $f^*$ assumes no condition on $k$. This respects the Witt-structures too:

\begin{lem}\label{lem:CW-Pull-back-RWC}
Let $k$ be a field and let $f: \Spec(R') \to \Spec(R)$ be a morphism between smooth affine $k$-schemes essentially of finite type. Then $f^* \colon (\TH(R;m))_{m \in \N} \to (\TH(R';m))_{m \in \N}$ is a morphism of restricted Witt-complexes over $R$.
\end{lem}

\begin{proof}
The lemma is equivalent to proving the following.
\begin{enumerate}
\item [(1)] $f^*$ commutes with products, \ \ (2) $f^*$ commutes with differentials,
\item [(3)] $f^*$ commutes with $\sR, \ F_r$ and $V_r$, \ \ and  \ \ (4) $f^*$ commutes with $\lambda_R$.
\end{enumerate}

The part (3) is shown in \cite[Theorem~7.1]{KP3} while (4) is evident from ~\eqref{eqn:Structure-map}. We prove (1) and (2).

Let $V = \Spec(R)$ and $V' = \Spec(R')$. To prove (1), it suffices to prove it for irreducible cycles. By \cite[Proof of Theorem~7.1]{KP}, there exists a finite set $\mathcal{W}$ of locally closed subsets of $V$ such that the map $f^*\colon \TZ^q_{\sW}(V, \bullet;m) \to \TZ^q(V', \bullet;m)$ given by $f^*(Z) = [f^{-1}(Z)]$, is well-defined. Here, the group on the left is defined as in Definition \ref{defn:subcx ass}. Choose an irreducible cycle $Z \in \TZ^q_{\sW}(V, \bullet;m)$. We claim that there is a finite set $\mathcal{C}$ of locally closed subsets of $V$ such that the following hold.
\begin{enumerate}
\item [(i)]  $f^*\colon \TZ^q_{\sC}(V, \bullet;m) \to \TZ^q(V', \bullet;m)$ is defined.
\item [(ii)] $Z \boxtimes Z' \in \TZ^q_{\{\Delta_V\}}(V \times V, \bullet;m)$ for all $Z' \in \TZ^q_{\sC}(V, \bullet;m)$.
\item [(iii)] $f^*(Z) \boxtimes f^*(Z') \in  \TZ^q_{\{\Delta_{X'}\}}(V' \times V', \bullet;m)$ for all $Z' \in \TZ^q_{\sC}(V, \bullet;m)$.
\end{enumerate}

Let $Z_f$ be the (finite) collection $\{Z_i \cap (V' \times \A^1 \times F)\}$, where $Z_i$ is an irreducible component of $f^*(Z)$ and $F \subset \square^{n-1}$ is a face. Under the moving lemma (\cite[Theorem 4.10]{KP3}), by the  argument of 
\cite[Lemmas~3.5, 3.10]{KPv}, there exists a finite collection $\sW'$ of locally closed subsets of $\{V \times \A^1 \times \square^{n_i}\}$ such that for $\TZ_{\mathcal{W}'} ^q (V, \bullet;m)$ in the sense of \cite[Definition 5.3]{KP Jussieu}, we have:
\begin{enumerate}
\item [(i)] $f^*\colon \TZ^q_{\sW'}(V, \bullet;m) \to \TZ^q_{Z_f}(V', \bullet;m)$ is well-defined.
\item [(ii)] $Z \boxtimes Z' \in \TZ^q_{\{\Delta_X\}}(V \times V, \bullet;m)$ for all $Z' \in \TZ^q_{\sW'}(V, \bullet;m)$.
\item [(iii)] $Z_i \boxtimes f^*(Z') \in  \TZ^q_{\{\Delta_{V'}\}}(V' \times V', \bullet;m)$ for all $Z' \in \TZ^q_{\sW'}(V, \bullet;m)$ and all irreducible components $Z_i$ of $f^*(Z)$.
\end{enumerate}

Furthermore, it follows also by the moving lemma (\cite[Theorem 4.10]{KP3}) and the argument of \cite[Lemma~3.4]{KPv} that there exists a finite collection $\sC'$ of locally closed subsets of $V$ such that $\TZ^q_{\sW'}(V, \bullet;m) = \TZ^q_{\sC'}(V, \bullet;m)$. Setting $\sC = \sW \cup \sC'$, we get the proof of the claim.

We now consider the commutative diagram
\begin{equation}\label{eqn:JJJJJ}
\xymatrix@C.8pc{
V' \times \A^1 \ar[d]_-{f} & & V'\times \A^1 \times \A^1 \ar[rr]^-{\Delta_{V'}} \ar[d]_-{f} \ar[ll]_-{\mu} & & V' \times V' \times \A^1 \times \A^1  \ar[d]^-{f \times f} \\
V \times \A^1 & & V \times \A^1 \times \A^1 \ar[rr]^-{\Delta_V} \ar[ll]_-{\mu} &  & V \times V \times \A^1 \times \A^1,}
\end{equation}
where $\mu \colon \mathbb{A} ^1 \times \mathbb{A}^1 \to \mathbb{A}^1$ is $(t_1, t_2) \mapsto t_1t_2$.

If we choose any irreducible cycle $Z' \in \TZ^q_{\sC}(V, \bullet;m)$, then it follows from the above claim that $\Delta^*_V(Z \boxtimes Z')$ and $\Delta^*_{V'} \circ (f\times f)^*(Z \boxtimes Z')$ are admissible cycles. Moreover, by the commutativity of the right square in \eqref{eqn:JJJJJ}, we have $\Delta^*_{V'} \circ (f\times f)^*(Z \boxtimes Z')= f^*(\Delta^*_V(Z \boxtimes Z'))$. In particular, $f^* \circ \Delta^*_V(Z \boxtimes Z')$ is an admissible cycle.

Now by \cite[Corollary 5.11]{KP3}, the cycles $\mu_* \circ \Delta^*_{V'} \circ (f\times f)^*(Z \boxtimes Z')$ and $\mu_* \circ \Delta^*_V \circ (f\times f)^*(Z \boxtimes Z')$ are admissible. Since the left square in \eqref{eqn:JJJJJ} is transverse, we get $ f^*(Z \cdot Z')  =  f^* \circ \mu_* \circ \Delta^*_V(Z \boxtimes Z') =  \mu_* \circ \Delta^*_V \circ (f\times f)^*(Z \boxtimes Z')  =  f^*(Z) \cdot f^*(Z').$

Finally, the inclusions $\TZ^q_{\sC}(V, \bullet;m) \inj \TZ^q(V, \bullet;m)$ and $\TZ^q_{\sW}(V, \bullet;m) \inj \TZ^q(V, \bullet;m)$ are quasi-isomorphisms by \cite[Theorem~4.10]{KP3}, and this shows (1).

To prove (2), recall from \cite[Definition~6.2]{KP3} that the differential $\delta=\delta_V \colon \TZ^q(V, n;m) \to \TZ^{q+1}(V,n+1;m)$ is defined as the push-forward with respect to the map 
\begin{equation}\label{eqn:Diff}
\delta_V \colon V \times \G_m \times \square^{n-1} \to V \times \A^1 \times \square^n; 
\end{equation}
\[
\ \ \delta (x, t, y_1, \ldots, y_{n-1}) = 
(x, t, t^{-1}, y_1, \ldots, y_{n-1}) \ \ \mbox{for} \ \ t \neq 0, 1.
\]
Here, $\delta_{V*}(Z)$ is an irreducible admissible cycle if $Z$ is so, by \cite[Lemma 6.3, Proposition 6.4]{KP3}. 

Since the square 
$$
\xymatrix@C.8pc{
V' \times \G_m \times \square^{n-1} \ar[d]_{f} \ar[rr]^-{\delta_{V'}} & & V'\times \A^1 \times  \square^{n} \ar[d]^{f} \\
V \times \G_m \times \square^{n-1} \ar[rr]^{\delta_{V}} & & V\times \A^1 \times  \square^{n}}
$$
is transverse, it follows that $\delta_{V' *} \circ f^*(Z) = f^* \circ \delta_{V *}(Z)$ for every irreducible cycle $Z \in \TZ^q_{\sC}(V, \bullet;m)$. We conclude again that $f^* \circ \delta = \delta \circ f^*$ from the quasi-isomorphism $\TZ^q_{\sC}(V, \bullet;m) \inj \TZ^q(V, \bullet;m)$. The proof of the lemma is now complete.
\end{proof}

\subsection{Witt-complex structure for regular algebras}\label{sec:W-imperf}
Our next goal is to generalize Lemmas~\ref{lem:WC-structure-perf} and ~\ref{lem:CW-Pull-back-RWC} to regular semi-local (not necessarily smooth) 
schemes over a field. We shall do this using the continuity property of the
additive higher Chow groups. We need  some intermediate results.

The following lemma is \cite[Lemma~2.2]{KP}.

\begin{lem}\label{lem:surjm}
Let $f : Y \to X$ be a surjective map of normal integral Noetherian
schemes. Let $D$ be a Cartier divisor on $X$ such that $f^*(D) \ge 0$ on $Y$. 
Then $D \ge 0$ on $X$.
\end{lem}

\begin{lem}\label{lem:Elem-comm}
Let $k \inj K$ be an extension of fields. Let $\lambda:R \to R'$ be a faithfully flat morphism of Noetherian $k$-algebras such that $K \inj R'$. Let $X = \Spec(R)$ and $X' = \Spec(R')$. Then the induced map of schemes $X' \times_K \A^1_K \times_K \ov{\square}^{n-1}_K \to X \times_k \A^1_k \times_k \ov{\square}^{n-1}_k$ is also faithfully flat. 
\end{lem}

\begin{proof}
Since the statement of the lemma is local on $X \times_k \A^1_k \times_k \ov{\square}^{n-1}_k$, we can replace $\ov{\square}^{n-1}_k \simeq \P^{n-1}_k$ by $\A^{n-1}_k$. The problem then reduces to showing that the map $R[t, y_1, \ldots , y_{n-1}] \to R'[t, y_1, \ldots , y_{n-1}]$ is faithfully flat. This is obvious because $R \to R'$ is faithfully flat. 
\end{proof}

The next result says that the additive higher Chow functor on the category of affine $k$-schemes with flat morphisms is a continuous contravariant functor.

\begin{lem}\label{lem:limit-0}
Let $k \inj K$ be a field extension, not necessarily finitely generated. Let $R$ be a 
$K$-algebra essentially of finite type. Let $\{R_i\}_{i \in I}$ be a direct system of $k$-algebras essentially of finite type, such that the transition map $\lambda_{ij} \colon R_i \to R_{j}$ is faithfully flat, and $\varinjlim_i R_i = R$. Then the flat pull-backs of (additive) higher Chow groups induce isomorphisms $\varinjlim_i  \ \CH^q(R_i,n) \xrightarrow{\simeq} \CH^q(R,n)$ and $\varinjlim_i  \  \TH^q(R_i, n;m) \xrightarrow{\simeq} \TH^q(R, n;m)$ for all $m \ge 0, n, q \ge 1$.
\end{lem}

\begin{proof}
The proofs for the higher Chow groups and additive higher Chow groups are identical. So, we prove the latter case only. Since the homology functor commutes with the direct limit, it suffices to show that the lemma holds on the level of cycle complexes. We set $V_i = \Spec(R_i)$ and $V = \Spec(R)$. Let $\lambda'_i: R_i \to R$ be the natural map. This is also faithfully flat; this follows from the fact that a direct limit of flat modules is flat, and an $R_i$-module $M$ is faithfully flat if and only if it is flat and $\fm M \neq 0$ for every nonzero maximal ideal $\fm \subset R_i$ (see \cite[Theorem 7.2, p.47]{Matsumura}).

The projections $\Spec(K) \to \Spec(k)$ and $V \to V_j \to V_i$ between $k$-schemes induce the projection maps $V \times_K \A^1_K \times_K \ov{\square}^{n-1}_K \to V_i \times_k \A^1_k \times_k \ov{\square}^{n-1}_k$. This morphism is faithfully flat. It induces the pull-back maps between the cycle groups $\TZ^q(V_i, n;m) \xrightarrow{\lambda^*_{ij}}  \TZ^q(V_j, n;m) \xrightarrow{\lambda '^*_{j}} \TZ^q(V, n;m)$. 

The injectivity of the map ${\varinjlim_i}\ \TZ^q(V_i, \bullet ;m) \to \TZ^q(V, \bullet;m)$ is obvious. To show its surjectivity, let $Z \in \TZ^q(V, n;m)$ be an irreducible admissible cycle and let $\ov{Z} \subset V \times_K \A^1_K \times_K \ov{\square}^{n-1}_K$ be its Zariski closure and $\nu_Z: \ov{Z}^N \to V \times_K \A^1_K \times \ov{\square}^{n-1}_K$ the normalization map. Since each of $f_i := \Spec(\lambda_i) : V_{i+1} \to V_i$ and $f'_i := \Spec(\lambda'_i) : V \to V_i$  is faithfully flat, by \lemref{lem:Elem-comm}, the product maps $\widetilde{f_i}: V_{i+1} \times_k \A^1_k \times_k \ov{\square}^{n-1}_k \to V_i \times_k \A^1_k \times_k \ov{\square}^{n-1}_k$ and $\widetilde{f'_i}: V \times_K \A^1_K \times_K \ov{\square}^{n-1}_K \to V_i \times_k \A^1_k \times_k \ov{\square}^{n-1}_k$ are faithfully flat maps of Noetherian $k$-schemes.

Since $V = \varprojlim_i  V_i$, it follows from the above faithfully flatness that there exists $i \gg 0$ and an irreducible cycle $\ov{Z}_i \inj V_i \times_k \A^1_k \times_k \ov{\square}^{n-1}_k$ such that $(\widetilde{f'_i})^*(\ov{Z}_i) = \ov{Z}$ and $Z_i := \ov{Z}_i \cap (V_i \times_k \A^1_k \times_k {\square}^{n-1}_k) \in z^q(V_i \times_k \A^1_k, n-1)$. 

Since the right square in the commutative diagram below
\begin{equation}\label{eqn:limit-0-01}
\xymatrix@C2pc{
\ov{Z}^N \ar[r] \ar[d] \ar@/^1.5pc/[rr]_{\nu_Z} & \ov{Z} \ar[r] \ar[d] & V \times_K \A^1_K \times_K \ov{\square}^{n-1}_K \ar[d]^{\widetilde{f'_i}} & V \times_K \A^1_K \times_K {\square}^{n-1}_K \ar@{_{(}->}[l] \ar[d] \\
\ov{Z}_i^N \ar[r] \ar@/_ 1.5pc /[rr]^{\nu_{Z_i}} & \ov{Z}_i \ar[r] & V_i \times_k \A^1_k \times_k \ov{\square}^{n-1}_k & V_i \times_k \A^1_k \times_k {\square}^{n-1}_k \ar@{_{(}->}[l]}
\end{equation}
is Cartesian, we must have $Z = (\widetilde{f'_i})^*({Z_i})$.

Since $\widetilde{f'_i}$ is faithfully flat, it follows that $\ov{Z} = (\widetilde{f'_i})^*(\ov{Z}_i) \to \ov{Z}_i$ is also faithfully flat. In particular, it is surjective. We conclude that the induced map $\ov{Z}^N \to \ov{Z}_i^N$ on the normalizations is also surjective. Since $\ov{Z}$ satisfies the modulus condition, we can now apply \lemref{lem:surjm} to conclude that $Z_i$ too satisfies the  modulus condition and hence it is in $\TZ^q(V_i, n;m)$. 
This finishes the proof.
\end{proof}

We can now prove the following improvement of \lemref{lem:WC-structure-perf}:

\begin{thm}\label{thm:WC-structure}
Let $k$ be any field and let $R$ be a regular semi-local $k$-algebra essentially of finite type. Then $(\TH^\bullet(R; m))_{m \in \N}$ is a restricted Witt-complex over $R$.

If $f \colon R \to R'$ is a morphism of regular semi-local $k$-algebras essentially of finite type, then there exists a pull-back 
map $f^* \colon (\TH^\bullet(R; m))_{m \in \N} \to (\TH^\bullet(R'; m))_{m \in \N}$, which is a morphism of restricted Witt-complexes over $R$.
\end{thm}

\begin{proof}
If $k$ is finite, the theorem follows from Lemmas~\ref{lem:WC-structure-perf} and ~\ref{lem:CW-Pull-back-RWC}. Thus we may now assume that $k$ is infinite and $p > 1$.

Let $\{R_i\}_{i \in I}$ be the direct system of subrings in $R$ as in \lemref{lem:Popescu-eft}. Since each transition map $\lambda_{ij} \colon R_i \to R_{j}$ is flat, we have the pull-back maps of projective systems $\lambda^*_{ij} \colon (\TH^\bullet(R_i; m))_{m \in \N} \to (\TH^\bullet(R_{j}; m))_{m \in \N}$. 

Since each $R_i$ is a smooth $\F_p$-algebra essentially of finite type, it follows from \lemref{lem:WC-structure-perf} that $(\TH^\bullet(R_i; m))_{m \in \N}$ is a restricted Witt-complex over $R_i$. Furthermore, it follows from \lemref{lem:CW-Pull-back-RWC} that each $\lambda^*_{ij}$ is a morphism of restricted Witt-complexes over $R_i$. Thus we conclude from \lemref{lem:Witt-complex-limit} that $(\varinjlim_i \TH^\bullet(R_i; m))_{m \in \N}$ forms a restricted Witt-complex over $R$. The first part of the theorem now follows \lemref{lem:limit-0}. The second part follows directly by combining Lemmas~\ref{lem:Popescu-eft-natural}, ~\ref{lem:CW-Pull-back-RWC} and ~\ref{lem:limit-0}. 
\end{proof}

Since the higher Chow groups of smooth schemes over a field have a ring structure, an identical proof also shows the following:

\begin{thm}\label{thm:HCG-product}
Let $k$ be any field and let $R$ be a regular semi-local $k$-algebra essentially of finite type. Then $(\CH^\bullet(R, n))_{n \ge 0}$ is a graded commutative ring. 

If $f \colon R \to R'$ is a morphism of essentially of finite type regular semi-local $k$-algebras, then there exists a pull-back ring homomorphism $f^* \colon (\CH^\bullet(R, n))_{n \ge 0} \to (\CH^\bullet(R', n))_{n \ge 0}$.
\end{thm}

Using Theorems~\ref{thm:WC-structure} and ~\ref{thm:HCG-product}, we get a direct extension of \propref{prop:cap product} to regular algebras over arbitrary fields.

\begin{cor}\label{cor:cap product*}
Let $k$ be any field and let $R$ be a regular $k$-algebra  essentially of finite type. Then there is a natural cap product map
$$\cap_R:  \TH^q (R, n; m) \otimes_{\mathbb{Z}} \CH^p (R, n') \to 
\TH^{p+q} (R, n+n'; m).$$ 
This commutes with the pull-back and push-forward whenever they exist. This is given by $a \cap_R b = \Delta_R ^* (a \times b)$, where $\Delta_R^*$ is the pull-back via the diagonal map $\Delta_R \colon \Spec (R) \to \Spec (R) \times \Spec(R)$.
\end{cor}

\begin{prop}\label{prop:Trace-PF}
Let $f \colon R \to R'$ be a finite and injective morphism of regular semi-local $k$-algebras essentially of finite type. Let $m \ge 0$ and $n, r \ge 1$ be integers. Then the push-forward maps $f_* \colon \TH^n (R', n;m ) \to \TH^n (R, n;m)$ satisfy
\begin{eqnarray*}
& & \mathfrak{R}f_* = f_* \mathfrak{R}; \ \delta f_* = f_* \delta; \ F_r f_* = f_* F_r; \ V_r f_* = f_* V_r.
\end{eqnarray*}
Furthermore, one has $f_*(y f^*(x)) = f_*(y)x$ for $x \in \TH^n (R, n;m)$ and $y \in \TH^n (R', n;m)$. 
\end{prop}

\begin{proof}
Note that $f$ must be flat in our case (see \cite[Exercise III-10.9, p.276]{Hartshorne} or \cite[Proposition (6.1.5), p.136]{EGA4-2}). The first part is \cite[Proposition~7.3]{KP3} and the second part is \cite[Theorem~3.10]{KP2} (whose proof does not require projectivity of the underlying schemes if the morphisms between them is finite and flat). 
\end{proof}

\section{The de Rham-Witt-Chow homomorphism}\label{sec:The-map}
In this section, we construct one of the main objects of our study: the de Rham-Witt-Chow homomorphism. This is a map from the de Rham-Wit forms of a regular semi-local $k$-algebra (essentially) of finite type to its additive higher Chow groups. The main goal of this paper is to show that this map is an isomorphism. This section also includes the key lemma to prove the surjectivity of this map. We fix an arbitrary field $k$. 

\subsection{The homomorphism $\tau^R_{n,m}$}
\label{sec:map-defn}
Let $R$ be a regular semi-local $k$-algebra essentially of finite type. Since $\{\W_m\Omega^{\bullet}_R\}_{m \in \N}$ is the initial object in the category of restricted Witt-complexes over $R$, it follows from \thmref{thm:WC-structure} that there is a unique homomorphism of groups
\begin{equation}\label{eqn:dRWC}
\tau^R_{n,m} \colon \W_m\Omega^{n-1}_R \to \TH^n (R, n;m).
\end{equation}
The homomorphisms $\{\tau^R_{n,m}\}_{m, n \ge 1}$ form a morphism of restricted Witt-complexes over $R$ such that $\tau^R_{1,m}$ is given by ~\eqref{eqn:Structure-map}. We shall often denote the collection $\{\tau^R_{n,m}\}_{m, n \ge 1}$ by $\tau^R_\bullet$.

An easy consequence of \thmref{thm:WC-structure} is the following functoriality of $\tau^R_{n,m}$.

\begin{prop}\label{prop:CW-Pull-back}
Let $f: R \to R'$ be a morphism between regular semi-local $k$-algebras essentially of finite type. Then the diagram
\begin{equation}\label{eqn:CWPB-0}
\xymatrix@C1.8pc{
\W_m \Omega^{n-1}_R \ar[rr]^-{\tau^R_{n,m}} \ar[d]_{f^*} & & \TH^n (R, n;m) \ar[d]^-{f^*} \\
\W_m \Omega^{n-1}_{R'} \ar[rr]^-{\tau^{R'}_{n,m}} & & \TH^n (R', n;m)}
\end{equation}
commutes for all integers $m,n \ge 1$.
\end{prop}

\begin{proof}
The left vertical arrow is a morphism of restricted Witt-complexes over $R$ by \thmref{thm:WC-structure}. The maps $\tau^R_\bullet$ and $\tau^{R'}_\bullet$ are morphisms of restricted Witt-complexes over $R$. It follows that $f^* \circ \tau^R_\bullet$ and $\tau^{R'}_\bullet \circ f^*$ are both morphisms of restricted Witt-complexes over $R$. Using ~\eqref{eqn:Structure-map}, the diagram ~\eqref{eqn:CWPB-0} is easily seen to commute when $n = 1$. Since $\W_m\Omega^{\bullet}_R$ is the universal restricted Witt-complex over $R$, we must have $f^* \circ \tau^R_\bullet = \tau^{R'}_\bullet \circ f^*$ on $\W_m\Omega^\bullet_R$.
\end{proof}

When $R$ is a smooth $k$-algebra, then \lemref{lem:WC-structure-perf} says that $(\TH^\bullet(R; m))_{m \in \N}$ is a restricted Witt-complex over $R$ even if
it not necessarily semi-local. We thus get:

\begin{thm}\label{thm:Perf**}
Let $k$ be any field and let $R$ be a smooth $k$-algebra essentially of finite type. Then there is a unique homomorphism
of groups
\begin{equation}\label{eqn:Perf**-0}
\tau^R_{n,m} \colon \W_m\Omega^{n-1}_R \to \TH^n (R, n;m)
\end{equation}
such that $\tau_{n,m} ^{R}$ is functorial in $R$. 
The homomorphisms $\{\tau^R_{n,m}\}_{m, n \ge 1}$ form a morphism of restricted Witt-complexes over $R$ such that $\tau^R_{1,m}$ is given by ~\eqref{eqn:Structure-map}.
\end{thm}

The following is also a direct consequence of the Witt-complex structure on the additive higher Chow groups of $R$.

\begin{cor}
Since the maps in \eqref{eqn:dRWC} give a morphism of restricted Witt-complexes over $R$, we deduce the following identities:
\begin{equation}\label{eqn:propertiesDRWC}
 \tau_{n,m} ^R d = \delta \tau_{n-1, m} ^R;  \ \tau_{n, rm+r-1} ^R V_r = V_r \tau_{n,m} ^R; \ \tau_{n,m} ^R F_r = F_r \tau_{n, rm+r-1} ^R.
\end{equation}
The second identity of \eqref{eqn:propertiesDRWC} implies the following variation, up to applying $\mathfrak{R}$:
\begin{equation}\label{eqn:propertiesDRWC2}
\tau_{n, m} ^R V_r = V_r \tau_{n, \lfloor{m/r}\rfloor} ^R.
\end{equation} 
\end{cor}

As an easy consequence of \propref{prop:inject RK}, we now have the
following easier part of Theorem \ref{thm:Main-1}:

\begin{lem}\label{lem:Injection}
Let $R$ be a regular semi-local $k$-algebra essentially of finite type. Then $\tau^R_{n,m}$ is a monomorphism
for all $m,n \ge 1$.
\end{lem}

\begin{proof}
We take $R'$ to be the total ring of quotients of $R$ and apply Propositions~\ref{prop:inject RK} and ~\ref{prop:CW-Pull-back}. Since $R'$ is a product of fields,  by the commutative diagram \eqref{eqn:CWPB-0} we are reduced to proving the
lemma when $R$ is a field. But this case follows from \cite[Theorem~1]{R} (when $p \ne 2$) and Theorem~\ref{thm:WC**} of the \S \ref{sec:Appendix} (when $p = 2$).
\end{proof}

\subsection{Traceability}\label{sec:Trace}
To prove that the de Rham-Witt-Chow homomorphism is surjective, we want to have the push-forward map on the additive higher Chow groups and the compatible trace map on the de Rham-Witt complexes, associated to a finite and dominant morphism of regular semi-local $k$-schemes essentially of finite type. The existence of the push-forward map on the additive higher Chow groups is known. Using his duality theory, Ekedahl \cite{Ekedahl} showed the existence of a theory of trace for the $p$-typical de Rham-Witt complex $W_m\Omega^*_R$. Using the $p$-typical decomposition, one can then construct a theory of trace maps for the big de Rham-Witt complex.

Using the sfs-moving lemma (see \thmref{thm:sfs-TCH}) and Ekedahl's trace, one may construct a map $\kappa^R_{n,m} \colon \TH^n(R,n;m) \to \W_m\Omega^{n-1}_R$. In order to either show directly that an sfs-cycle lies in the image of $\tau^R_{n,m}$ or to show that $\tau^R_{n,m} \circ \kappa^R_{n,m}$ is identity, we shall need to verify that $\tau^R_{n,m}$ commutes with Ekedahl's trace on $\W_m\Omega^{n-1}_R$ and the push-forward map on $\TH^n(R,n;m)$. Since this is difficult to check, we use an indirect device to complete our program.

We define the notion of ``traceability" of de Rham-Witt forms using the de Rham-Witt-Chow maps and the push-forwards on the cycle groups. Towards the surjectivity of the de Rham-Witt-Chow homomorphism, we show that all de Rham-Witt forms of a regular semi-local $k$-algebra essentially of finite type are traceable. 

We will start with the known trace map when $n =1$. We will then use a double induction on $m,n \ge 1$ and various properties of Witt-complexes to show the traceability for all $m,n \ge 1$. This is the main goal of \S \ref{sec:Trace}. 
Thanks to the sfs-moving lemma (Theorem~\ref{thm:sfs-TCH}), we only need to do this for simple finite extensions of rings.

We now define traceability:

\begin{defn}\label{defn:traceable}
Let $f \colon R \to S$ be a finite and injective morphism between regular semi-local $k$-algebras essentially of finite type. Let $m,n \ge 1$ be two integers. Given this, we obtain a diagram:
\begin{equation}\label{eqn:Trace-0}
\xymatrix@C1.8pc{
\W_m \Omega^{n-1}_S \ar[r]^-{\tau^S_{n,m}} & \TH^n (S, n;m) \ar[d]^-{f_*} \\
\W_m \Omega^{n-1}_R \ar[r]^-{\tau^R_{n,m}} & \TH^n (R, n;m).}
\end{equation}
We say that a de Rham-Witt form $\omega \in \W_m \Omega^{n-1}_S$ is \emph{traceable to $R$} (via cycles) if  
$f_* \circ \tau^S_{n,m}(\omega) \in {\rm Im}(\tau^R_{n,m})$.
\end{defn}

\begin{defn}\label{defn:simple ext}
A ring extension $R \inj S$ is said to be a \emph{simple extension} if it is flat and there exists a monic polynomial $p(t) \in R[t]$ such that $S \simeq R[t] / (p(t))$. 
\end{defn}

Suppose that $R$ is semi-local and let $e= \deg (p(t))$. Let $a:= t \mod (p(t))$ in $S$. Then $S$ is a finite free $R$-module with an $R$-basis $\{1, a, a^2, \ldots , a^{e-1}\}$. We need the following basic fact about the ring of Witt vectors.

\begin{lem}\label{lem:Witt finite}Let $S$ be an $R$-algebra which is
free as an $R$-module with basis $\{x_1, \ldots, x_e\}$. Let $T$ be a finite truncation set. Then every $\omega \in \mathbb{W}_T (S)$ is uniquely written as 
$$\omega = \sum_{n \in T}  \sum_{i=1} ^eV_n ( [ c_{n,i} ]_{T/n} \cdot [x_i ]_{T/n} ),$$
 for some $c_{n,i} \in R$, where $T/n$ is the truncation set $\{ a \in \mathbb{N} \ | \ an \in T\}$,  $[ - ]_{T/n}$ denotes the Teichm\"uller lift in $\mathbb{W}_{T/n} (R)$, and $V_n$ is the $n$-th Verschiebung operator. 
\end{lem}
\begin{proof}Its proof is similar to that of \cite[Lemma 2.20]{R}, for instance. Let $\omega = (\omega_n)_{n \in T} \in \mathbb{W}_T (S)$. Suppose $\omega \not = 0$, for otherwise there is nothing to prove. Define an operator $\varphi$ as follows: first choose $s_0 = \min\{ s \in T | \ \omega_s \not = 0 \}$. This minimum exists because $\omega \not = 0$. Here, $\omega_{s_0} \in S$ so that there exists a unique expression $\omega _{s_0} = \sum_{i=1} ^e c_{s_0, i} \cdot  x_i$ in $S$ for some $c_{s_0, i} \in R$. Define $\varphi (\omega):= \omega - V_{s_0} (\sum_{i=1} ^e [ c_{s_0, i}]_{T/ s_0} \cdot [ x_i ] _{T/ s_0}).$ We have either $\varphi (\omega) = 0$, or $\varphi (\omega) \not = 0$. In the former case, the argument stops, while in the latter case, there exists $s_1: = \min \{ s \in T | \ \varphi (\omega)_s \not = 0 \}$. By construction, we have $s_1 > s_0$. 
We repeat this procedure. Since $|T|< \infty$, there exists $N \geq 1$ such 
that eventually $\varphi^N (\omega) = 0$.
\end{proof} 

When $S$ is a simple extension of $R$, and $T=\{1, \ldots, m\}$, we immediately deduce that $\omega =  \sum_{i=1} ^m \sum_{j=0} ^{e-1} V_i ([c_{i, j}]_{\lfloor{m/i}\rfloor} \cdot [a]^j_{\lfloor{m/i}\rfloor} )$.

\bigskip

Recall from \cite[Proposition~A.9]{R} that for a finite free extension of rings $R \hookrightarrow S$, and $m \ge 1$, there is a trace map ${\rm Tr}_{S/R} \colon \W_m(S) \to \W_m(R)$ which commutes with the Frobenius and the Verschiebung operators, and satisfies other usual properties of the trace maps. This $\Tr_{S/R}$ is given as follows: for the finite free extension $R[[t]] \to S[[t]]$, we have the norm map ${\rm N}_{S/R} \colon (S[[t]])^{\times} \to (R[[t]])^{\times}$ given by the determinant of the left multiplication maps. This induces a map ${\rm N}_{S/R} \colon (1 + t S[[t]])^{\times} / (1 + t^{m+1} S[[t]])^{\times} \to (1 + tR[[t]])^{\times} / (1 + t^{m+1} R[[t]])^{\times}$. This ${\rm N}_{S/R}$ is the definition of $\Tr_{S/R}$ via the identification \eqref{eqn:WF-1}. 

\begin{lem}\label{lem:tr-p-f-Witt 2}
Let $R \inj S$ be a simple extension of regular semi-local $k$-algebras essentially of finite type and let $m \ge 1$ be an integer. Then the diagram 
\begin{equation}\label{eqn:tr-p-f}
\xymatrix{ 
\mathbb{W}_m(S) \ar[r] ^{\tau^{S}_{1,m} \ \ \ \ } \ar[d]_{\Tr_{S/R}} & \TH^1 (S, 1;m) \ar[d] ^{f_*} \\
\mathbb{W}_m(R) \ar[r]_{\tau^R _{1,m}\ \ \ } & \TH^1 (R, 1;m)}
\end{equation}
commutes, where $f: \Spec (S) \to \Spec (R)$ is the induced map.
\end{lem}  

\begin{proof}
Using Proposition \ref{prop:Trace-PF} and Lemma \ref{lem:Witt finite}, it remains to check that $\tau^R_{1,m} (\Tr_{S/R} ([x])) = f_* ([ \Gamma_{(1-xt)}])$ for all $x \in S$, where $\Gamma_{(1-xt)}$ is the cycle in $\TH^1 (S, 1;m)$ corresponding to the ideal $(1-xt) \subset S[t]$. Since $[x] \in \mathbb{W}_m (S)$ corresponds to $1-xt \in (1+ t S[[t]])^{\times} / (1 + t^{m+1} S[[t]])^{\times}$, by the definition of $\Tr_{S/R}$, we have $\Tr_{S/R} ([x]) = {\rm N}_{S/R} (1 - xt)$. On the other hand, for a polynomial representative $g(t) \in (1 + t R[[t]])^{\times} / (1 + t^{m+1} R[[t]])^{\times}$, we have $\tau_{1,m} ^R ( g(t)) = [\Gamma_{(g(t))}]$, by definition . Hence, $\tau_{1,m} ^R (\Tr_{S/R} ([x])) = \tau_{1,m} ^R ({\rm N}_{R/S} (1-xt)) = [\Gamma_{({\rm N}_{S/R} (1-xt))}]$. 

By \cite[Proposition 1.4(2)]{Fulton}, we have $f_* ({\rm div} (1-xt)) = [ {\rm div} ( {\rm N}_{S/R} (1-xt))]$. Since $1-xt$ and ${\rm N}_{S/R} (1-xt)$ are regular functions, we have ${\rm div} (1-xt) = \Gamma_{(1-xt)}$ and ${\rm div} ({\rm N}_{S/R} (1-xt)) = \Gamma_{({\rm N}_{S/R} (1-xt))}$. This yields the equality of cycles $[\Gamma_{({\rm N}_{S/R} (1-xt))}] = f_* ([ \Gamma_{ (1-xt)}])$. Hence, we have $\tau_{1,m} ^R (\Tr_{S/R} ([x])) =[\Gamma_{({\rm N}_{S/R} (1-xt))}] = f_* ([ \Gamma_{ (1-xt)}])$, as desired.
\end{proof}

The main result of this section is the following.

\begin{prop}\label{prop:Trace-simple}
Assume that $k$ is a perfect field. Let $R \inj S$ be a simple extension of regular semi-local $k$-algebras essentially of finite type and let $m,n \ge 1$ be integers. Then every $ \omega \in \W_m\Omega^{n-1}_S$ is traceable to $R$.
\end{prop}

\begin{proof}
Let $p(t) \in R[t]$ be a monic polynomial of degree $e$ such that $S \simeq {R[t]}/{(p(t))}$. Let $a= t \mod p(t)$ so that $\{1, a, \ldots, a^{e-1}\}$ is an $R$-basis of $S$. For $m,n \ge 1$, let $P_{n,m}$ be the statement: \emph{every member of $\mathbb{W}_m \Omega_{S} ^{n-1}$ is traceable to $R$.}

We prove the proposition by a double induction argument on the variables $(n,m) \in \mathbb{N} \times \mathbb{N}$. 
Note that the statement $P_{1,m}$ holds for all $m \geq 1$ by Lemma~\ref{lem:tr-p-f-Witt 2}. In particular, $P_{1,1}$ is also true.

\textbf{Case 1:} We show first that $P_{n,1}$ holds for all $n \geq 1$. 

\textbf{Subcase 1-1:} To show $P_{2,1}$, note that every element of $\W_1 \Omega^1_S \simeq \Omega^1_{S/\mathbb{Z}}$ is a finite sum of $1$-forms of the type $ca^id(c'a^j) = ca^{i+j}dc' + jcc'a^{i+j-1}da$ for some $c, c' \in R$. So, we are reduced to showing that $1$-forms of the types $ca^idc'$ and $ca^ida$ are traceable for all $c, c' \in R$ and $i \ge 0$. 

For $ca^i dc'$, we have 
\begin{equation}\label{eqn:Tr-cycle-0}
\begin{array}{lll}
 &  \ f_* \circ \tau^S_{2,1}(ca^idc')  {=}^{\dagger}  f_*\left(  \tau^S_{1,1}(a^i)  \cdot  \tau^S_{2,1}(cdc')   \right)   {=}   f_*\left(  \tau^S_{1,1}(a^i)  \cdot  \tau^S_{2,1}(f^*(cdc'))  \right) \\ 
& {=}^\ddagger  f_*\left( \tau^S_{1,1}(a^i)  \cdot  f^*\left(\tau^R_{2,1}(cdc')\right)   \right) {=}^1 f_* (\tau^S_{1,1}(a^i) )\cdot  \tau^R_{2,1}(cdc')    \\
& {=}^2 \tau^R_{1,1}\left({\rm Tr}_{S/R}(a^i)\right)  \cdot  \tau^R_{2,1}(cdc')  {=}^\dagger  \tau^R_{2,1} \left({\rm Tr}_{S/R}(a^i) \cdot (cdc')  \right).
\end{array}
\end{equation}

Here, the equalities ${=}^\dagger$ hold because $\tau_{m,n} ^R$ and $\tau_{m,n} ^S$ are morphisms of DGAs. The equality ${=}^\ddagger$ holds by \propref{prop:CW-Pull-back}, ${=}^1$ by the projection formula for the additive higher Chow groups (e.g., see \cite[Theorem~3.19]{KP2}, whose proof does not require the projectivity of the underlying schemes), and ${=}^2$ holds by Lemma~\ref{lem:tr-p-f-Witt 2}. We conclude that $1$-forms of the type $ca^idc'$ are traceable to $R$.

Next, for $ca^i da$, note that by the part of definition of a restricted Witt-complex in \S \ref{sec:DRW}(v), we can write $ca^i da = c F_{i+1} d[a]$, using the Frobenius operator. Hence,
\begin{equation}\label{eqn:Tr-cycle-1}
\begin{array}{lll}
 &\  f_* \circ \tau^S_{2,1}(ca^ida)  {=}   f_* \circ \tau^S_{2,1}(cF_{i+1}d[a])  =  f_* \circ \tau^S_{2,1}(f^*(c)F_{i+1}d[a])\\
 &=^0 f_* (\tau_{1,1} ^S (f^* (c))\cdot \tau_{2,1} ^S (F_{i+1} d[a]))  = ^{\dagger}  f_* \left(\tau^S_{1,1}(f^*(c)) \cdot F_{i+1} \delta \tau^S_{1,2i+1}([a])\right) \\
 & {=}^\ddagger  f_* \left(f^*(\tau^R_{1,1}(c) ) \cdot F_{i+1} \delta \tau^S_{1,2i+1}([a])\right) {=}^1  \tau^R_{1,1}(c) \cdot f_*F_{i+1} \delta \tau^S_{1,2i+1}([a]) \\
 & {=}^2 \tau^R_{1,1}(c) \cdot F_{i+1} \delta f_*\tau^S_{1,2i+1}([a])  {=}^3  \tau^R_{1,1}(c) \cdot F_{i+1} \delta \tau^R_{1,2i+1}({\rm Tr}_{S/R}([a])) \\
 & {=}^\dagger  \tau_{1,1} ^R (c) \cdot F_{i+1} \tau_{2, 2i+1} ^R (d ({\rm Tr} _{S/R} ([a]))) = ^\dagger \tau^R_{1,1}(c) \cdot \tau^R_{2,1} F_{i+1} d ({\rm Tr}_{S/R}([a])) \\ 
 &{=}^0  \tau^R_{2,1} \left(c \cdot F_{i+1} d ({\rm Tr}_{S/R}(a))\right),
\end{array}
\end{equation}
where the equalities $=^0$ hold because $\tau_{n,m} ^R$ and $\tau_{n,m} ^S$ are morphisms of DGAs, the equalities $=^\dagger$ hold by \eqref{eqn:propertiesDRWC}, the equality ${=}^\ddagger$ holds by Proposition~\ref{prop:CW-Pull-back}, ${=}^1$ holds by the projection formula for $f_*$ and $f^*$, ${=}^2$ holds by Proposition \ref{prop:Trace-PF}, and ${=}^3$ holds by Lemma~\ref{lem:tr-p-f-Witt 2}. We conclude that $1$-forms of the type $ca^ida$ are traceable to $R$. Hence, $P_{2,1}$ is true.

\textbf{Subcase 1-2:} Suppose now that $n>2$ and that the statements $P_{i,1}$ are true for all $1 \leq i < n$. It suffices to show that the forms of the type $\omega = c_0 a^{i_0} d (c_1 a^{i_1}) \wedge \cdots \wedge d (c_{n-1} a^{i_{n-1}})$ are traceable to $R$, where $c_0, \ldots, c_{n-1} \in R$ and $i_0, \ldots, i_{n-1} \geq 0$ are integers. Each $d (c_j a^{i_j}) $ is equal to $a^{i_j} d c_j + i_j c_j a^{i_j -1} d a$ by the Leibniz rule, so that expanding the terms of $\omega$, we reduce to show that every element of the form
\[
\omega_0:= c_0 a^i dc_1 \wedge \cdots \wedge dc_s \wedge \underset{ n-s-1}{\underbrace{ da \wedge \cdots \wedge da}}
\]
is traceable to $R$, where $0 \leq s \leq n-1$ and $c_0, \ldots, c_s \in R$. 

$\bullet$ If $n-s-1 =0$, then the traceability of $\omega$ follows by repeating the steps in \eqref{eqn:Tr-cycle-0} verbatim. 

$\bullet$ If $n-s-1 =1$, let $\omega_0 ' := a^i da$ so that $\omega_0 = c_0 dc_1 \wedge \cdots \wedge dc_s \wedge \omega_0 '$. Since $P_{2, 1}$ is true, we can write $f_* \tau^S_{2,1}(\omega_0') =^\spadesuit \tau^R_{2,1}(\omega_0 '')$ for some $\omega_0 '' \in \Omega^1_{R/\mathbb{Z}}$. Set $\eta := c_0 dc_1 \wedge \cdots \wedge dc_s \in \Omega^s_{R/\mathbb{Z}}$. Then, we have
\begin{equation}\label{eqn:Tr-cycle-2}
\begin{array}{lll}
& \ f_* \circ \tau^S_{n,1}(\omega_0)  =   f_* \circ \tau^S_{n,1} \left(\eta \wedge \omega_0 '\right)  =  f_* \circ \tau^S_{n,1} \left(f^*(\eta) \wedge \omega_0 '\right) \\
& {=}^\dagger  f_*  \left( \tau_{n-1,1} ^S f^* (\eta) \wedge \tau_{2,1} ^S (\omega_0 ') \right) =^\ddagger   f_* \left(f^* \tau^R_{n-1,1}(\eta) \wedge \tau^S_{2,1}(\omega_0 ') \right)  \\
& {=}^1   \tau^R_{n-1,1}(\eta) \wedge f_* \tau^S_{2,1}(\omega_0 ')  {=}^\spadesuit   \tau^R_{n-1,1}(\eta) \wedge \tau^R_{2,1}(\omega_0 '')  {=} ^{\dagger}   \tau^R_{n,1} \left(\eta \wedge \omega_0 ''\right),
\end{array}
\end{equation}
where the equalities ${=}^\dagger$ hold because $\tau_{n,m} ^R$ and $\tau_{n,m}^S$ are morphisms of DGAs, $=^{\ddagger}$ holds by ~\propref{prop:CW-Pull-back}, and ${=}^1$ holds by the projection formula for $f_*$ and $f^*$. We conclude that $\omega_0$ is traceable to $R$.

$\bullet$ If $n-1-s > 1$, then $\omega_0 = 0$ because $da \wedge da=0$ by Lemma \ref{lem:dada=0}. So, $\omega_0$ is traceable to $R$. 

We have thus shown so far that $P_{n,1}$ and $P_{1,m}$ are true for all $n,m \geq 1$. 

\vskip .2cm

\textbf{Case 2:} We now show that $P_{n,m}$ is true in general by using double induction on $(n,m)$. Fix $m, n \geq 2$ and suppose that we know $P_{i,j}$ holds for all $1 \leq i \leq n$, $1 \leq  j \leq m$, except $(i,j) = (n,m)$.

Through the surjection $\Omega_{\mathbb{W}_m (S)} ^{n-1} \surj \mathbb{W}_m \Omega _S ^{n-1}$ and Lemma \ref{lem:Witt finite}, we know that every element in $\W_m\Omega^{n-1}_S$ is a sum of de Rham-Witt forms of the type $\omega = V_{r_0}([c_0][a] ^{i_0}) \cdot d V_{r_1} ([c_1][a]^{i_1}) \wedge \cdots \wedge dV_{r_{n-1}} ([c_{n-1}][a] ^{i_{n-1}})$, where $c_0, \ldots, c_{n-1} \in R$, $r_0, \ldots, r_{n-1} \in \{ 1, \ldots, m \}$, and $ 0 \leq i_0, \cdots, i_{n-1} \leq e-1$.

\textbf{Subcase 2-1:} First, consider the case $r_0 >1$. Let $\omega_0 := d V_{r_1} ([c_1][a]^{i_1}) \wedge \cdots \wedge dV_{r_{n-1}} ([c_{n-1}][a] ^{i_{n-1}})$. In this case, we can write 
\[
\omega = V_{r_0} ([c_0][a] ^{i_0}) \cdot \omega_0 = V_{r_0} ( [ c_0][a] ^{i_0} \cdot F_{r_0} (\omega_0))
\] 
by the projection formula for $V_r$ and $F_r$ (see \S \ref{sec:DRW}(iii)). Since $\omega_0' := [c_0][a] ^{i_0} \cdot F_{r_0} (\omega_0) \in \W_{\lfloor {m/{r_0}} \rfloor}\Omega^{n-1}_S$, it is traceable to $R$ by the induction hypothesis $P_{n,\lfloor {m/{r_0}} \rfloor}$. In particular, there exists $\eta \in \W_{\lfloor {m/{r_0}} \rfloor}\Omega^{n-1}_R$ such that $f_* \tau^S_{n, \lfloor {m/{r_0}} \rfloor}(\omega_0 ') =^\clubsuit \tau^R_{n, \lfloor {m/{r_0}} \rfloor}(\eta)$. This in turn yields 
\[
\begin{array}{lll}
& \ f_* \tau^S_{n,m}(\omega) = f_* \tau^S_{n,m} V_{r_0}(\omega_0 ') =^{\dagger} f_* V_{r_0} \tau_{n, \lfloor {m/{r_0}} \rfloor} ^S (\omega_0')\\
&=^\ddagger V_{r_0} f_* \tau^S_{n, \lfloor {m/{r_0}} \rfloor}(\omega_0 ') = ^\clubsuit  V_{r_0} \tau^R_{n, \lfloor {m/{r_0}} \rfloor}(\eta)= ^\dagger \tau^R_{n,m} \left(V_{r_0} \eta\right),
\end{array}
\]
where the equalities $\dagger$ hold by \eqref{eqn:propertiesDRWC2}, and $\ddagger$ holds by Proposition \ref{prop:Trace-PF}. This shows that $\omega$ is traceable to $R$.

\textbf{Subcase 2-2:} Suppose now that $r_0 = 1$, but for some $j>0$, we have $r_j > 1$. We may assume that $r_1 >1$ without loss of generality. We let $\omega_0:=d V_{r_2} ([c_2][a]^{i_2}) \wedge \cdots \wedge dV_{r_{n-1}} ([c_{n-1}][a] ^{i_{n-1}})$. By the Leibniz rule, we have 
\[
\begin{array}{lll}
& \ V_{r_0}([c_0][a] ^{i_0}) \cdot d V_{r_1} ([c_1][a] ^{i_1})  =  [c_0][a] ^{i_0} \cdot d V_{r_1} ([c_1][a] ^{i_1}) \\
& = d ( [ c_0][a]^{i_0} \cdot V_{r_1} ([c_1][a] ^{i_1})) - V_{r_1} ([c_1][a] ^{i_1})  \cdot d ([c_0][a] ^{i_0}).
\end{array}
\]

Hence, $\omega = \omega_1 - \omega_2$, where $\omega_1:= d ([c_0][a] ^{i_0}\cdot V_{r_1} ([c_1][a] ^{i_1}))\wedge \omega_0 $ and $\omega_2 = V_{r_1} ([c_1][a] ^{i_1}) \cdot d ([c_0][a] ^{i_0}) \wedge \omega_0. $ Let $\omega_1 ' := [c_0][a] ^{i_0}\cdot V_{r_1} ([c_1][a] ^{i_1})$ so that $\omega_1 = d \omega_1' \wedge \omega_0 = d (\omega_1 ' \cdot \omega_0)$. 

Since $\omega_1 ' \cdot \omega_0 \in \W_m\Omega^{n-2}_S$, it follows by the induction hypothesis $P_{n-1, m}$ that there is an element $\eta \in \W_m\Omega^{n-2}_R$ such that $f_* \tau^S_{n-1,m}(\omega_1 ' \cdot \omega_0) =^{\heartsuit} \tau^R_{n-1,m}(\eta)$. Thus, 
\[
\begin{array}{lll}
& \ f_* \tau^S_{n,m}(\omega_1)  =  f_* \tau^S_{n,m}(d(\omega_1 ' \cdot \omega_0)) = ^{\dagger} f_* \delta (\tau_{n-1, m} ^S (\omega_1 ' \cdot \omega_0))\\
& =^\ddagger  \delta f_*\tau^S_{n-1,m}(\omega_1 ' \cdot \omega_0)  = ^{\heartsuit} \delta \tau^R_{n-1,m}(\eta)  =^{\dagger}  \tau^R_{n,m}(d\eta),
\end{array}
\]
where the equalities $=^\dagger$ hold by \eqref{eqn:propertiesDRWC} because $\tau_{n,m}^R$ and $\tau_{n,m} ^S$ are morphism of DGAs, and $=^\ddagger$ holds by Proposition \ref{prop:Trace-PF}. Hence $\omega_1$ is traceable to $R$. Since $\omega_2$ is of the form considered in \textbf{Subcase 2-1}, it is also traceable to $R$. Thus, $\omega= \omega_1 - \omega_2$ is traceable to $R$.

\textbf{Subcase 2-3:} Now, the remaining case is when all $r_0= r_1 = \cdots = r_{n-1} = 1$, {i.e.}, $\omega= [c_0] [a] ^{i_0} d ([c_1][a] ^{i_1}) \wedge \cdots \wedge d( [c_{n-1}][a] ^{i_{n-1}})$. Its proof is almost identical to that of \textbf{Subcase 1-2}, which we argue now. Each $d ([c_j][a] ^{i_j})$ is equal to $[a]^{i_j} d [c_j] + i_j [c_j] [a]^{i_j -1} d [a]$ by the Leibniz rule, so that expanding the terms of $\omega$, we are reduced to show that elements of the form
\begin{equation}\label{eqn:c3}
\omega_0:= [c_0] [a ]^i d [c_1] \wedge \cdots \wedge d [c_s] \wedge 
\underset{n-s-1}{\underbrace{ d[a] \wedge \cdots \wedge d[a] }}
\end{equation}
are traceable, where $0 \leq s \leq n-1$ and $c_0, \ldots, c_s \in R$. 

$\bullet$ If $n-s-1 =0$, then we can use ~\propref{prop:CW-Pull-back}, Lemma~\ref{lem:tr-p-f-Witt 2}, and repeat the steps of \eqref{eqn:Tr-cycle-0} verbatim to conclude that $\omega_0$ is traceable to $R$.

$\bullet$ If $n-s-1 =1$, let $\omega_0 ' := [a]^i d[a]$ so that $\omega_0 = [c_0] d[c_1] \wedge \cdots \wedge d[c_{n-2}] \wedge \omega_0 '$. By the part of definition of a restricted Witt-complex in \S \ref{sec:DRW}(v), we can write $\omega_0 ' = [a] ^i d[a] = F_{i+1}d[a]$. Set $\eta = [c_0] d[c_1] \wedge \cdots \wedge d[c_{n-2}] \in \W_m\Omega^{n-2}_R$, so that $\omega_0= \eta \wedge F_{i+1} d [a]$. (Remember, here $n \geq 2$.) This yields
\[
\begin{array}{lll}
& \ f_*\tau^S_{n,m}(\omega_0)  =  f_*\tau^S_{n,m}(\eta \wedge F_{i+1}d[a]) =  f_* \tau^S_{n,m} \left(f^*(\eta) \wedge  F_{i+1}d[a] \right) \\
&=^{\dagger} f_* ( \tau_{n-2, m} ^S (f^* (\eta))\wedge \tau_{2,m} ^S F_{i+1} d[a]) =^{\ddagger}  f_*\left(f^* \tau^R_{n-2,m}(\eta) \wedge \tau^S_{2,m} F_{i+1}d[a]\right)  \\
&{=}^0  \tau_{n-2, m} ^R (\eta) \wedge f_*\left( \tau_{2,m } ^S F_{i+1} d[a] \right) =^1 \tau_{n-2, m} ^R \wedge f_* ( F_{i+1} \tau_{2, (i+1)m+i} ^S d[a] ) \\
&=^1 \tau^R_{n-2,m}(\eta) \wedge f_*(F_{i+1}  \delta \tau^S_{1,(i+1)m+i}([a]))  {=}^2  \tau^R_{n-2,m}(\eta) \wedge F_{i+1} \delta f_* \tau^S_{1,(i+1)m+i}([a]) \\
& {=}^3  \tau^R_{n-2,m}(\eta) \wedge F_{i+1} \delta \tau^R_{1,(i+1)m+i}({\rm Tr}_{S/R}([a]))  \\
&= ^1 \tau^R_{n-2,m}(\eta) \wedge F_{i+1}  \tau^R_{2,(i+1)m+i} d ({\rm Tr}_{S/R}([a]))  =^1  \tau^R_{n-2,m}(\eta) \wedge \tau^R_{2,m} F_{i+1} d ({\rm Tr}_{S/R}([a]))  \\
&=^\dagger  \tau^R_{n,m} \left(\eta \wedge F_{i+1} d ({\rm Tr}_{S/R}([a])\right),  
\end{array}
\]
where the equalities $=^\dagger$ hold because $\tau_{n,m} ^R$ and $\tau_{n,m} ^S$ are morphisms of DGAs, the equality $=^{\ddagger}$ holds by \propref{prop:CW-Pull-back}, the equality ${=}^0$ is the projection formula for $f_*$ and $f^*$, the equalities $=^1$ hold by \eqref{eqn:propertiesDRWC}, the equality ${=}^2$ holds by Proposition \ref{prop:Trace-PF}, and ${=}^3$ follows from Lemma~\ref{lem:tr-p-f-Witt 2}. This shows that $\omega_0$ is traceable to $R$.

$\bullet$ If $n-s-1 > 1$, we set $\omega_0 ' =  \underset{n-s-1}{\underbrace{ d[a] \wedge \cdots \wedge d[a] }}$. By Lemma \ref{lem:dada=0}, we have $d[a] \wedge d[a] = 0$ in $\W_m\Omega^{2}_S$. In particular, $\omega_0 ' = 0$ in $\W_m\Omega^{n-s-1}_S$ so that $\omega_0 =0$, which is traceable to $R$. We have thus shown that $P_{n,m}$ holds. The proof of the proposition is now complete. 
\end{proof}

\section{The proof of the main result}\label{sec:PRF}
In this section, we complete the proof of our main result, \thmref{thm:Main-1}. We do this first under the assumption that the base field is infinite and perfect, where we apply the sfs-moving lemma (Theorem~\ref{thm:sfs-TCH}). We then use a pro-$\ell$ extension argument to prove the result over any perfect field. Finally, the case of imperfect base field is done using the limit arguments of \S~\ref{sec:WCADC}. 

\subsection{Symbolicity of sfs-cycles}\label{sec:subsection:PML}
For a while in \S \ref{sec:subsection:PML}, we suppose $k$ is an infinite perfect field. 
We begin with the following description of the de Rham-Witt-Chow homomorphism $\tau^R_{n,m}$ on special kinds of de Rham-Witt forms.

\begin{lem}\label{lem:Main-2-symbol}
Let $k$ be an infinite perfect field.
Let $R$ be a regular semi-local $k$-algebra essentially of finite type. Let $m,n \ge 1$ be two integers. Let $a \in R$ and $b_i \in R^{\times}$ for $1 \le i \le n-1$. Then 
\[
\tau^R_{n,m}([a]d\log[b_1] \wedge \cdots \wedge d\log[b_{n-1}]) = Z_{a, \un{b}},
\]
where $Z_{a, \un{b}} = \Spec(\frac{R[t, y_1, \ldots , y_{n-1}]} {(1-at, y_1-b_1, \ldots , y_{n-1}- b_{n-1})})$.
\end{lem}

\begin{proof}
This is an easy consequence of the fact that $\tau^R_{\bullet}$ is a morphism of restricted Witt-complexes over $R$ (see ~\eqref{eqn:dRWC}) and it follows exactly by the method used in the computations in \cite[(7.5)]{KP3}. Indeed, by recursively applying the fact that $\tau^R_{\bullet}$ is a morphism of DGAs, it suffices to show the lemma when $n = 2$.

In this case, we have $\tau^R_{2, m}([a]d\log[b]) {=}^{\dagger} \tau^R_{2, m}([ab^{-1}]d[b]) = \tau^R_{1,m}([ab^{-1}]) \wedge (\delta \circ
\tau^R_{1,m}([b]))$, where ${=}^{\dagger}$ holds because the Teichm{\"u}ller lift map is multiplicative.
Using the definition of the differential $\delta$ on additive higher Chow groups (see ~\eqref{eqn:Diff}), we have
\[
\begin{array}{lll}
& & \tau^R_{1,m}([ab^{-1}]) \wedge (\delta \circ \tau^R_{1,m}([b]))  =  \left[\Spec\left(\frac{R[t]}{(1-ab^{-1}t)}\right)\right] \wedge \left[\delta\left(\Spec(\frac{R[t]}{(1-bt)})\right)\right] \\
& = & \left[\Spec\left(\frac{R[t]}{(1-ab^{-1}t)}\right)\right] \wedge \left[\Spec\left(\frac{R[t,y]}{(1-bt, y-b)}\right)\right] =  \Delta^*_R\left(\frac{(R \otimes_k R)[t, y]} {(1-(ab^{-1}\otimes b)t, y - (1 \otimes b))}\right) \\
\end{array}
\]
and the last term is equal to $Z_{a,b}$ because $\Delta_R: \Spec (R) \to \Spec (R) \times \Spec (R) \simeq \Spec ( R \otimes _k R)$ is induced by the product map $R \otimes_k R \to R$.
\end{proof}

\begin{prop}\label{prop:sfs-sym}
Let $k$ be an infinite perfect field. Let $R$ be a regular semi-local $k$-algebra essentially of finite type. Let $m,n \ge 1$ be two integers. Then every cycle class in $\TH^n_{\sfs}(R,n;m)$ is in the image of $\tau_{n,m} ^R$.
\end{prop}

\begin{proof}
Let $Z \in \TZ^n_{\sfs}(R,n;m)$ be an irreducible $\sfs$-cycle. By Proposition \ref{prop:sfs cycles}, we know that $Z$ is a closed subscheme of $\Spec (R)  \times \A^1_k \times \ov{\square}^{n-1}_k$, which is in fact contained in $\Spec (R)  \times \A^1_k \times \A^{n-1}_k$. Moreover, $\partial Z = 0$ and the ideal $I(Z)$ of $Z$ inside $R[t, y_1, \ldots , y_{n-1}]$ is given by the equations of the form:
\begin{equation}\label{eqn:triangular}
P(t) = 0, \ \ \ Q_1 (t, y_1) = 0, \ \  \ \ldots,  \ \ \ Q_{n-1} (t, y_1, \ldots, y_{n-1}) = 0,
\end{equation}
such that if we let $R_0 = R, R_1 = {R[t]}/{(P(t))}$ and $R_i = {R_{i-1}[t_i]}/{(Q_{i-1})}$ for $2 \le i \le n$, then the rings $\{R_i\}_{1 \le i \le n}$ are all smooth semi-local $k$-algebras such that each extension $R_{i-1} \subset R_i$ is simple. 

Let $f_i \colon \Spec (R_i) \to \Spec (R_{i-1}) $ be the induced finite surjective map of smooth semi-local schemes for $1 \le i \le n$. They are all flat by \cite[Exercise III-10.9, p.276]{Hartshorne} (or \cite[Proposition (6.1.5), p.136]{EGA4-2}). Let $f = f_1 \circ \cdots \circ f_n$.

Let $c^{-1}: = t \ {\rm mod} \ I(Z)$ and $b_i := y_i \ {\rm mod} \ I(Z)$ for $1 \le i \le n-1$. Note that a consequence of the $\sfs$-property of $Z$ is that $c^{-1}, b_i \in R_n^{\times}$ for all $1 \le i \le n-1$. Let $Z_n = 
\Spec\left(\frac{R_n[t, y_1, \ldots , y_{n-1}]}{(1-ct, y_1 - b_1, \ldots , y_{n-1} - b_{n-1})}\right)$ and let $\eta_n:= [c] d\log[b_1] \wedge \cdots \wedge d\log[b_{n-1}]$. It follows that $Z_n \in \TZ^n(R_n,n;m)$ such that $Z = f_*(Z_n)$.

By \lemref{lem:Main-2-symbol}, we have $Z_n =  \tau_{n,m} ^{R_n} (\eta_n)$.
Since $f_n$ is a simple extension, Proposition \ref{prop:Trace-simple} implies that, we have $f_{n*} \tau_{n, m} ^{R_n} (\eta_n) = \tau_{n, m} ^{R_{n-1}} (\eta_{n-1})$ for some $\eta_{n-1} \in \mathbb{W}_m \Omega_{R_{n-1}}^{n-1}$. Since $f= f_1 \circ \cdots \circ f_n$ and each $f_i$ is a simple extension, by successive applications of $f_{i*}$'s and Proposition \ref{prop:Trace-simple}, we obtain $Z = f_* (Z_n) = \tau_{n,m} ^R (\eta_0)$ for some $\eta_0 \in \mathbb{W}_m \Omega_R ^{n-1}$. 
\end{proof}

\begin{cor}\label{cor:main local inf perf}
Theorem \ref{thm:Main-1} holds over all infinite perfect base fields.
\end{cor}

\begin{proof}
A combination of \thmref{thm:sfs-TCH} and \propref{prop:sfs-sym} implies that $\tau_{n,m} ^R$ is surjective, while \lemref{lem:Injection} implies that $\tau_{n,m} ^R$ is injective. 
\end{proof}

\subsection{The finite field case}\label{sec:FF}
In \S \ref{sec:FF}, we prove Theorem \ref{thm:Main-1} when the base field is finite. In particular, together with Corollary \ref{cor:main local inf perf}, this shows that Theorem \ref{thm:Main-1} holds for all perfect fields. We will achieve this using a pro-$\ell$ extension argument.

For a moment, we let $k$ be any field with the exponential characteristic $p \geq 1$, and let $R$ be a regular semi-local $k$-algebra essentially of finite type. Let $m,n \ge 1$ be two integers. We shall use the following general result.

\begin{lem}\label{lem:fp_inj0}
Let $k \hookrightarrow k'$ be a finite separable extension of fields with $d:= [k':k]$ such that $(p, d)= 1$. Let $f \colon R \to R_{k'}$ be the base change homomorphism. Then $f^*\colon \TH^n (R, n;m) \to \TH^n (R_{k'}, n;m)$ is injective.
\end{lem}

\begin{proof}
We write $R' = R_{k'}$. Since $f$ is a finite \'etale morphism of degree $d$, we have $f_* f^* = d \cdot ({\rm Id})$ on $\TH^n (R, n;m)$. On the other hand, $\TH^n (R, n;m)$ is an $\mathbb{W}_m (k)$-module (see \cite[Theorem 6.13]{KP3}), while $d \in \mathbb{W}_m (k)^{\times}$. This is clear if $p =1$.
Otherwise, this follows from the fact that $\W_m(k)$ is a
$\W_m(\F_p)$-algebra and the latter is a finite product of rings of the form
${\Z}/{p^n}$. Thus, $\frac{1}{d} f_* \circ f^* = {\rm Id}$ on $\TH^n (R, n;m)$. In particular, $f^*$ is injective.
\end{proof}

\begin{cor}\label{cor:fp_inj0-1}
Assume $k$ is a finite field and let $k'$ be a pro-$\ell$ algebraic extension of $k$ for a prime $\ell \not = p$. Let $f \colon R \to R_{k'}$ be the base change homomorphism. Then $f^*\colon \TH^n (R, n;m) \to \TH^n (R_{k'}, n;m)$ is injective.
\end{cor}

\begin{proof}
We can express $k'= \varinjlim_{i \in I} k_i$ for fields $k_i$ indexed by a directed set $I$, where each $k_i$ is a finite extension of $k$ such that $[k_i : k] = \ell ^{N_i}$ for some $N_i \in \mathbb{N}$. In particular, $(p, [k_i, k]) = 1$. We then have $R_{k'} = \varinjlim_{i \in I} R_{k_i}$ and $\TH^n (R_{k'}, n;m) = \varinjlim_{i \in I} \TH^n (R_{k_i}, n;m)$ by \lemref{lem:limit-0}. The assertion of the corollary therefore follows from Lemma \ref{lem:fp_inj0}.
\end{proof}

\begin{lem}\label{lem:Trace*}
Assume $k$ is a perfect field and let $k'$ be a finite extension of $k$ of degree $d$ such that $(d,p) = 1$. Let $f \colon R \to R_{k'}$ be the base change homomorphism. Let $\alpha \in \TH^n (R, n;m)$ be an element such that $f^*(\alpha) \in {\rm Im}(\tau^{R_{k'}}_{n,m})$. Then $\alpha \in {\rm Im}(\tau^{R}_{n,m})$.
\end{lem}

\begin{proof}
We write $R' = R_{k'}$. Note that $R \inj R'$ is a finite {\'e}tale extension. In particular, $R'$ is a regular semi-local $k$-algebra essentially of finite type. Recall also that a consequence of Witt-complex structures is that $\W_m\Omega_{R'} ^{n-1}$ and $\TH^n (R', n;m)$ are $\W_m(k)$-modules and $\tau^{R'}_{n,m}$ is $\W_m(k)$-linear. Let $f^*(\alpha) = \tau^{R'}_{n,m}(\omega')$ for some $\omega' \in \mathbb{W}_m \Omega_{R'} ^{n-1}$.  Let $\omega = d^{-1} \omega' \in \W_m\Omega_{R'} ^{n-1}$. This makes sense since
$d \in \W_m(k)^{\times}$.
Then $\tau^{R'}_{n,m}(\omega) = d^{-1}f^*(\alpha)$.

We also have the push-forward map $f_* \colon \TH^n (R', n;m) \to \TH^n (R, n;m)$ such that $f_* \circ f^* = d \cdot {\rm Id}$ on $\TH^n (R, n;m )$ . Using this, we get
\[
\alpha = f_*(d^{-1}f^*(\alpha)) = f_* \circ \tau^{R'}_{n,m}(\omega).
\]
On the other hand, the primitive element theorem implies that $R \inj R' = R_{k'}$ is a simple extension of regular semi-local $k$-algebras essentially of finite type. We can therefore apply \propref{prop:Trace-simple} to conclude that
$\alpha \in {\rm Im}(\tau^{R}_{n,m})$.
This finishes the proof.
\end{proof}

\begin{prop}\label{prop:Main-finite}
Theorem \ref{thm:Main-1} holds over all perfect base fields.
\end{prop}

\begin{proof}
Since the map $\tau^R_{n,m}$ is injective by \lemref{lem:Injection}, we only need to show that $\tau^R_{n,m}$ is 
surjective. Since we already know that Theorem \ref{thm:Main-1} holds for infinite perfect fields by \corref{cor:main local inf perf}, we can now assume that the base field $k$ is finite. 

We let $k'$ be a pro-$\ell$ algebraic extension of $k$ for a prime $\ell \not = p$. Let $f \colon R \to R_{k'}$ be the base change homomorphism. We write $k'= \varinjlim_{i \in I} k_i$ for fields $k_i$ indexed by a directed set $I$ as in the proof of \corref{cor:fp_inj0-1}. Let $R_i :=  R_{k_i}:= R \otimes_k k_i$ and $R' := R_{k'}$.

For every $i \in I$, we have a commutative diagram
\begin{equation}\label{eqn:Main-finite-0}
\xymatrix@C.8pc{
\mathbb{W}_m \Omega_{R} ^{n-1} \ar[r]^-{f_i^*} \ar[d]_-{\tau^R_{n,m}} & \mathbb{W}_m \Omega_{R_i} ^{n-1} \ar[r] \ar[d]_-{\tau^{R_i}_{n,m}} & \varinjlim_{i \in I} \mathbb{W}_m \Omega_{R_i} ^{n-1} \ar[r]^-{\simeq} \ar[d] & \mathbb{W}_m \Omega_{R'} ^{n-1} \ar[d]^-{\tau^{R'}_{n,m}}  \\
\TH^n (R, n;m) \ar[r]^-{f^*_i} & \TH^n (R_i, n;m) \ar[r] & \varinjlim_{i \in I} \TH^n (R_i, n;m) \ar[r]^-{\simeq} &  \TH^n (R', n;m)}
\end{equation} 
by \propref{prop:CW-Pull-back} such that the composite horizontal arrows on the top and the bottom are $f^*$. The horizontal top right arrow is an isomorphism by \cite[Proposition~1.16]{R} and the horizontal bottom right arrow is an isomorphism by \lemref{lem:limit-0}. 

Let $\alpha \in \TH^n (R, n;m)$ be arbitrary. Since $k'$ is infinite and perfect, we can write $f^*(\alpha) = \tau^{R'}_{n,m}(\omega)$ for some $\omega \in \mathbb{W}_m \Omega_{R'} ^{n-1}$. It follows that there exist $i \in I$ and $\omega_i \in \mathbb{W}_m \Omega_{R_i} ^{n-1}$ such that $\lambda^*_i(f^*_i(\alpha) - \tau^{R_i}_{n,m}(\omega_i)) = 0$, where $\lambda_i \colon R_i \inj R'$ is the inclusion map. We conclude from \corref{cor:fp_inj0-1} that $f^*_i(\alpha) = \tau^{R_i}_{n,m}(\omega_i)$. We now apply \lemref{lem:Trace*} to deduce that $\alpha \in {\rm Im} (\tau_{n, m} ^R)$. This finishes the proof.
\end{proof}

\subsection{The general case}\label{sec:General}
We shall now finish the proof of \thmref{thm:Main-1} using a limit argument. Let $k$ be an arbitrary field of exponential characteristic $p > 1$. Let $R$ be a regular semi-local $k$-algebra essentially of finite type. We write $R = \varinjlim_{i \in I} R_i$ as in \lemref{lem:Popescu-eft}, for a directed system of regular semi-local $\F_p$-algebras $R_i$ essentially of finite type, whose transition maps are injective and faithfully flat. Let $m,n \ge 1$ be two integers. This leads to a diagram of restricted Witt-complexes over $R$:
\begin{equation}\label{eqn:Final-local}
\xymatrix@C.8pc
{{\varinjlim_i} \ \W_m\Omega^{n-1}_{R_i} \ar[rrr]^{{\varinjlim_i}  \tau^{R_i}_{n,m} \ \ \ } \ar[d] & & & {\varinjlim_i} \ \TH^n(R_i, n;m) \ar[d] \\
\W_m\Omega^{n-1}_{R} \ar[rrr]^{\tau^{R}_{n,m} \ \ } & & & \TH^n(R, n;m).}
\end{equation}

The left vertical arrow is an isomorphism of restricted Witt-complexes by \cite[Proposition~1.16]{R} and the right vertical arrow is an isomorphism of restricted Witt-complexes by \lemref{lem:limit-0}. It follows from \propref{prop:CW-Pull-back} that ~\eqref{eqn:Final-local} is commutative. 

\begin{prop}\label{prop:Main-gen}
Theorem \ref{thm:Main-1} holds over all base fields.
\end{prop}

\begin{proof}
Since the result holds for all perfect base fields by \propref{prop:Main-finite}, to deal with the remaining imperfect base field case we can assume $p > 1$. In this case, we can use the commutative diagram ~\eqref{eqn:Final-local}. It follows from \propref{prop:Main-finite} that $\tau^{R_i}_{n,m}$ is an isomorphism for each $i \in I$. In particular, the top horizontal arrow in ~\eqref{eqn:Final-local} is an isomorphism. Since the two vertical arrows are isomorphisms, we conclude that the bottom arrow $\tau^R_{n,m}$ is an isomorphism. The naturality of $\tau^R_{n,m}$ follows from \propref{prop:CW-Pull-back}.
\end{proof}

\subsection{$p$-typicalization of additive Chow groups}\label{sec:p-add}
Let $k$ be any field of exponential characteristic $p > 1$ and let $R$ be a regular $k$-algebra essentially of finite type. Let $\sP = \{1, p, p^2, \ldots\}$. Recall from ~\eqref{eqn:p-typical-decom} that there is a direct product decomposition into the p-typical ones for each $n \geq 1$:
\begin{equation}\label{eqn:p-typical-decom1}
\theta_R \colon \W_m \Omega_R ^{n-1} \xrightarrow{\simeq} 
{\underset{n \in I_p}\prod} \W_{\sP \cap ( \un{m}/ n) } \Omega_R ^{n-1},
\end{equation}
where $\un{m}= \{ 1, \cdots,m \}$ and $I_p$ is the set of positive integers prime to $p$. This decomposition is natural in $R$.

\begin{defn}\label{defn:p-typ0}
Suppose $R$ is semi-local in addition. We can use Theorem \ref{thm:Main-1} to define the $p$-typical additive higher Chow groups as

\begin{equation}\label{eqn:p-typical-add-0}
\TH^n(R;p^i) := \tau^{R}_{n, p^{i-1}} \circ \theta^{-1}_R(\W_{\sP \cap \un{p^{i-1}}} \Omega^{n-1}_R),
\end{equation}
where $W_i \Omega_R ^{n-1} = \W_{\sP \cap \un{p^{i-1}}} \Omega^{n-1}_R$ is identified with
$\W_{\sP \cap \un{p^{i-1}}} \Omega^{n-1}_R \times \{0 \} \times \cdots \times \{0\}$ on the right hand side of \eqref{eqn:p-typical-decom1}. 
\end{defn}
We thus have an immediate application of \thmref{thm:Main-1}:

\begin{cor}\label{cor:p-typ*}
Let $k$ be a field of exponential characteristic $p>1$ and let $R$ be a regular semi-local $k$-algebra essentially of finite type. For every $n, i \ge 1$, there is a natural isomorphism of $p$-typical Witt-complexes
\begin{equation}\label{eqn:p-typical-add-1}
\tau^R_{n,p^i} \colon W_i\Omega^{n-1}_R \xrightarrow{\simeq} \TH^n(R;p^i).
\end{equation}
\end{cor}

\begin{remk}
One may hope that these $p$-typical additive higher Chow groups can be defined for a bit more general schemes beyond the semi-local case, e.g., all regular $k$-schemes essentially of finite type. For this purpose, we may not be able to use Theorem \ref{thm:Main-1}. We do not know how to do it in general, but at least when $k$ is perfect, we can use Lemmas~\ref{lem:WC-structure-perf} and ~\ref{lem:CW-Pull-back-RWC} (or Theorem \ref{thm:Perf**}). Here, $\{ \TH^n (R, n; m) \}_{n, m \geq 1 }$ forms a restricted Witt-complex over $R$. When $m= p^{i-1}$, the ring $\mathbb{W}_m (R)$ acts on $\TH^n (R, n; p^{i-1})$. We shall define the
$p$-typical additive higher Chow groups of regular $k$-algebras as follows.
\end{remk}

\begin{defn}\label{defn:p-typ**}
Let $k$ be a perfect field of exponential characteristic $p>1$ and let $R$ be a regular $k$-algebra essentially of finite type. 
For any integers $n, i \ge 1$, we define the $p$-typical additive higher Chow groups of $R$ to be the $\W_{p^{i-1}}(R)$-submodule
\begin{equation}\label{eqn:p-typical-add}
\TH^n(R;p^i) := \W_{\sP \cap \un{p^{i-1}} }(R) \cdot \TH^n(R,n;p^{i-1}),
\end{equation}
where we are regarding the factor $W_{i} (R) = \W_{\sP \cap \un{p^{i-1}} }(R)$ of $\mathbb{W}_{p^{i-1}} (R)$ in the decomposition \eqref{eqn:p-typical-decom1} for $n=1$, as an ideal of $\mathbb{W}_{p^{i-1}} (R)$. 
\end{defn}

The subgroups $\TH^n(R;p^i) \subset \TH^n(R,n;p^{i-1})$ are natural in $R$. Note that when $R$ is semi-local and $k$-perfect, the definitions \eqref{eqn:p-typical-add-0} and \eqref{eqn:p-typical-add} coincide. 

\section{Applications}\label{sec:Appln}
In this section, we discuss several applications of \thmref{thm:Main-1}. Apart from these, we mention that \thmref{thm:Main-1} is also an essential part of the proofs of the main results of \cite{GK-1} and \cite{GK-2}. We fix a field $k$.

\subsection{Crystalline cohomology via algebraic cycles}\label{sec:Crys}
Let $k$ be a perfect field. It was shown in \cite[Theorem~4.5]{KP3} that the additive higher Chow groups form presheaves of abelian groups on the category $\SmAff^\ess_k$ of smooth affine schemes essentially of finite type over $k$. Using the left Kan extension, these groups extend to presheaves on $\Sch^\ess_k$ as follows.

For $X \in \Sch^\ess_k$, let $(X \downarrow \SmAff^\ess_k)$ be the category whose objects are the $k$-morphisms $X \to A$, with $A \in \SmAff^\ess_k$. A morphism from $h_1: X \to A$ to $h_2: X \to B$, with $A, B \in \SmAff^\ess_k$ is given by a $k$-morphism $g: A \to B$ such that $g\circ h_1 = h_2$. The category $(X \downarrow \SmAff^\ess_k)$ is clearly cofiltered (see \cite[\S~4.4]{KP3}). 

For $m \ge 0$, $n, q \ge 1$ and $X \in \Sch^{\ess}_k$, we let
\[
\wt{\TH} ^q (X,n;m):= 
\underset{A \in (X \downarrow \SmAff^\ess_k)^{\op}}{\colim} \TH^q(A, n;m).
\]

By \cite[Proposition 4.8]{KP3}, we know that $\wt{\TH} ^q (-, n;m)$ is a presheaf on $\Sm^\ess_k$ and $\Sch^\ess_k$. There is a natural homomorphism $\alpha_X \colon \wt{\TH} ^q  (X, n;m) \to \TH^q (X, n;m)$. This is an isomorphism for $X \in \SmAff^{\rm ess}_k$. We let $\tch ^q (n;m)$ denote the sheaf on the big Zariski site of $\Sch^\ess_k$ associated to the presheaf $\wt{\TH} ^q (-, n;m)$. For $X \in \Sch^\ess_k$, we denote the restriction of $\tch ^q (n;m)$ to the small Zariski site of $X$ by  $\tch ^q (n;m)_X$. 

\vskip .2cm

Assume now that $p >1$. For $n, i \ge 1$ and $X \in \Sch^\ess_k$, using Definition \ref{defn:p-typ**}, we let
\begin{equation}\label{eqn:presheaf TCH-0}
\wt{\TH}^n (X; p^i) := \underset{A \in (X \downarrow \SmAff^\ess_k)^{\op}}{\colim} \TH^n(A;p^i).
\end{equation}
Since the $p$-typical additive higher Chow groups are presheaves on smooth affine schemes over $k$, it follows from \cite[Proposition 4.8]{KP3} that $\wt{\TH}^n(-; p^i)$ is a presheaf on  $\Sm^\ess_k$ and $\Sch^\ess_k$. We let $\tch ^n (p^i)$ denote the sheaf on the big Zariski site of $\Sch^\ess_k$ associated to the presheaf $\wt{\TH} ^q (-; p^i)$. For $X \in \Sch^\ess_k$, we denote the restriction of $\tch ^n (p^i)$ to the small Zariski site of $X$ by  $\tch ^n (p^i)_X$. The Leibniz rule for the additive Chow higher groups implies that $\tch ^\bullet (p^i) := \bigoplus_n \tch ^n (p^i)$ is a subcomplex of $\tch ^\bullet (\bullet; p^{i-1})$.

Recall that $\W_m\Omega^n_X$ and $W_i\Omega^n_X$ are already sheaves of quasi-coherent $\W_m\sO_X$-modules on the Zariski (in fact {\'e}tale) site of $X$. 
Moreover, it follows from \thmref{thm:Perf**} that there are morphisms between the Zariski sheaves of Witt-complexes $\W_m\Omega^\bullet_X \to \tch ^\bullet(\bullet;m)_X[1]$ and
the $p$-typical Witt-complexes $W_i\Omega^\bullet_X \to \tch ^\bullet (p^i)_X[1]$.

\vskip .2cm

We can now state the following result which describes the crystalline cohomology over $W(k)$ defined by Berthelot-Grothendieck \cite{Berthelot} in terms of algebraic cycles. This result can therefore be viewed as the cycle-theoretic avatar of the theorem of Bloch \cite{Bloch crys} which described the crystalline cohomology in terms of algebraic $K$-theory. Recall that the crystalline cohomology was originally defined by Berthelot-Grothendieck as the cohomology of the structure sheaf $\sO_{X, {\rm crys}}$ on the complicated crystalline site of $X$.

\begin{thm}\label{thm:Cryst-coh}
Let $k$ be a perfect field of exponential characteristic $p > 1$. Let $X$ be a smooth quasi-projective scheme over $k$ and $n \ge 0$ an integer. Then there is a canonical isomorphism
\[
{\rm H}^n_{\rm crys}(X/{W}) \xrightarrow{\simeq} {\underset{i \ge 1}\varprojlim} \
\H^{n+1}_{\zar}(X, \tch ^\bullet (p^i)_X).
\]
\end{thm}
\begin{proof}
By \cite[Th{\'e}or{\`e}me~II.1.4]{Illusie}, there is a canonical isomorphism of cohomology groups ${\rm H}^n_{\rm crys}(X/{W}) \xrightarrow{\simeq} {\underset{i \ge 1}\varprojlim} \ \H^{n}_{\zar}(X, W_i\Omega^\bullet_X)$. On the other hand, we have a morphism of the inverse systems of the complexes of Zariski sheaves $W_i\Omega^\bullet_X \to \tch ^\bullet (p^i)_X[1]$. This is an isomorphism by \corref{cor:p-typ*}. The theorem now follows.
\end{proof}

\subsection{Gersten conjecture for additive higher Chow groups}\label{sec:Gersten}
Let $k$ be an arbitrary field. For a presheaf $F$ of abelian groups defined on $\Sm_k$, the Gersten conjecture for $F$ asks whether the Cousin complex of $F$ is locally exact. The exactness of this complex plays an important role in the study and applications of the presheaf $F$. This conjecture is known for many presheaves. For instance, it is known for higher algebraic $K$-theory by Quillen \cite[Theorem 5.11]{Quillen K}, for Milnor $K$-theory by Kerz \cite{Kerz} and, for the de Rham-Witt complex by Gros \cite{Gros}. We have the following answer for the additive higher Chow groups. For $X \in \Sm_k$, let $X^{(i)}$ denote the set of codimension $i$ points on $X$.

\begin{thm}\label{thm:Gersten-exact}
Let $k$ be an arbitrary field and let $X$ be a smooth scheme of finite type over $k$. Then the Cousin complex of the Zariski sheaves of additive higher Chow groups is exact on $X$. That is, there is an exact sequence of Zariski sheaves
\begin{eqnarray*}
0 & \to& \tch ^n (n; m)_X \to \coprod_{x \in X^{(0)}} (i_x)_* {\rm H}_x ^0  (X,\tch ^{n} (n; m)_X )  \to \\
& &  \coprod_{x \in X^{(1)}} (i_x)_* {\rm H}_x ^1  (X,\tch ^{n} (n; m)_X ) \to  \coprod _{ x \in X^{ (2)}} (i_x)_* {\rm H}_x ^2  (X, \tch ^{n} (n; m)_X ) \to \cdots 
\end{eqnarray*}
for every $m \ge 0$ and  $n \ge 1$.
\end{thm}

\begin{proof}
Since all groups are zero for $m =0$ by \cite[Theorem~1.5]{KP-4}, we can assume $m \ge 1$. Since $\tch ^n (n; m)_X$ is a Zariski sheaf on $X$, all maps in the above sequence are well-defined. To show its exactness, we can assume that $X = \Spec(R)$, where $R$ is a smooth local $k$-algebra essentially of finite type. Using Theorem \ref{thm:Main-1}, it suffices to prove the exactness for the Cousin complex of $\W_m\Omega^n_X$. If $p = 1$, then $\W_m\Omega^n_X$ is a direct sum of the sheaves of absolute K{\"a}hler differentials for which the exactness is classically known. If $p > 1$ and $k$ is perfect, the exactness follows directly from \cite[Proposition 5.1.2]{Gros} and the $p$-typical decomposition of $\W_m\Omega^n_X$.

Suppose now that $k$ is imperfect. Using the $p$-typical decomposition, it suffices again to prove the theorem for the $p$-typical de Rham-Witt-complex. We can now use \lemref{lem:Popescu-eft} and write $R = \varinjlim_i R_i$, where $\{R_i\}_{i \in I}$ is a direct system of local $\F_p$-algebras essentially of finite type, such that the transition maps are faithfully flat and the inclusions $R_i \inj R$ are also faithfully flat. In particular, the maps $X \to \Spec(R_i)$ are surjective. As we saw in the proof of \lemref{lem:Popescu-eft}, any prime ideal $\fp \subset R$ is an extension of a prime ideal $\fp_i \subset R_i$ for some $i \in I$. In this case, it is clear that $(R_i \cap \fp) R = \fp$. In other words, there is $i \in I$ and a prime ideal $\fp_i \subset R_i$ such that $\fp_i = \fp \cap R_i$ and $\fp = \fp_i R$.

On the other hand, if $\fp_i \subset R_i$ is a prime ideal of height $j$ for any $i \in I$ and $j \ge 0$ and $\fp$ is a minimal prime of $\fp_i R$, then we must have $\fp_i = \fp \cap R_i$. In this case, the flatness of the inclusion $R_i \inj R$ and \cite[Corollary~14.95]{Gortz} together imply that $\fp$ is a prime ideal of $R$ of height $j$. It follows that $X^{(j)}$ coincides with the inverse limit of $\Spec(R_i)^{(j)}$ where $i$ runs through $I$. Since the cohomology with support commutes with direct limits, it follows that the Cousin complex of $W_m\Omega^n_{R}$ is a direct limit of the Cousin complexes of $W_m\Omega^n_{R_i}$. But the exactness of these latter complexes follows from the case of perfect base field shown above.
\end{proof}

Using Theorem \ref{thm:Main-1} and \propref{prop:inject RK}, we deduce the following piece of the Gersten conjecture for the additive higher Chow groups of regular (not necessarily smooth) semi-local $k$-algebras.

\begin{cor}\label{cor:Inj-ach}
Let $k$ be an arbitrary field and $R$ be a regular semi-local $k$-algebra essentially of finite type. Let $K$ denote the total ring of quotients of $R$. Then the canonical map
\[
\TH^n(R,n;m) \to \TH^n(K,n;m)
\]
is injective for all $m \ge 0$ and $n \ge 1$.
\end{cor}

\subsection{Applications to algebraic $K$-theory}\label{sec:K-thry}
Let $k$ be an arbitrary field with the exponential characteristic $p \geq 1$ and let $R$ be a regular semi-local $k$-algebra essentially of finite type. We can use \thmref{thm:Main-1} and the results of \cite{HessK} to describe the algebraic $K$-theory of the truncated polynomial algebras over $R$ in terms of algebraic cycles as follows. This shows that the algebraic $K$-groups of these truncated polynomial algebras are motivic. Using \thmref{thm:Main-1}, it is shown in \cite{GK-1} and \cite{GK-2} that the Milnor $K$-groups of these truncated polynomial algebras are also motivic.

For $m \ge 1$, we let $R_m = {R[t]}/{(t^m)}$.
Let $K(R_m, (t))$ denote the relative $K$-theory spectrum
for the augmentation ideal $(t) \subset R_m$.

\begin{thm}\label{thm:HeMa-cycles}
Let $m \ge 2$ and $q \ge 1$ be two integers.
\begin{enumerate}
\item If $p = 1$, there is a natural isomorphism
\[
K_q(R_m, (t)) \xrightarrow{\simeq} {\underset{n \ge 0}\bigoplus} \TH^{q +1 - 2n}(R, q +1 - 2n; m-1).
\]
 
\item If $p > 1$, there is a natural long exact sequence
\[
\cdots \to {\underset{i \ge 0}\bigoplus} \TH^{q-2i}(R, q-2i; m (i+1))
\xrightarrow{V_m} 
{\underset{i \ge 0}\bigoplus} \TH^{q-2i}(R, q-2i;i+1) 
\]
\[
\hspace*{8cm} \xrightarrow{\epsilon} K_q(R_m, (t)) \to \cdots .
\]
\end{enumerate}
\end{thm}

\begin{proof}
Combine Theorem \ref{thm:Main-1} with \cite[Theorem 10.1, p.27]{HessK} for (1), and with \cite[Theorem 12.1, p.32]{HessK} for (2). 
\end{proof}

\subsection{Milnor $K$-theory and de Rham-Witt complex}\label{sec:Milnor}
Let $k$ be an arbitrary field and let $R$ be a regular semi-local $k$-algebra essentially of finite type. If $k$ is infinite, it is easy to see that the map $d\log : (R^{\times})^{\otimes n} \to \Omega^{n}_R$, given by $a_1 \otimes \cdots \otimes a_n \mapsto d\log(a_1) \wedge \cdots \wedge d\log(a_n)$, defines a group homomorphism $K^M_n(R) \to \Omega^n_R$. Even if the existence of this map 
at the level of de Rham-Witt complex was a folklore, a written proof was provided
relatively recently in \cite[Appendix~B]{GH}.
When $R$ is a field, this map was probably already known to
Bloch-Kato-Gabber \cite[Corollary~2.8]{BK} even if they did not write a proof.
As the referee pointed out, the dlog map for general $k$-algebras can be deduced
from the case of fields.

Using \thmref{thm:Main-1}, we obtain an alternate and simple proof
of the existence of dlog map from the Milnor $K$-theory to 
the de Rham-Witt complex of
regular semi-local $k$-algebras using algebraic cycles.
Recall from \S~\ref{subsection:Milnor K} that $K^M_*(R)$ denotes the Milnor $K$-theory in the sense of Gabber and Kerz \cite{Kerz10}, and it may differ from the classical Milnor $K$-theory if $k$ is finite.

\begin{cor}\label{cor:Milnor-dlog}
For every $m, n \ge 1$, there is a natural homomorphism
\[
d\log \colon K^M_n(R) \to \W_m \Omega^n_R;
\]
\[
d\log(\{a_1, \ldots , a_n\}) = d\log([a_1]) \wedge \cdots \wedge d\log([a_n]).
\]
\end{cor}

\begin{proof}
We have the natural maps
\[
\CH^n(R,n) \to \CH^n(R,n) \otimes_{\Z} \TH^1(R,1;m) \to \TH^{n+1}(R,n+1;m),
\]
where the first arrow takes an element $\alpha \in \CH^n(R,n)$ to $\alpha \otimes \Gamma_{(1 -t)}$ (see ~\eqref{eqn:Structure-map}). The second arrow is the cap product map of \corref{cor:cap product*}. We denote the composite map by $d\log'$.

By Totaro \cite{Totaro}, the map $\{a_1, \ldots, a_n\} \mapsto V (t_1 - a_1, \ldots , t_n -a_n)$ defines an isomorphism $K^M_n(F) \xrightarrow{\simeq} \CH^n(F, n)$ for every field $F$. By comparing the Gersten resolutions of Milnor $K$-theory from \cite[Proposition~10]{Kerz10} and higher Chow groups from \cite[Corollary, p.300]{Bl1} and using the Totaro's theorem for fields, it follows that Totaro's map is also defined for $R$ and yields an isomorphism $\nu^R_n \colon K^M_n(R) \xrightarrow{\simeq} \CH^n(R,n)$. We remark here that Bloch proved the Gersten resolution when $R$ is smooth over $k$. However, the argument of \thmref{thm:Gersten-exact} easily extends Bloch's theorem to regular algebras.

We therefore get a diagram
\begin{equation}\label{eqn:Milnor-dlog-0}
\xymatrix@C1pc{
(R^{\times})^{\otimes n} \ar@{->>}[r] \ar[dr]_-{d\log} &  K^M_n(R) \ar[rr]^-{\nu^R_n} \ar@{.>}[d] &&  \CH^n(R,n) \ar[d]^-{d\log'} \\
& \W_m\Omega^n_R \ar[rr]^-{\tau^R_{n+1,m}} & &  \TH^{n+1}(R,n+1;m)}
\end{equation}
in which $\nu^R_n$ is an isomorphism. It is clear from the definitions of various maps that the outer trapezium is commutative (see \lemref{lem:Main-2-symbol}). Note that the classical Milnor $K$-theory of $R$ surjects onto our Milnor $K$-theory $K^M_*(R)$ (see \S~\ref{subsection:Milnor K}). This shows that $(R^{\times})^{\otimes n} \to K^M_n(R)$ is surjective. Since $\tau^R_{n+1,m}$ is also an isomorphism by \thmref{thm:Main-1}, it follows that the $d\log$ map factors through the quotient $d\log \colon K^M_n(R) \to \W_m\Omega^n_R$ such that
\begin{equation}\label{eqn:Milnor-dlog-2}
d\log' \circ \nu^R_n = \tau^R_{n+1,m} \circ d\log.
\end{equation}
This finishes the proof.
\end{proof}

Since $\W_m\Omega^n_R$ is a $\W_m(R)$-module, it follows from \corref{cor:Milnor-dlog} that there is a natural map of $\W_m(R)$-modules:
\begin{equation}\label{eqn:Milnor-dlog-1} 
 K^M_n(R) \otimes_{\Z} \W_m(R) \to \W_m\Omega^n_R;
\end{equation}
\[
\{a_1, \ldots, a_n\} \otimes a \mapsto a d\log([a_1]) \wedge \cdots \wedge d\log([a_n]).
\]
We similarly have maps
\[
K^M_{n-1}(R) \otimes_{\Z} \W_m(R) \xrightarrow{d\log \otimes {\rm Id}} \W_m\Omega^{n-1}_R \otimes_{\Z} \W_m(R) \to \W_m\Omega^n_R,
\]
where the last map takes $\omega \otimes a$ to $d(a) \wedge \omega$. We thus get a map
\begin{equation}\label{eqn:Milnor-dlog-3} 
(K^M_n(R) \oplus K^M_{n-1}(R)) \otimes_{\Z} \W_m(R) \to \W_m\Omega^n_R.
\end{equation}
We denote this map by $\gamma^R_{n,m}$.

The following result provides a simple presentation of the de Rham-Witt complex:

\begin{cor}\label{cor:HK}
Assume that either $k$ is infinite or $R$ is local. Let $m, n \ge 1$ be any two integers. Then the map
\[
\gamma^R_{n,m} \colon (K^M_n(R) \oplus K^M_{n-1}(R)) \otimes_{\Z} \W_m(R) \to \W_m\Omega^n_R
\]
is surjective.
\end{cor}

\begin{proof}
This follows directly from \corref{cor:Milnor-dlog} and \cite[Proposition~4.7]{RS} (if $R$ is local) and \cite[Lemma~4.3]{GK-2} (if $k$ is infinite).
\end{proof}

\subsection{Application to motivic cohomology}\label{sec:MC}
Let $k$ be an arbitrary field of exponential characteristic $p > 1$ and let $R$ be a regular semi-local $k$-algebra essentially of finite type. For $m, i \ge 1$, we have the natural maps
\begin{equation}\label{eqn:MC-0}
\CH^n(R,n) \xrightarrow{d\log'} \TH^{n+1}(R,n+1; p^{i-1}) \surj
\TH^{n+1}(R;p^i),
\end{equation}
where the second arrow is the quotient map given by the $p$-typical decomposition of $\TH^{n+1}(R,n+1; p^{i-1})$ (see ~\eqref{eqn:p-typical-decom} and ~\eqref{eqn:p-typical-add-0}). We let $\TH^{n+1}_{\log}(R; p^i)$ denote the image of the composite map in ~\eqref{eqn:MC-0}. Let ${\rm H}^a_{\mathcal{M}} (R, {\Z}/{p^i}(b))$ denote Voevodsky's mod-$p$ motivic cohomology of $R$. Another consequence of \thmref{thm:Main-1} is the following cycle-theoretic description of this motivic cohomology and Milnor $K$-theory mod-$p$:

\begin{cor}\label{cor:MC-main}
There are natural isomorphisms
\[
{\CH^n(R,n)}/{p^i} \xrightarrow{\simeq} {\rm H}^n_{\mathcal{M}} (R, {\Z}/{p^i}(n)) \xrightarrow{\simeq} {K^M_n(R)}/{p^i}  \xrightarrow{\simeq} K_n(R, {\Z}/{p^i}) \xrightarrow{\simeq} \TH^{n+1}_{\log}(R; p^i).
\]
\end{cor}

\begin{proof}
The field case of the two middle isomorphisms are shown in \cite[Theorems~8.1, 8.3]{GL}. These isomorphisms for $R$ follow from the Gersten resolutions of the underlying groups in \cite{Kerz} and \cite{Kerz10}. It suffices therefore to show that ${K^M_n(R)}/{p^i} \simeq \TH^{n+1}_{\rm log}(R; p^i)$.

We now consider the commutative diagram
 
\begin{equation}\label{eqn:MC-main-0}
\xymatrix@C.8pc{
{K^M_n(R)}/{p^i} \ar[rr]^-{\nu^R_n} \ar[d]_-{d\log} & &  {\CH^n(R,n)}/{p^i} \ar[d]^-{d\log'} \\
W_i\Omega^n_R \ar[rr]^-{\tau^R_{n+1, p^i}} & &  \TH^{n+1}(R;p^i).}
\end{equation}

We have seen previously that $\nu^R_n$ is an isomorphism, while $\tau^R_{n+1, p^i}$ is an isomorphism by \corref{cor:p-typ*}. The map $d\log \colon K^M_n(R) \to W_i\Omega^n_R$ of \corref{cor:Milnor-dlog} factors through ${K^M_n(R)}/{p^i}$ and induces an isomorphism of the latter group onto its image (see \cite[Theorems~1.2, 5.1]{Morrow}). We conclude from the definition of $\TH^{n+1}_{\rm log}(R;p^i)$ that $d\log'$ induces an isomorphism ${\CH^n(R,n)}/{p^i} \xrightarrow{\simeq}\TH^{n+1}_{\log}(R; p^i)$. This finishes the proof. 
\end{proof}

\subsection{Trace maps for de Rham-Witt complex}
\label{sec:Trace*}
As we briefly discussed earlier, there is a theory of trace for the $p$-typical de Rham-Witt forms due to Ekedahl \cite{Ekedahl}. Using the $p$-typical decomposition, this allows one to get trace maps for the big de Rham-Witt forms. However, Ekedahl's trace is obtained using a complicated duality theory for these objects. Consequently, it becomes very hard to work with his trace. Furthermore, it takes a lot of work to check that Ekedahl's trace behaves well with various operators.

Using Theorem \ref{thm:Main-1}, we can give a much simpler construction of the trace map on the big de Rham-Witt forms for all finite extensions of regular semi-local $k$-algebras. Our trace is nothing but the easily and explicitly defined push-forward map between the additive higher Chow groups. Furthermore, using this definition, it becomes simple to check all desired properties of the trace map. For instance, we can verify the compatibility of the trace with $V$ and $F$-operators directly because they are simply the pull-back and push-forward on cycles via the map $\phi_r \colon \A^1_k \to \A^1_k$, given by $t \mapsto t^r$.

\begin{thm}\label{thm:trace}
Let $k$ be an arbitrary field. Let $f \colon R \inj S$ be a finite extension of regular semi-local $k$-algebras essentially of finite type. Then there exists a trace map $\Tr_{S/R} \colon \mathbb{W}_m \Omega_{S} ^n \to \mathbb{W}_m \Omega_R ^n$ for every $n \ge 0$ and $m \ge 1$. It is transitive: if $R \subset S \subset S'$ are finite extensions of regular semi-local $k$-algebras essentially of finite type, then we have $\Tr_{S'/R} = \Tr_{S/R} \circ \Tr_{S'/S}$. Moreover, there is a commutative diagram
\begin{equation}\label{eqn:trace-00}
\xymatrix@C.8pc{
\W_m \Omega^{n-1}_{S}  \ar[rr]^-{\tau^{S}_{n,m}} \ar[d]_-{\Tr_{S/R}} & & \TH^n(S, n;m) \ar[d]^-{f_*} \\
\W_m \Omega^{n-1}_R \ar[rr]^-{\tau^R_{n,m}} & &  \TH^n(R, n;m).}
\end{equation}
\end{thm}

\begin{proof}
We consider the diagram
\begin{equation}\label{eqn:trace-01}
\xymatrix@C.8pc{
\W_m \Omega^{n-1}_{S}  \ar[rr]^-{\tau^{S}_{n,m}} \ar@{.>}[d] & & \TH^n(S, n;m) \ar[d]^-{f_*} \\
\W_m \Omega^{n-1}_R \ar[rr]^-{\tau^R_{n,m}} & &  \TH^n(R, n;m).}
\end{equation}

The push-forward map $f_*$ on the level of additive higher Chow groups exists (e.g., see \cite{KL} or \cite{KPv}). The two horizontal arrows in ~\eqref{eqn:trace-01} are isomorphisms by \thmref{thm:Main-1}. It follows that there is a unique map ${\Tr_{S/R}} \colon \mathbb{W}_m \Omega_{S} ^{n-1} \to \mathbb{W}_m \Omega_R ^{n-1}$ such that \eqref{eqn:trace-01} is completed to be a commutative diagram. The transitivity of ${\Tr_{S/R}}$ is clear from the corresponding property of the push-forward map on additive higher Chow groups.
\end{proof}

Over perfect base fields, we can prove the existence of the trace map without assuming that the underlying rings are semi-local.

\begin{thm}\label{thm:trace-gen}
Let $k$ be a perfect field. Let $f \colon R \inj S$ be a finite extension of regular $k$-algebras essentially of finite type. Then there exists a trace map $\Tr_{S/R} \colon \mathbb{W}_m \Omega_{S} ^n \to \mathbb{W}_m \Omega_R ^n$ for every $n \ge 0$ and $m \ge 1$ satisfying the properties described in \thmref{thm:trace}.
\end{thm}

\begin{proof}
Let $X = \Spec(R)$, $Y = \Spec(S)$ and $f \colon Y \to X$ is the corresponding finite map. Recall from \S~\ref{sec:Crys} that $\tch ^n (n; m)_X$ is the Zariski sheaf on $X$ associated to the presheaf $U \mapsto \TH^n(U, n;m)$ for $n, m \ge 1$. We let $\wh{\tch} ^n (n; m)_X$ denote the Zariski sheaf on $X$ associated to the presheaf $U \mapsto  \TH^n(f^{-1}(U), n;m)$.

Since the push-forward map $f_*$ on the additive higher Chow groups commutes with flat pull-back, it induces a natural map $f_* \colon \wh{\tch} ^n (n; m)_X \to \tch ^n (n; m)_X$ of Zariski sheaves. It follows from \thmref{thm:Perf**} that for $n \ge 1$, there is a diagram of sheaves on $X_{\rm Zar}$:
\begin{equation}\label{eqn:trace-gen-0}
\xymatrix@C.8pc{
f_*(\W_m \Omega^{n-1}_{Y})  \ar[rr]^-{\tau^{Y}_{n,m}} \ar@{.>}[d] & & \wh{\tch} ^n (n; m)_X \ar[d]^-{f_*} \\
\W_m \Omega^{n-1}_X \ar[rr]^-{\tau^X_{n,m}} & &  \tch ^n (n; m)_X.}
\end{equation}

The two horizontal arrows in ~\eqref{eqn:trace-gen-0} are isomorphisms of Zariski sheaves on $X$ by \thmref{thm:Main-1}. It follows that there is a unique map of Zariski sheaves $\Tr_{Y/X} \colon f_*(\W_m \Omega^{n-1}_{Y}) \to \W_m \Omega^{n-1}_X$ on $X$ satisfying the properties described in \thmref{thm:trace}. Since the de Rham-Witt forms are quasi-coherent Zariski sheaves of $\W_m\sO_X$-modules, the map $\Tr_{Y/X}$ induces the desired trace map
\[
\Tr_{S/R} \colon \mathbb{W}_m \Omega_{S} ^n \xrightarrow{\simeq}  {\rm H}^0(X, f_*(\W_m \Omega^{n-1}_{Y})) \to {\rm H}^0(X, \W_m \Omega^{n-1}_{X}) \xleftarrow{\simeq} \mathbb{W}_m \Omega_R ^n
\]
for all $m,n \ge 1$.
\end{proof}

\begin{remk}\label{remk:Witt-vector-trace}
  From the construction of ${\Tr_{S/R}}$ in Theorems~\ref{thm:trace} and ~\ref{thm:trace-gen} and from ~\eqref{eqn:Structure-map}, one checks that $\Tr_{S/R}$ coincides with the trace map for the ring of Witt-vectors in \cite[Appendix]{R} when we take $n = 0$. It also follows from \cite[Lemma~3.22]{R} that ${\Tr_{S/R}}$ of Theorem~\ref{thm:trace} coincides with the one in  \thmref{thm:trace**} in \S \ref{sec:Appendix} Appendix, when $R$ and $S$ are fields. Using the compatibility of Ekedahl's trace with various operators, one knows that his trace coincides with that of \thmref{thm:trace**} for fields. Using \propref{prop:inject RK} and \corref{cor:Inj-ach}, one can then check that trace maps of Theorems~\ref{thm:trace} and ~\ref{thm:trace-gen} coincide with that of Ekedahl for regular semi-local rings.
\end{remk}

\subsection{Galois descent for additive higher Chow groups} \label{sec:Galois}
Let $k$ be an arbitrary field and let $R$ be a regular semi-local $k$-algebra essentially of finite type. Let $G$ be any finite group. Assume that $G$ acts freely on $R$ as $k$-algebra automorphisms so that the inclusion $R^G \inj R$ is finite and {\'e}tale, where $R^G$ denotes the ring of invariants. In particular, $R^G$ is a regular semi-local $k$-algebra essentially of finite type. Using \thmref{thm:Main-1}, we can prove the following Galois descent for the Milnor range additive higher Chow groups.

\begin{thm}\label{thm:G-descent}
The inclusion $R^G \inj R$ induces an isomorphism
\[
  \TH^n(R^G,n;m) \xrightarrow{\simeq} (\TH^n(R,n;m))^G
\]
for all $m \ge 1$ and $n \ge 0$.
\end{thm}

\begin{proof}
The group $G$ acts on $\TH^n(R,n;m)$ and $\W_m\Omega^n_R$. 
Since $G$  is finite, it immediately 
follows from the functoriality of $\tau^R_{n,m}$ 
(see \propref{prop:CW-Pull-back}) that it is $G$-equivariant.

We now let $X = \Spec(R^G)$ and $X' = \Spec(R)$. We let $f \colon X' \to X$ denote the quotient map.
We then have a Cartesian square
\begin{equation}\label{eqn:desc-0}
    \xymatrix@C.8pc{
      G \times X' \ar[r]^-{\mu} \ar[d]_-{p_{X'}} & X' \ar[d]^-{f} \\
      X' \ar[r]^-{f} & X,}
\end{equation}
where $p_{X'}$ is the projection and $\mu$ is the $G$-action map. Since $f$ is {\'e}tale and $\W_m\Omega^{n-1}_X$ is an {\'e}tale sheaf, it follows using the equalizer definition of a sheaf that the canonical map $f^* \colon \W_m\Omega^{n-1}_{R^G} \to (\W_m\Omega^{n-1}_R)^G$ is an isomorphism (between $\W_m(R^G)$-modules).

We now prove the theorem. We let $n \ge 1$ and consider the diagram 
\begin{equation}\label{eqn:G-descent-0}
\xymatrix@C.8pc{
\W_m \Omega^{n-1}_{R^G}  \ar[rr]^-{\tau^{R^G}_{n,m}} \ar[d]_-{f^*} & &  \TH^n(R^G, n;m) \ar[d]^-{f^*} \\
(\W_m \Omega^{n-1}_R)^G \ar[rr]^-{\tau^R_{n,m}} & &  \TH^n(R, n;m)^G.}
\end{equation}

This diagram is commutative by \thmref{thm:Main-1}. The lower horizontal arrow is induced by the $G$-equivariance of ${\tau^R_{n,m}}$. The two horizontal arrows are isomorphisms by \thmref{thm:Main-1}.
The left vertical arrow is an isomorphism as shown above. It follows that the right vertical arrow is also an isomorphism.
\end{proof}

\bigskip

\noindent\emph{Acknowledgments.}  We thank Spencer Bloch for his lessons that guided us during the last decade on this subject, and for his patience and encouragement.
We thank James Borger, H\'el\`ene Esnault and Kay R\"ulling for helpful conversations during this work. We thank Kay R{\"u}lling especially for contributing the Appendix and for bringing Ekedahl's work to our attention. We thank the editor and the referees for their comments that helped in improving the article.

We acknowledge that part of this work was done during JP's visits to TIFR and AK's visits to KAIST. We thank both the institutions. JP would like to thank Damy for his peace of mind at home throughout the hard time this paper was written. JP was partially supported by the National Research Foundation of Korea (NRF) grant (2015R1A2A2A01004120 and 2018R1A2B6002287) funded by the Korean government (MSIP), and TJ Park Junior Faculty Fellowship funded by POSCO TJ Park Foundation.

\vskip .4cm

\setcounter{tocdepth}{1}
\section{Appendix}\label{sec:Appendix}
\begin{center}
{\bf {THE DE RHAM-WITT COMPLEX AND ADDITIVE CHOW}}
\end{center}
\begin{center}
{\bf{GROUPS OVER A FIELD: THE CHARACTERISTIC $2$ CASE}} 
\end{center}
\begin{center}
{Kay R\"ulling} \footnote{The author is supported by the DFG Heisenberg Grant 
RU 1412/2-2.}
\end{center}

The main theorems of \cite{R} (see also \cite{R-1})  were stated only for fields of characteristic $\neq 2$.
This originates in the use of \cite[Theorem~4.2.8]{HeMa-2} which was only for odd primes.
In this appendix, it is explained that thanks to \cite{Costeanu}, the results of \cite{R} extend directly to 
the characteristic $2$ case.
I thank  Amalendu Krishna and Jinhyun Park for the opportunity to detail this extension here.


\subsection{}
Let  $(\W_S\Omega^\bullet_A)_S$ denote the big de Rham Witt complex from 
\cite{Hesselholt2},
where $A$ is a ring and $S$ is running through all truncation sets.
It comes with the maps $d$, $F_n$, $V_n$, restriction, and multiplication.
For $m\ge 1$ we set 
$\W_m\Omega^\bullet_A:= \W_{\{1,\ldots, m\}}\Omega^\bullet_A$ and for a fixed prime $p$
we denote the $p$-typical de Rham-Witt complex by
$W_n\Omega^\bullet_A:=\W_{\{1, p, \ldots, p^{n-1}\}}\Omega^\bullet_A$.
Note that in general $d\circ d$ is not zero in $\W_S\Omega^\bullet_A$.
Denote by $\W_S\Omega^\bullet_{A/\Z}$ the Witt complex
from \cite[Remark~4.8]{Hesselholt2} (with  $k=\W(\Z)$); it is the initial object in the category of Witt
complexes with $\W(\Z)$-linear differential.
Note that $\W_S\Omega^\bullet_{A/\Z}$ is always a DGA, in particular we have $d\circ d=0$.
Furthermore,  if $A$ contains a field, then $\W_S\Omega^\bullet_{A/\Z}= \W_S\Omega^\bullet_A$,
see \cite[Remark 4.2, c)]{Hesselholt2}.
If $A$ is an $\F_p$-algebra, then the $p$-typical de Rham-Witt complex is the one from Bloch-Deligne-Illusie.
Also in case $A$ is an $\F_p$ (or a $\Q$)-algebra, the decomposition
\begin{equation}\label{eqn:para:DRW1}
{\W_S\Omega^\bullet_A= \prod_{(j,p)=1} \W_{\sP\cap S/j} \Omega^\bullet_A,}
\end{equation}
from \cite[Theorem~1.11]{R} is valid in the case $p=2$.
 (Indeed, in this case the construction of the  $V$-complex in 
\cite[Proposition 1.2]{R}
goes through and the same proof as in \cite[Theorem 1.11]{R}  shows that it is the initial object in the category of Witt
complexes as in \cite{Hesselholt2} and that it decomposes as in 
\eqref{eqn:para:DRW1}.)

\begin{thm}[{\cite[Theorem 3.20]{R} for char$(k)\neq 2$}]\label{thm:WC**}
Let $k$ be a field. There is a canonical isomorphism 
\[\W_m\Omega^{n-1}_k \xrightarrow{\simeq} \TH^{n}(k,n;m), \quad m, n\ge 1,\]
where the right hand side is the additive Chow groups of Bloch-Esnault.
Furthermore, via this isomorphism, the maps $d$, $F_n$, $V_n$, restriction and multiplication on the de Rham-Witt side
correspond to $\sD$, $\sF_n$, $\sV_n$, restriction, and $*$, on the Chow side, as defined in \cite[Definition-Proposition 3.9]{R}.
\end{thm}

Thanks to \cite{Costeanu}, which was not at disposal 
when \cite{R} was written, the proof of {\em loc. cit.} goes through, also for $p=2$.
We will explain this in more detail in the following.
Note that $\TH^{n}(k,n;m)$ is written as $\CH^{n+1}(\A^1_k|(m+1) \cdot \{0\},n)$
in \cite{R}.

\subsection{}\label{sec:N-2}
We fix a prime $p$. Let $A$ be a $\Z_{(p)}$-algebra and denote by 
$A[x]$ the  polynomial ring in the variable $x$. 
Then  the group $W_n\Omega^q_{A[x]/\Z}$ (resp. $W_n\Omega^q_{A[x, 1/x]/\Z}$) is freely generated by elements
of the following type:
\begin{align*}
a[x]^j, \quad  &  a\in W_n\Omega^q_{A/\Z}, \, j\ge 0 \,(\text{resp. } j\in \Z),\\
b[x]^{j-1}d[x],\quad &  b\in W_n\Omega^{q-1}_{A/\Z}, \, j \ge 1 \, (\text{resp. }j\in \Z),\\
V^s(a[x]^j),\quad & a\in W_{n-s}\Omega^q_{A/\Z}, \,s\in \{1,\ldots, n-1\},\,  j\ge 1  \text{ with } (j,p)=1\\
                           &                                                    (\text{resp. } j\in \Z\setminus p\Z),\\
dV^s(b[x^j]), \quad & b\in W_{n-s}\Omega^{q-1}_{A/\Z},\, s\in \{1,\ldots, n-1\},\, j\ge 1 \text{ with } (j,p)=1\\
                     &                                                                           (\text{resp. } j\in \Z\setminus p\Z).
\end{align*}

For $A[x]$ and $p$ odd, this is \cite[Theorem 4.2.8]{HeMa-2} and for $p=2$, this follows from
\cite[Theorem 4.3]{Costeanu} (one has to observe that the functor $P$ constructed in these references
sends a $\W(\Z)$-linear $p$-typical Witt complex over $A$ to a $\W(\Z)$-linear $p$-typical Witt complex over $A[x]$ and 
that it preserves surjections). The result for $A[x, 1/x]$ is deduced from this as in \cite[Theorem 2.1]{R},
where (at least in the case $p=2$) 
the reference to \cite[Proposition 1.18]{R} should be replaced by 
\cite[Theorem C]{Hesselholt2}.

\begin{thm}[{\cite[Theorem 2.6]{R} for char$(k)\neq 2$}]\label{thm:trace**}
Let $L/k$ be a finite field extension. Then there exists a trace map
\[\Tr_{L/k} : \W_S\Omega^\bullet_L\to \W_S\Omega^\bullet_k,\]
which satisfies  the properties (i) - (v) from \cite[Theorem 2.6]{R}.
Furthermore \cite[Proposition 2.7]{R} holds.
\end{thm}
\begin{proof}
The proof for $p=2$ is the same as in {\em loc. cit.}, once we made the following remarks:
\cite[Lemma 1.20]{R} also holds for any $\F_2$-algebra $A$ with the same proof;
\cite[Lemma 2.3]{R} also holds for $p=2$, this follows from 
\cite[I, Proposition 3.2, 3.4]{Illusie} and a limit argument;
\cite[Proposition 2.4]{R} holds with the same proof also for $p=2$ once the reference to \cite[Theorem 2.1]{R}
is replaced by the first paragraph of \S~\ref{sec:N-2} above. 
\end{proof}

\subsection{}\label{sec:N-3}
Let $p$ be a prime and $A$ a $\Z_{(p)}$-algebra.
For a finite truncation set $S$, we define
\[{\rm Fil}_{S,j}:= {\rm Ker} (\W_S\Omega^\bullet_{A[[t]]/\Z}\to \W_S\Omega^\bullet_{A[[t]]/(t^j)/\Z}), \quad j\ge 1,\]
and 
\[\W_S\hat{\Omega}^\bullet_{A((t))/\Z}=\varprojlim_{j} \W_S\hat{\Omega}^\bullet_{A((t))/\Z}/ {\rm Fil}_{S,j}.\]
Then any element in $\W_S\hat{\Omega}^\bullet_{A((t))/\Z}$ can be uniquely written as in 
\cite[(2.9.1)]{R}, and we can define
\[\wh{\rm Res}^q_{t,n}: W_n\wh{\Omega}^q_{A((t))/\Z}\to W_n\Omega^{q-1}_{A/\Z}\]
as in \cite[(2.9.2)]{R}. (Using \S~\ref{sec:N-2}), 
the proof is similar as in \cite[Lemma 2.9]{R}.)

We define ${\rm Res}^q_{t,n}$ as the composition
\[W_n\Omega^q_{A((t))/\Z}\xrightarrow{\rm can.} 
W_n\wh{\Omega}^q_{A((t))/\Z}\xrightarrow{\wh{\rm Res}^q_{t,n}} 
W_n\Omega^{q-1}_{A/\Z}.\]
If $A$ contains a field
and $S$ is a finite truncation set, 
then we define 
\begin{equation}\label{eqn:N-4}
{\rm Res}^q_{t,S}: \W_S\Omega^q_{A((t))}\to \W_S\Omega^{q-1}_A
\end{equation}
as in \cite[Definition 2.11]{R}, using that in this case 
we have $\W_S\Omega^\bullet_{A}=\W_S\Omega^\bullet_{A/\Z}$ and that 
the decomposition \eqref{eqn:para:DRW1} also extends to $\W_S\hat{\Omega}$.

For any $\Z_{(p)}$-algebra $A$, the map ${\rm Res}_{t,n}^q$ satisfies the properties (i) - (viii) of \cite[Proposition 2.12]{R}
and \cite[Lemma 2.14]{R} holds.
If $A$ contains a field, the same holds for ${\rm Res}^q_{t, S}$, $S$ any finite truncation set.
(The case ${\rm Res}_{t,n}^q$ for $p=2$ is proven as in {\em loc. cit.}, the case ${\rm Res}_{t,S}^q$ follows from this.
Note however, that even in the case $A$ contains a field, the proofs use reduction to torsion-free case.
Since  the absolute de Rham-Witt complex of a torsion-free $\Z_{(2)}$-algebra is not a DGA,
we prefer to work with the $\W(\Z)$-linear complex.)

\begin{remk}
For an $\F_p$-algebra $A$, the residue ${\rm Res}_t: W_n\Omega^q_{A((t))}\to W_n\Omega^{q-1}_A$ was also constructed in \cite[\S 2, Proposition 3]{Kato} using algebraic $K$-theory and Bloch's approach to the de Rham-Witt complex.
\end{remk}

\begin{thm}[{\cite[Theorem 2]{R-1} for char$(k)\neq 2$}]\label{thm:res}
Let $C$ be a connected regular projective curve over a field $k$ with function field $K=k(C)$.
Let $S$ be a finite truncation set. Then 
\[\sum_{P\in C} {\rm Res}_P(\omega)=0, \quad \text{for all }\omega\in \W_S\Omega^q_K, \, q\ge 1,\]
where the sum is over all closed points in $C$ and ${\rm Res}_P: \W_S\Omega^q_K\to \W_S\Omega^{q-1}_k$
is defined as in \cite[Definition-Proposition 1]{R-1} (using ${\rm Res}^q_{t,S}$ from ~\eqref{eqn:N-4} and $\Tr$ from \thmref{thm:trace**}.)
\end{thm}
\begin{proof}
The  proofs of \cite[Theorem 2]{R-1}, \cite[Theorem 2.19]{R} work also for $p=2$ once we made the following 
remarks: the proof of the well-definedness of ${\rm Res}_P$ is the same as in 
\cite[Definition-Proposition 2.15]{R}
since \cite[Lemma 1.16]{R} holds in general.
\cite[Proposition 2.18]{R} 
holds with the trace from Theorem~\ref{thm:trace**}.
At the end of the proof of \cite[Theorem 2.19]{R} (on page 140/141), an 
element is lifted to
$\W_S\Omega^1_{A((t))}$, with $A=\Z_{(p)}[z_a, z_b,z_c]$.
Replace this by the following argument (at least if $p=2$): first observe that 
the desired vanishing can be reduced
to the $p$-typical case by the definition of ${\rm Res}_P$. Then lift the element $\omega_{2,P}$
to the $W(\Z)$-linear complex $W_n\Omega^1_{A((t))/\Z}$ and proceed as in the proof using the 
${\rm Res}_t$ from ~\eqref{eqn:N-4}.
\end{proof}

\begin{proof}[{Proof of Theorem \ref{thm:WC**}}.]
The proof of \cite[Theorem 3.20]{R} (see also \cite{R-1}) works also for $p=2$ once we made the following remarks:
in \cite[Lemma 3.5]{R}, replace $\W_m\Omega^r_{A}$ by $\W_m\Omega^r_{A/\Z}$ (at least if $p=2$);
at the end of the proof of  \cite[Theorem 3.6]{R} (on page 148), refer to Theorem \ref{thm:res} instead of 
\cite[Theorem 2.19]{R}. In \cite[Lemma 3.15]{R}, observe that if char$(k)=2$,
then $\sD\sD(\alpha)=0$, since in $K^M_2(k)$, we have $\{a,a\}=\{a,-1\}=0$, and similarly
also $\sF_r\sD\sV_r=\sD$; this implies that \cite[Proposition 3.17]{R} also holds if char$(k)=2$.
For the rest of the proof of \cite[Theorem 3.20]{R}, use \thmref{thm:trace**} 
for properties of the trace and \S~\ref{sec:N-2}
instead of \cite[Theorem 2.1]{R}.
\end{proof}

\end{document}